\documentclass[10pt]{article}

\usepackage[latin1]{inputenc}
\usepackage{subcaption}
\usepackage[pagebackref=true, colorlinks=true, linkcolor=blue, citecolor=blue, linkcolor=blue, anchorcolor=red, urlcolor=blue]{hyperref}
\usepackage{url, amsmath,enumerate,fancyhdr,amssymb, amsthm, 
pstricks, epsf, lscape, ifpdf, graphicx, float,tikz}

\usetikzlibrary{decorations.pathreplacing,arrows.meta}
\usetikzlibrary{patterns,decorations.pathreplacing}

\usepackage{placeins}

\usepackage{xcolor}

\usepackage{algorithm}
\usepackage{algpseudocode}
\usepackage{float}        
\usepackage{hyperref}

\makeatletter
\newcommand{\leqnomode}{\tagsleft@true}
\newcommand{\reqnomode}{\tagsleft@false}
\makeatother

\newcounter{substep}

\usepackage{algorithm}
\usepackage{algpseudocode}
\usepackage{float}        
\usepackage{hyperref}

\usepackage{algpseudocode} 
\usepackage{enumitem}
\usepackage{color}
\usepackage{mathtools}
\usepackage{algpseudocode} 
\algnewcommand\algorithmicinput{\textbf{Input:}}
\algnewcommand\INPUT{\item[\algorithmicinput]}
\algnewcommand\algorithmicinitialization{\textbf{Initialization:}}
\algnewcommand\INITIALIZATION{\item[\algorithmicinitialization]}

\usepackage{xcolor}
\usepackage{tcolorbox}

\newcommand{\pha}{\phantom} 
\newcommand{\ti}{\times}

\newcommand{\mymatrix}[1]{{\cal R}^{#1}}

\newcommand{\spread}{\operatorname{spread}}

\newcommand{\myvec}{\operatorname{vec}}
\newcommand{\svec}{\operatorname{svec}}
\newcommand{\Diag}{\operatorname{Diag}}

	\newcommand\bovermat[2]{%
		\makebox[0pt][l]{$\smash{\overbrace{\phantom{%
						\begin{matrix}#2\end{matrix}}}^{\text{#1}}}$}#2}
	
	\usepackage{mathtools}
	\makeatletter
	\renewcommand*\env@matrix[1][*\c@MaxMatrixCols c]{%
		\hskip -\arraycolsep
		\let\@ifnextchar\new@ifnextchar
		\array{#1}}
	\makeatother
	
	\usepackage{tikz,array,calc, xcolor}

				\newcommand{\circledentry}[1]{%
					\tikz[baseline=(char.base)]{
						\node[
						draw,
						circle,
						minimum size=4.5mm,
						inner sep=0pt
						] (char) {$\displaystyle #1$};
					}
				}

\usepackage[margin=3.2cm]{geometry}

\mathchardef\mhyphen="2D

\newtheorem{Definition}{Definition}

\newtheorem{Example}{Example}
\newtheorem{Proposition}{Proposition}
\newtheorem{Lemma}{Lemma}
\newtheorem{Theorem}{Theorem}
\newtheorem{Corollary}{Corollary}
\newtheorem{Remark}{Remark}
\newtheorem{Assumption}{Assumption}

\newcommand{\ba}{\begin{array}}  
	\newcommand{\ena}{\end{array}}

\newcommand{\A}{\mathcal A}

\newcommand{\lin}{\operatorname{lin}}

\newcommand{\bpx}{\begin{pmatrix}}
\newcommand{\epx}{\end{pmatrix}}

\newcommand{\bbx}{\begin{bmatrix}}
\newcommand{\ebx}{\end{bmatrix}}

\newcommand{\bdef}{\begin{Definition}} 
\newcommand{\commentout}[1]{}
\newcommand{\co}[1]{}

\newcommand{\coab}[1]{}

\newcommand{\nin}{\noindent}

\newcommand{\pf}[1]{\vspace{.35cm} \nin {\bf Proof {#1} }}

\newcommand{\norm}[1]{\parallel \! #1 \! \parallel}
\newcommand{\sym}[1]{{\cal S}^{#1}}
\newcommand{\psd}[1]{{\cal S}_+^{#1}}

\newcommand{\rad}[1]{\mathbb{R}^{#1}}

\newcommand{\symn}{\sym{n}}
\newcommand{\psdn}{\psd{n}}
\newcommand{\radn}{\rad{n}}

\newcommand{\la}{\langle}
\newcommand{\ra}{\rangle}

\newcommand{\bX}{\bar{X}}
\newcommand{\bZ}{\bar{Z}}
\newcommand{\DX}{\Delta {X}}
\newcommand{\DZ}{\Delta {Z}}

\newcommand{\rank}{\operatorname{rank}}

\newcommand{\beq}{\begin{equation}}
\newcommand{\eeq}{\end{equation}}
\newcommand{\beqa}{\begin{eqnarray}}
\newcommand{\eeqa}{\end{eqnarray}}
\newcommand{\bac}{\begin{array}{ccccccccccc}}
\newcommand{\eac}{\end{array}}
\newcommand{\bprop}{\begin{Proposition}}
\newcommand{\eprop}{\end{Proposition}}

\newcommand{\beqast}{\begin{eqnarray*}}
\newcommand{\eeqast}{\end{eqnarray*}}
\newcommand{\benum}{\begin{enumerate}}
\newcommand{\eenum}{\end{enumerate}}
\newcommand{\bit}{\begin{itemize}}
\newcommand{\eit}{\end{itemize}}
\newcommand{\bth}{\begin{Theorem}}
\newcommand{\enth}{\end{Theorem}}
\newcommand{\ble}{\begin{Lemma}}
\newcommand{\ele}{\end{Lemma}}
\newcommand{\bex}{\begin{Example}}
\newcommand{\eex}{\end{Example}}
\newcommand{\bcor}{\begin{Corollary}}
\newcommand{\ecor}{\end{Corollary}}
\newcommand{\brem}{\begin{Remark}}
\newcommand{\erem}{\end{Remark}}
\newcommand{\bass}{\begin{Assumption}}
\newcommand{\eass}{\end{Assumption}}

\newcommand{\bA}{\bar{\A}}
\newcommand{\bb}{\bar{b}}

\newcommand{\bsmx}{\begin{small} \begin{pmatrix}}
\newcommand{\esmx}{\end{pmatrix} \end{small}}

\title{\Large Strict complementarity in semidefinite programming, singularity degree,
	 and the (dis)connection of forward and backward errors}  
\author{G\'{a}bor  Pataki \thanks{Department of Statistics and Operations Research, University of North Carolina at Chapel Hill} \hspace{1cm} 
}

\begin{document}

\maketitle 

\begin{abstract}

Strict complementarity of a primal-dual pair of optimal solutions
is fundamental in the numerical analysis of semidefinite programs (SDPs). 
Strict complementarity  drives the convergence behavior of interior point algorithms.
When it fails,  some pathological examples show a striking gap between two error measures of approximate solutions. The first of these is  the forward or  "true" error, i.e.,  the distance to the optimal solution set. The second is the less informative  (but still widely used) backward error, measured by the  constraint violation.

We first characterize  the lack of strict complementarity in SDPs via a simple normal form. Our normal form has three key features: (i) it is obtained using  elementary row operations and rotations; (ii) it makes the lack of strict complementarity evident; and (iii) it lets us construct any such SDP by a simple algorithm. A variant of our generating algorithm allows us to construct any SDP that fails strict complementarity but satisfies  Slater's condition on both the primal and dual sides.
 Thus,  we {\em parametrize} the data of all SDPs that lack strict complementarity in a manner similar to how  the Jordan normal form  parametrizes square matrices with given eigenvalue structure.

We next study the singularity degree of SDPs, the minimum number of facial reduction steps needed to induce  Slater's condition. We precisely characterize when the singularity degree of an SDP equals the number of constraints -- a result that underlies our generating algorithms and that we believe is  of independent interest.

We construct and publicly share a set of SDPs that lack strict complementarity and present a detailed computational study. We find  that forward--backward error gaps are quite common: 
in many small SDPs (with matrix order $\leq  20$), the forward ("true") error exceeds the backward error by up to seven  orders of magnitude.
Further, in several data sets the forward and backward errors are {\em inversely} correlated. In other words, the worse the "true" forward error is, the harder it is to detect.

	\end{abstract}
	
	
	\vspace{-.5cm}

	\section{Introduction} 
	
Semidefinite programs (SDPs) are convex optimization problems in which the decision variable is a positive semidefinite matrix,
and  the objective function and constraints are linear. 
SDPs substantially generalize linear programming and they 
have found uses in many areas 
of optimization, such as control theory,  engineering, and combinatorial optimization. Besides 
having great modeling power, they can be solved efficiently by modern 
interior-point and first-order optimization algorithms  \cite{NestNemirov:94, nesterov2018lectures}.

	This work  focuses on a property of SDPs that sharply sets them apart from linear programming. We consider a primal-dual pair of  SDPs 
\begin{center}
	\begin{minipage}{0.5\linewidth}
		\leqnomode
		\begin{equation}\label{p}
			\begin{split}
				\inf  & \,\, \la C,  X  \ra  \\
				s.t. & \,\, \la A_i, X \ra \, = \, b_i \, (i=1, \dots, m) \\
				& \,\, X \succeq 0
			\end{split}\tag{P}
		\end{equation}
	\end{minipage}%
	\begin{minipage}{0.5\linewidth}
		\reqnomode
		\begin{equation}\label{d}
			\begin{split}
				\sup & \,\, \la b, y \ra  \\
				s.t.   & \,\,  \sum_{i=1}^m  y_i A_i \preceq C \\
			\end{split}\tag{D} 
		\end{equation}
	\end{minipage}
\end{center}
where  the $A_i$ and $C$ are $n \times n$ 
symmetric matrices and  the $b_i$ are scalars. 
For  symmetric matrices $S$ and $T$  we write 
$S \preceq T$ to say that $T - S$ is positive semidefinite (psd) 
and we write $\la T, S \ra :=  {\rm trace}(TS)$ to denote their inner product. 
We also write $\la b, y \ra$ for the inner product of $b := (b_1, \dots, b_m)^\top$ and $y.$ 
For simplicity, we call Z feasible (optimal) in \eqref{d} if there is $y \in \rad{m}$ for which $(y,Z)$ is feasible (optimal).

Suppose $X$ is a feasible solution in \eqref{p} and $Z$ in \eqref{d}.
Since 
$$ 
\la C, X \ra - \la b, y \ra = \la X, Z \ra \geq 0,
$$ 
it follows that 
$X$ and $Z$ are both optimal iff $\la X, Z \ra = 0.$ Then 
the range spaces of $X$ and $Z$ are orthogonal, hence 
\begin{equation} \label{eqn-rankX-rankZ} 
	\rank X + \rank Z \leq n.
\end{equation}
We say that a pair of optimal solutions $(X,Z)$ \emph{satisfies strict complementarity} if 
in \eqref{eqn-rankX-rankZ} equality holds; and that it \emph{fails strict complementarity} 
if strict inequality holds. 
We also say that the pair \eqref{p}$\mhyphen$\eqref{d} satisfies strict complementarity, if
a pair of maximum rank optimal solutions do; and that they fail strict complementarity, if any pair of maximum rank optimal solutions does.

 Strict complementarity plays several fundamental roles in semidefinite
 programming. The first of its roles we highlight is  in the analysis of Newton's method and 
 interior-point methods (ipms).  
 The influential work \cite{alizadeh1998primal} showed that, under strict complementarity and nondegeneracy assumptions,  the Newton system associated with their scaling   is well defined and has a nonsingular Jacobian at the solution.  
   Strict complementarity has also become a standard assumption in proving superlinear convergence of ipms.
    The first works that  proved  superlinear convergence of ipms under this assumption were 
    \cite{kojima1998local,  potra1998superlinearly, luo1998superlinear, kojima1999predictor, lu2005note}. Only very recently did 
    \cite{mohammadisiahroudi2025quantum}
  prove quadratic convergence of an iterative refinement method without
  any regularity assumptions on the underlying SDP.

  Strict complementarity also impacts  the connection of two fundamental error measures in SDPs.
  We first explain these errors for a {\em feasibility} problem.
  The first is the "true error", i.e., the forward error in numerical analysis parlance, 
     defined as the distance to the set of feasible solutions. The second 
  is the backward error, which is just  the constraint violation. 
   
  While the forward error is considerably more useful, it is usually impossible to compute. Indeed, 
  computing it means finding an {\em exact} solution which is  closest to a given  {\em approximate} solution; but we computed that  approximate 
  solution, since we could not find an exact one. 
 Thus, instead of the forward error, we often use the backward error as a proxy, and seek bounds that relate the two. 

  As the classic  paper \cite{Hoffman1952ApproximateSolutions} showed, the two errors are of the same order in linear programs.
  A highly influential work \cite{lewis1998error} showed  the same for convex inequality systems which satisfy Slater's condition and another technical condition called well-posedness -- see their Corollary 1. This result directly applies to semidefinite systems.
  
 However, the two errors can be drastically different in SDPs when Slater's condition fails.
   As to how different, in seminal work, Sturm~\cite{Sturm:00} proved the error bound
  \begin{equation}        \label{eqn-sturm-bound} 
  	\text{forward error}
  	= O \bigl( \text{(backward error)}^{1/2^d} \bigr),
  \end{equation} 
  where $d \in \{0,\dots,n-1\}$ is the {\em singularity degree} of the SDP.
  The singularity degree is the minimum number of facial reduction  \cite{BorWolk:81,Pataki:13,drusvyatskiy2017many} 
  steps needed  to reach the minimal face containing the feasible region
  \cite{BorWolk:81,Pataki:13,drusvyatskiy2017many}. Thus, a feasible SDP 
  satisfies Slater's condition exactly when its singularity degree is zero.
    Sturm also showed that the \eqref{eqn-sturm-bound} bound  is best possible: see our discussion in Example \ref{example-sturm}.
    
    The singularity degree has since become 
    an important measure of irregularity of SDPs.
  It appears, for example, in the analysis of error bounds and slow
  convergence~\cite{sremac2021error} and  in convergence results for
  alternating projections~\cite{drusvyatskiy2017note}.

  This theory of error bounds also applies to  the {\em optimal set} of an SDP.
  To see how, we let $Z^*$ be an optimal solution of \eqref{d}.
   We may assume that $Z^*$ has a positive definite block of order, say $s$ in the lower right corner, and all
  its other elements are zero.
  We then represent the optimal set of \eqref{p} by attaching the equation 
  \begin{equation} \label{eqn-Z*X=0} 
  	\la Z^*, X \ra = 0
  \end{equation}
  to its  feasible set.
   Then every feasible
  matrix $X$ of the enlarged system has its last $s$ rows and columns equal to
  zero. Deleting these rows and columns from the $A_i$ and dropping
  \eqref{eqn-Z*X=0} yields a reduced system. In turn, this reduced system satisfies
  Slater's condition, equivalently, it  has singularity degree zero, exactly when
  there exists an optimal solution of \eqref{p} that is strictly complementary
  with $Z^*$.

 The following illuminating example is adapted from 
 Example 2 in \cite{Sturm:00}. It will be our running example.
 
\begin{Example} \label{example-sturm} 
	Consider the SDP
	\begin{equation} \label{problem-sturm} 
		\begin{array}{rrcl} 
			\inf & x_{11} \\
			s.t. & x_{14} & = & x_{22}  \\
			     & x_{24} & = &  x_{33} \\
			     & X & \text{is}  & 4 \times 4, \, \text{psd}. 
			     \end{array} 
	\end{equation}
By positive semidefinitess the optimal value of 
\eqref{problem-sturm}  is nonnegative, 
and there is an optimal solution with value 
$0$ (e.g., $X=0$).

 If $X$ is any optimal solution, then by 
 $x_{11}=0$ the first row and column of $X$ is zero, hence 
$x_{14}=0, $ hence $x_{22}=0, \,$ and so on. 
Thus the only  element of $X$ which may be nonzero is $x_{44}.$ 
 We can also check that in a dual optimal solution 
both dual variables $y_1$ and $y_2$ are zero.

Thus $X^*$ and $Z^*$ below is a unique (up to the choice of a nonnegative $x_{44}$) 
pair of primal-dual optimal solutions:
\begin{equation} \label{eqn-Xstar-Zstar-Sturm} 
X^*  = \begin{pmatrix}
	0 & 0 & 0 & 0\\
	0 & 0 & 0 & 0\\
	0 & 0 & 0 & 0\\
	0 & 0 & 0 & x_{44}
\end{pmatrix}, \,  Z^* = \begin{pmatrix}
1 & 0 & 0 & 0\\
0 & 0 & 0 & 0\\
0 & 0 & 0 & 0\\
0 & 0 & 0 & 0
\end{pmatrix},
\end{equation}
This pair fails strict complementarity. 

However, for any $\epsilon > 0$ and $x_{44} \geq 3$ 
 the following $X_\epsilon$ is an approximate optimal solution:
 \footnote{To show  $X_\epsilon \succeq 0,$ we
 	subtract $\epsilon^{-1/2}$ times the first column from the last, and do the analogous operation on rows. This replaces $x_{44}$ by $x_{44}-1. $ 
 	In a similar manner we proceed to eliminate all offdiagonal elements.} 
\begin{equation} \label{eqn-Xepsilon} 
	X_\epsilon = \bpx \epsilon & 0 & 0 & \epsilon^{1/2} \\
                    0 & \epsilon^{1/2} & 0 & \epsilon^{1/4} \\
                    0 & 0 & \epsilon^{1/4} & \epsilon^{1/8}  \\
                    \epsilon^{1/2} & \epsilon^{1/4} & \epsilon^{1/8} & x_{44} \epx.
\end{equation}
 We see that $X_\epsilon$ satisfies all constraints. However, its distance to the set of optimal solutions 
is at least $\epsilon^{1/8}, \,$ which is much  larger than $\epsilon.$ Say,
 suppose  $\epsilon=10^{-6}, \,$ so $X_\epsilon$ is declared optimal by most solvers; then $\epsilon^{1/8}=0.1778, $ which is five orders of magnitude larger.

A modification of Example 2 in \cite{sturm1999using} is stated over $n \times n$ matrices: 
it shows an approximate solution with $\epsilon$ constraint violation, whose 
distance to the optimal solution set is $\epsilon^{1/2^{n-1}}.$ 
We will study this example in detail  in Subsection   
\ref{subsection-computational-results-Sturm}. 

\end{Example} 

We  focus on the following questions: 
\begin{enumerate}[label=(Q\arabic*)]
	\item \label{question1} How common are these pathological behaviors -- i.e., huge discrepancies between forward and backward error ?
	In fact, can we add more equations to the SDP \eqref{problem-sturm}, or to its $n \times n$ variant, while preserving both the lack of strict complementarity and its pathological behavior? 
	
	\item \label{question2}  Suppose both \eqref{p} and \eqref{d} satisfy Slater's 
	condition, i.e. there is a positive definite feasible $X$ in \eqref{p} and 
	a positive  definite feasible $Z$ in \eqref{d}. (Note that \eqref{problem-sturm} satisfies 
	Slater's condition, but its dual does not.) 	
	Does this resolve the discrepancy 
	between backward and forward errors? 
	
\item \label{question3} Less immediately, we also ask: 
is there a compact representation of {\em all} SDPs that lack strict complementarity?

\end{enumerate}

To sketch our contributions, we revisit Example \ref{example-sturm}:
\begin{Example} \label{example-sturm-2} (Example \ref{example-sturm} continued) 
	We write \eqref{problem-sturm} as 
	an SDP with 
	\begin{equation} \label{problem-sturm-2} 
		\begin{aligned}
			A_1 &= \begin{pmatrix}
				0 & 0 & 0 & -1 \\
				0 & 2 & 0 & 0 \\
				0 & 0 & 0 & 0 \\
				-1 & 0 & 0 & 0
			\end{pmatrix},
			\quad
			A_2 = \begin{pmatrix}
				0 & 0 & 0 & 0 \\
				0 & 0 & 0 & -1 \\
				0 & 0 & 2 & 0 \\
				0 & -1 & 0 & 0
			\end{pmatrix},
			\quad
			b = (0,0)^\top, \\[1ex]
			C &= \begin{pmatrix}
				1 & 0 & 0 & 0 \\
				0 & 0 & 0 & 0 \\
				0 & 0 & 0 & 0 \\
				0 & 0 & 0 & 0
			\end{pmatrix},
		\end{aligned}
	\end{equation}
	and recall $X^*$ and $Z^*$ from \eqref{eqn-Xstar-Zstar-Sturm}.
	As we discussed before, the equations 
	\begin{equation} \label{eqn-certify-Xstar}
		\la Z^*, X \ra = \la A_1, X \ra = \la A_2, X \ra = 0
	\end{equation}
	certify that $X^*$ is a maximum rank optimal solution in the primal.

We also want to certify that $Z^*$ is a maximum  rank optimal solution in the dual. For that, let  
	\begin{equation} \label{eqn-Y1Y2-sturm} 
		\begin{aligned}
			Y_1 &= \begin{pmatrix}
				0 & 0 & 0 & 0  \\
				0 & 0 & 0 & 1 \\
				0 & 0 & 1 & 0 \\
				0 & 1 & 0 & 0
			\end{pmatrix},
			\quad
			Y_2 = \begin{pmatrix}
				0 & 0 & 0 & 1 \\
				0 & 1 & 0 & 0 \\
				0 & 0 & 0 & 0 \\
				1 & 0 & 0 & 0
			\end{pmatrix}.
		\end{aligned} 
	\end{equation}
	Then $Y_1$ and $Y_2$ have zero inner product with $A_1, A_2, \,$ and with $C$ so they also have zero inner product with 
	any solution feasible in the dual.
	We claim that the equations 
		\begin{equation} \label{eqn-certify-Zstar}
		\la X^*, Z \ra = \la Y_1, Z \ra = \la Y_2, Z \ra = 0
	\end{equation}
		certify that $Z^*$ is a maximum rank optimal solution in the dual. Indeed, suppose $Z$ is optimal in the dual. 
		Then $\la X^*, Z \ra = 0$ implies that 
		the last row and column of $Z$ is zero; and the other equations imply that the $3$rd and $2$nd rows and columns are zero as well.
		
\end{Example} 
Can we certify the maximum rank solutions in {\em any} SDP pair that lacks strict complementarity, in such a simple manner?
We certainly cannot: the certificates even in Examples \ref{example-sturm}--\ref{example-sturm-2} may disappear after some elementary transformation.
For example, they are typically ruined if we perform elementary row operations on the equations, and/or a similarity transformation on all the $A_i$ and on 
$C.$  

However, we will show that the same transformations  "clean up" any SDP that lacks strict complementarity. 
To explain the analogy, we recall a classical result from linear algebra:
the linear system of equations $Ax=b$ with $m$ constraints 
is infeasible iff 
we can  bring it to the normal form 
\begin{equation} \label{problem-Aprimex=bprime} 
	\begin{array}{rcl}
	A^\prime x  & = & b^\prime  \\
	0^\top x & = & 1,
	\end{array} 
\end{equation}
where $A^\prime$ is some matrix with $m-1$ rows and $b^\prime \in \rad{m-1}.$ 
This normal form has three important features.
First, it is constructed in a very simple fashion, by elementary row operations.
Second, it makes infeasibility of 
$Ax=b$ easy to verify. Third, it allows us to systematically generate any infeasible 
linear system of equations: for that, all we must do is write down a system like 
\eqref{problem-Aprimex=bprime}, then "mess it up" using elementary row operations.
This common sense algorithm  always succeeds and it produces any 
 infeasible linear system.
 
 For most of the rest of the paper, we fix a primal feasible solution and a dual feasible solution 
 $X^*$ and $Z^*:$ 
  \begin{equation} \label{eqn-Xstar-Zstar} 
 	X^* = \bpx 0  & 0 & 0 \\
 	0 & 0 & 0 \\
 	0 & 0 & \Lambda  \epx, \, Z^* = \bpx \Gamma  & 0 & 0 \\
 	0 & 0 & 0 \\
 	0 & 0 & 0 \epx.
 \end{equation}	
 Here $\Lambda$ and $\Gamma$ are diagonal  positive definite, of order $r$ and $s, \,$ 
 respectively, and $r + s < n.$ 
 Since $X^*$ and $Z^*$ have zero inner product, they are both optimal.
 
 Our first main contribution, Theorem \ref{thm-main} will characterize when $X^*$ and $Z^*$ are both maximum rank optimal solutions.
  We sketch  it in the

 {\bf Informal Theorem \ref{thm-main}}  Assume that $X^*$ and $Z^*$ given in 
 \eqref{eqn-Xstar-Zstar} is a pair of optimal solutions 
 in \eqref{p}$\mhyphen$\eqref{d}. Then 
 they are both maximum rank optimal solutions 
 $\Leftrightarrow$ there  are elementary row operations and similarity transformations performed on all the $A_i$ and $C,$
  after which 
 \begin{itemize}
 	\item  $X^*$ and $Z^*$ are still a pair of optimal solutions in the same form given in \eqref{eqn-Xstar-Zstar}; and 
     \item \eqref{p} is transformed into a form  with constraint matrices denoted by $A_i^\prime$ and objective denoted by $C^\prime$ such that 
 \begin{itemize}
 	\item the equations
 		\begin{equation} \label{eqn-certify-Xstar-informal}
 		\la Z^*, X \ra = \la A_1^\prime, X \ra = \dots = \la A_k^\prime, X \ra = 0
 	\end{equation}
 	certify that $X^*$ is a maximum rank primal optimal solution.
 	\item the equations
 		\begin{equation} \label{eqn-certify-Zstar-informal}
 		\la X^*, Z \ra = \la Y_1, Z \ra = \dots = \la Y_\ell, Z \ra = 0
 	\end{equation}
 		certify that $Z^*$ is a maximum rank dual optimal solution.
\end{itemize}
\end{itemize}
 Here $k$ and $\ell$ are positive integers and $Y_1, \dots, Y_\ell \in \symn$ are suitable matrices.
 
 The equations \eqref{eqn-certify-Xstar-informal}  "certify" the maximum rank by 
 an elementary linear algebraic argument:  any $X$ that satisfies them must have its first $n-r$ rows and columns equal to zero.
 The equations \eqref{eqn-certify-Zstar} similarly certify that any $Z$ that satisfies them must have its last $n-s$ rows and columns equal to zero. 
  \co{
 
 Before we get to the precise statement in Subsection \ref{subsection-thenormalform}, we explain this  by Figure \ref{fig:sparsity-interval}. This figure depicts depicts the sparsity structure of $X^*, Z^*, \,$ the  $A_i, \, $ and the $Y_j$ in an SDP which is in our normal form. Here  $n=10, m=5, \, $ and $k = \ell=3.$  (Also, in the figure we do not use the $()^\prime$ notation ).  
 
  The following elements in these matrices in the upper row 
  are positive: the $(1,1)$ and $(2,2)$ elements of $Z^*;$ the $(3,3)$ and $(4,4)$ elements of $A_1; $ the 
  $(5,5)$ and $(6,6)$ elements of $A_2; $ and the $(7,7)$ and $(8,8)$ elements of $A_3.$ Thus, $\la Z^*, X \ra = 0$ implies that the diagonal elements in the upper left corner of $X$ are zero, hence the first two rows and columns of $X$ are zero. Continuing, the equations \eqref{eqn-certify-Xstar-informal} indeed certify that in any $X$ that satisfies those equations the first $8$ rows and columns are zero. 
 
 \begin{figure}[h]
 	\centering
 	
 	\includegraphics[width=\textwidth]{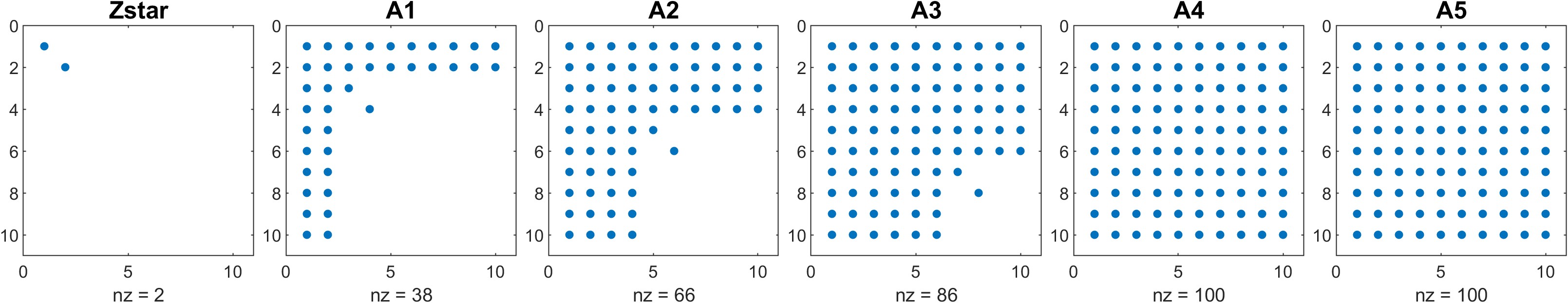}
 	
 	\vspace{1cm}
 	
 	\includegraphics[width=0.75\textwidth]{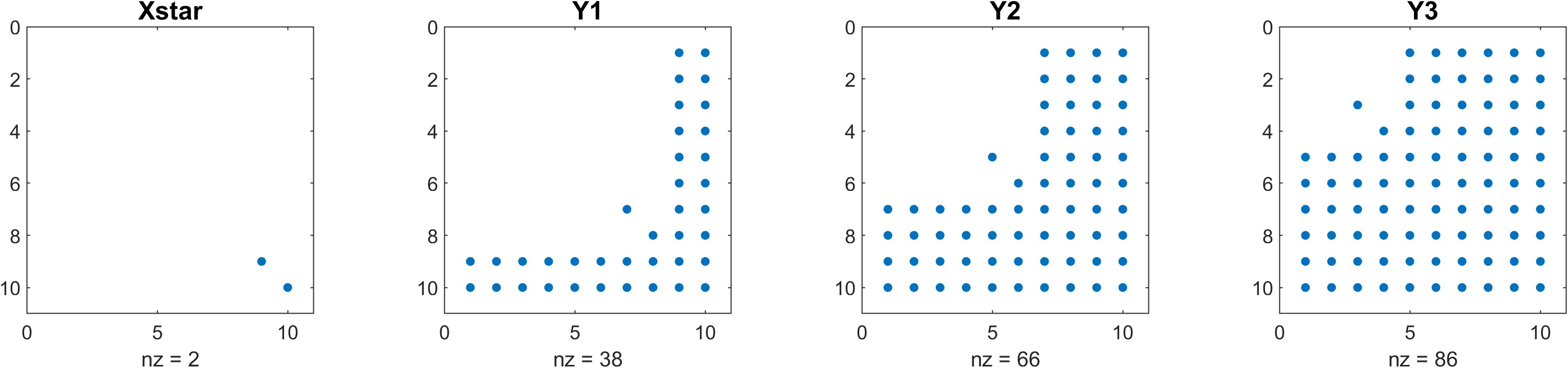}
 	
 	\caption{Sparsity structure of $A_i$ and $Y_j$ in an SDP which is in the normal form of Theorem \ref{thm-main}} 
 \label{fig:sparsity-interval}
 \end{figure}
 
}
 Thus, Theorem \ref{thm-main} provides a  normal form of SDPs that lack strict complementarity.
Further, 
\begin{enumerate}

	\item \label{item2} We build on this normal form and present 
	a simple algorithm that always succeeds
	and whose outputs include 
	every SDP in the form of \eqref{p} which lacks
	 strict complementarity with its dual.

	\item \label{item3} We then present a variant of our algorithm to produce any SDP pair that lacks strict complementarity, but in which  both primal and dual satisfy Slater's condition.
	
	\item  Next we precisely characterize when \eqref{p} has singularity degree 
	equal to the number of constraints. This result plays an important role 
	in our generating algorithms; further, given the importance of the singularity degree, we believe it is of independent interest.

		\item Lastly, we generate and share a  challenging 
	set of SDP test problems and present a computational study. 
	We first answer questions \ref{question1} and \ref{question2} above. In particular, 
	our algorithms generate many SDPs which are not as structured as \eqref{problem-sturm} and have moderate sizes ($n=10, \,$ or $n=20$); still the forward error is seven orders of magnitude larger than the constraint violation.
	
	Besides, our algorithms  
	also produced  SDPs, which are challenging in the traditional sense: 
	in other words, leading solvers struggle to find near optimal solutions with small constraint violation. 
\end{enumerate}

Theorem \ref{thm-main} and items \eqref{item2} and \eqref{item3} answer question \ref{question3} in the affirmative. 
Put slightly differently, they mean that 
we \emph{parametrize} the data  of all SDPs that lack strict complementarity by the normal form and by elementary row operations and similarity 
transformations. Thus, our results put SDPs that lack strict complementarity into a category similar to infeasible systems of equations.

To be fair, the analogy with the normal form \eqref{problem-Aprimex=bprime} is not perfect. Indeed,  to construct our normal form of SDPs that lack strict complementarity, we would need to find an exact  solution to an alternative system 
\eqref{eqn-A*y-alternative}. 
A  more perfect analogy is with square matrices with given eigenvalues;  
these are parametrized by the Jordan normal form and their eigenvectors. 
Normal forms and parametrizations in a similar spirit 
were given for "well behaved" \cite{pataki2019characterizing} semidefinite systems; infeasible \cite{LiuPataki:15}, 
and weakly infeasible semidefinite systems \cite{pataki2022echelon}.

Our reformulation process is a simplified version of {\em facial reduction}, a technique 
that iteratively constructs the smallest face of the semidefinite cone  that contains the feasible set. 
Facial reduction originated in \cite{BorWolk:81}, and a simplified version was given e.g., in 
\cite{Pataki:13}. For a survey that also covers its applications in SDPs that arise in combinatorial optimization we refer to 
\cite{drusvyatskiy2017many}.

 Our paper relies  on minimal technical machinery: the only result from optimization theory we use is the theorem of the alternative   in Proposition \ref{prop-fr}.
The rest of our constructions rely on  elementary linear algebra only.

\paragraph{Organization of the paper, and reader's guide}
Subsection \ref{subsection-literature-review} reviews related literature.
Subsection \ref{subsection-notation-and-preliminaries} introduces notation that is used throughout  the paper. 
In Subsection \ref{subsection-thenormalform} we state 
Theorem \ref{thm-main} about our normal form,
 and to build intuition we prove the "easy" direction.
 In Section \ref{section-generate-nonstrict-SDP} we present our generating algorithms.
  Section \ref{section-the-singularity-degree}  presents our results on the singularity degree.
  Section \ref{section-computational-results} gives our computational 
 results. For ease of exposition we present the proof of Theorem \ref{thm-main} 
 in Section \ref{section-proof-of-theorem-1}.

\subsection{Related literature} 
\label{subsection-literature-review} 

The first papers that proved local quadratic convergence of an ipm under strict complementarity and nondegeneracy assumptions 
were \cite{kojima1998local} and \cite{potra1998superlinear}. The need for nondegeneracy was eliminated in \cite{luo1998superlinear} and 
\cite{kojima1999predictor}; and the proof in the latter was completed in \cite{lu2005note}. 

Since the work of Sturm \cite{Sturm:00}, the singularity degree of semidefinite programs has been studied extensively.
For example, \cite{drusvyatskiy2017note} showed that the convergence rate of the alternating projection method for semidefinite feasibility problems is precisely governed  by the singularity degree.
In \cite{sremac2021error} the authors proved that for a certain class of
central paths, the eigenvalues of the primal and dual optimal solutions 
vanish at a fast rate precisely when the singularity degree is at most one. 

Instance generators for SDPs that fail strict complementarity were described in \cite{wei2010generating} and \cite{mohammadisiahroudi2024generating}. These generators have been useful for computational testing, e.g. in \cite{sremac2021error}. However  -- in contrast to our generators given in Section \ref{section-generate-nonstrict-SDP} -- they do not capture all SDPs that fail strict complementarity.

Error bounds in mathematical optimization have a rich literature and here we can mention only a few important references.
The work \cite{lewis1998error} extended  Hoffman's  error bound philosophy from linear inequality systems to  convex inequality systems. 
This paper also showed that Slater's condition is necessary for a global error bound to hold.
Recent work \cite{pena2019new} generalized Hoffman's constant to the case when the system of inequalities includes 
a {\em reference polyhedron;} \cite{pena2019new}  also designed algorithms to compute the Hoffman constant.

 The error bounds furnished in these works are {\em Lipschitzian, } i.e., the relation between 
the two errors is linear. In contrast, the celebrated Sturm error bound \eqref{eqn-sturm-bound} 
is {\em H\"{o}lderian, } i.e, the forward error is bounded by a  power of the backward error. 
A more general error bound, for mixed semidefinite-second order conic systems  of this type was 
given in \cite{luo2000error}. We mention an error bound for conic linear programs over amenable cones given in 
\cite{lourencco2017amenable}, a class of cones which include the semidefinite and second order cones.

\subsection{Notation and preliminaries} 
\label{subsection-notation-and-preliminaries} 

All vectors in the paper are column vectors.

We denote by $\symn$ the set of $n \times n$ symmetric matrices, and by 
$\psdn$ the set of $n \times n$ symmetric psd matrices.
We denote by $[n]$ the set $\{1,\dots,n\}$.
To write \eqref{p}\text{-}\eqref{d} in a more compact form, 
we define the linear operator $\A$ and its adjoint as 
$$
\A X = (\la A_1, X \ra, \dots, \la A_m, X \ra)^\top, \, \A^* y = \sum_{i=1}^m  y_i A_i.  
$$
We assume throughout that all the $A_i$ are linearly independent.

For a matrix $M \in \rad{m \times n}$ we define the column vector 
$\myvec(M) \in \rad{mn}$ obtained by stacking the columns of $M.$ 
For 
$T \in \symn$ we define the vector $\svec(T) \in \rad{n(n+1)/2}$ as follows:
	first we multiply the offdiagonal elements of $T$ by $\sqrt{2}, \, $ then 
	we stack the upper triangular part of $T$ into $\svec(T).$ Thus, for $S, T \in \symn$ we have 
	$$
	\la S, T \ra = \la \svec(S), \svec(T) \ra. 
	$$
	
Given a matrix $M \in \mathbb{R}^{n \times n}$ and $R, S  \subseteq [n]$ we denote the submatrix of $M$ corresponding to rows in $R$ and columns in $S$ by $M(R,S).$ When $R = \{r\}$ is a singleton, we simply write $M(r,S)$ for $M(\{r\}, S).$  For brevity, we let $M(R):=M(R,R).$

The two main players in the paper are {\em reformulations} and {\em regular facial reduction sequences}, which we now define.

\begin{Definition} \label{definition-reformulation} 
	We say that we 
	\begin{enumerate}
		\item {\em rotate} a set of matrices say $M_1, \dots, M_t, \,$ by an invertible  matrix
		 $T \,$  if we replace 
		$M_i$ by $T^\top M_i T$ for all $i; $ and that we \emph{rotate}  \eqref{p} by an invertible matrix $T \,$ if we rotate all $A_i$ and 
		$C$ by $T.$ 
		
		Equivalently, we  say that we  apply the rotation 
		$T^\top( \cdot)T$ to these matrices, or to \eqref{p}. \footnote{We do not assume 
		$T$ to have orthonormal columns, so strictly speaking the terminology 
		"similarity transformation" would be more precise. We call these transformations rotation for simplicity.} 
		
		\item 
		{\em reformulate}  \eqref{p}  if we apply the following operations (in any order):
		\begin{enumerate}
			\item  We rotate \eqref{p} by an invertible  matrix. 
			\item \label{exch}  For some $i \neq j$ we exchange equations 
			$$\la A_i, X \ra = b_i \; \text{and} \; \la A_j, X \ra = b_j. \;  $$
			\item \label{replace} We replace an equation by a linear combination of equations. That is, for some \mbox{$i \in \{1, \dots, m \}$} 
			 and $y \in \rad{m}$ such that $y_i \neq 0$ 
			we replace 
			$$\la A_i, X \ra = b_i \; \text{by} \; \la \A^* y, X \ra = \la b, y \ra \; 
			$$
		
		\end{enumerate}
		
		\item  a \emph{reformulation} of \eqref{p} is any SDP obtained by applying some of the above operations (in any order). 
		
	\end{enumerate}
\end{Definition}

To construct a reformulation of \eqref{p} it is enough to rotate all $A_i$ and $C$ just once, by the product of all rotation matrices used in the reformulation process. 


The following proposition is immediate from the definition of a reformulation:
\begin{Proposition} \label{prop-reform-invariance} 
	Suppose we reformulate \eqref{p} and we let $T$ be the product of all rotation matrices used in the reformulation process. 
Then 
	\begin{enumerate}
		\item \label{prop-reform-invariance-p}  $X$ is feasible (optimal) in \eqref{p} before the reformulation iff 
		$T^{- 1 }  X T^{- \top}$  is feasible (optimal) after the reformulation.
		\item \label{prop-reform-invariance-d} $Z$ is feasible (optimal) in \eqref{d} before the reformulation iff 
			$T^{\top}  Z T$    is feasible (optimal) after the reformulation. 
		\end{enumerate}
	\qed 
\end{Proposition}

For a sequence of real numbers $r_1, r_2, \dots, r_t$ and $1 \leq i \leq j \leq t$ we write 
$$
r_{i:j} := r_i + r_{i+1} + \dots + r_j. 
$$
Also, for brevity we just write
	$r_{1:i-1}$ for $r_{1:(i-1)};$ and $r_{1:i+1}$ for $r_{1:(i+1)};$ etc.

\begin{Definition} \label{definition-regfr} 
	We say that $(Y_1, \dots, Y_t)$ is a {\em regular facial reduction sequence}  {\em for $\psd{n}$} 
	if i) all $Y_i$ are in $\symn; $ and ii) after 
	a simultaneous permutation of rows and columns
 they have the form 
	
		\begin{equation} \label{eqn-regfr} 
			Y_1   = 
			\bordermatrix{
				& \overbrace{\qquad}^{\textstyle r_{1}} &  \overbrace{\qquad \qquad}^{\textstyle{n-r_1 }} \cr\\
				& \Lambda_1  &  0   \cr
				& 0   &  0  \cr}, \, \dots, Y_i  = 
			\bordermatrix{
				& \overbrace{\qquad }^{\textstyle r_{1:i-1}} & \overbrace{\qquad}^{\textstyle r_{i}} & \overbrace{\qquad}^{\textstyle n-r_{1:i}} \cr\\
				& \times  &  \times  &  \times \cr
				& \times  &  \Lambda_i    &  0 \cr
				& \times  &  0  &  0 \cr} 
		\end{equation}
		for  $i= 1,\dots, t. \,$ Here the $r_i$ are nonnegative integers, the $\Lambda_i$ diagonal positive definite matrices, and the $\times$ symbols 
	stand for blocks with arbitrary elements. 
	
	If $(Y_1, \dots, Y_t), \,$ is a regular facial reduction sequence of the form
	\eqref{eqn-regfr}, then we say it {\em has structure 
	$(P_1, \dots, P_t), \,$ } where  
	\begin{equation}
		\ba{rcl}
		P_1 & = &  \{ 1, \dots, r_1  \}   \\
		P_2 & = & \{ r_1 + 1, \dots, r_{1:2} \} \\
		& \vdots & \\
		P_t & = & \{ r_{1:t-1} + 1, \dots, r_{1:t} \}.
		\ena
	\end{equation}
	If rows and columns in all the $Y_j$ are simultaneously permuted by the permutation $\pi, \,$ then the resulting  regular facial reduction sequence has 
	structure  $(\pi(P_1), \dots, \pi(P_t)).$

\end{Definition}

	\begin{Example} \label{example-sturm-3} 
		In Example \ref{example-sturm-2}
		\begin{itemize}
			\item 
			$(Z^*, A_1, A_2)$ is a regular facial reduction sequence with structure 
			$(\{1\}, \{2\}, \{3 \}).$
			\item  $(X^*, Y_1, Y_2)$ is 
			a regular facial reduction sequence with structure 
			$(\{4\}, \{3\}, \{2 \}).$
		\end{itemize} 
	
	\end{Example}

Regular facial reduction sequences certify that certain rows and columns of a psd matrix must be zero, just like they 
did in Example \ref{example-sturm-2}. To see this, suppose $(Y_1, \dots, Y_t)$ is such a sequence, with the $r_i$ as in \eqref{eqn-regfr}, and let $X \succeq 0$ satisfy $\langle Y_i, X \rangle = 0$ for all $i$. Then $\langle Y_1, X \rangle = 0$ implies that the first $r_1$ rows and columns of $X$ are zero; similarly, $\langle Y_2, X \rangle = 0$ implies that the next $r_2$ rows and columns are zero, and so on.

For brevity, we introduce the following terminology. 
Let $(Y_1, \dots, Y_t)$ be a regular facial reduction sequence. 
We say that $(Y_1, \dots, Y_t)$ reduces $\psd{n}$ to 
$\psd{n} \cap Y_1^\perp \cap \dots \cap Y_{t}^\perp.$  

For nonnegative integers $r$ and $s$, we use the notation
$$
\begin{pmatrix}
0  & 0 \\
0 & \oplus_r  \end{pmatrix}, \, \text{and by} \, 	\bpx \oplus_s  & 0 \\
0 & 0 \epx.
$$ 
to denote, respectively, 
the subset $\psd{n}$  in which only the lower right order $r$ minor may be nonzero; and 
the subset of $\psd{n}$  in which only the upper left order $s$ minor may be nonzero.

\subsection{The normal form for lack of strict complementarity}
\label{subsection-thenormalform}

In this subsection we formally state our main result, Theorem \ref{thm-main}. We need one more definition. Recall 
$X^*$ and $Z^*$ from \eqref{eqn-Xstar-Zstar}. We will rely on rotations with a special structure. 
We say that a rotation matrix $T$ is {\em admissible,} 
if it has the form 
$$
T \, = \, \bpx I_s  & 0 & 0 \\
0 & U & 0 \\
0 & 0 & I_r  \epx
$$
for some $U$ invertible matrix of order $n-r-s.$ 
  
\begin{Theorem} \label{thm-main} 
	 Assume that $X^*$ and $Z^*$ given in 
	 \eqref{eqn-Xstar-Zstar} is a pair of optimal solutions 
	 in \eqref{p}$\mhyphen$\eqref{d}. Then 
	 they are both maximum rank optimal solutions 
		  $\Leftrightarrow$  there are positive integers  $k$ and $\ell$ 
	and a reformulation 
	\begin{equation}\label{p-prime}  
		\begin{array}{rl} 
			\inf  & \,\, \la C^\prime, X   \ra  \\
			s.t. & \,\, \la A_i^\prime, X \ra \, =  \, b_i^\prime \; \text{for} \; i=1, \dots, m  \\
			& \,\, X \succeq 0, 
		\end{array} \tag{$P^\prime$} 
	\end{equation}
	with the following properties:
	\begin{enumerate}
		\item  \label{thm-main-1} 
		\begin{enumerate}
			\item \label{thm-main-1-a}  $b_1^\prime  = \dots = b_k^\prime  = 0.$ 
			\item  \label{thm-main-1-b}  $(Z^*, A_1^\prime, \dots, A_k^\prime)$ is a regular facial reduction sequence, which reduces 
			$\psd{n}$ to 
			$$
			\bpx 0 & 0 \\
			0 & \oplus_r  \epx.
			$$
		\end{enumerate}  
		
		\item \label{thm-main-2}  
			there is $Y_1, \dots, Y_\ell$ such that 
					\begin{enumerate}
				\item \label{thm-main-2-a}
				\begin{equation} \nonumber 
			\begin{rcases}
				\mathcal{A}^{\prime}Y_j = 0 \\
				\langle C^{\prime}, Y_j \rangle = 0
			\end{rcases}
			\quad \text{for } j = 1,\dots,\ell .
		\end{equation}
		\item \label{thm-main-2-b} $(X^*, Y_1, \dots, Y_\ell)$ is a regular facial reduction sequence, 
		which reduces 
		$\psd{n}$ to 
		$$
		\bpx \oplus_s  & 0 \\
		0 & 0 \epx.
		$$
		
	\end{enumerate}
	\item \label{thm-main-3} The only rotations in the reformulation process are admissible.

		\end{enumerate}
	 Here we understand that $\A^\prime$ is represented with 
	 symmetric matrices   $A_i^\prime$ as 
	 \begin{equation} \label{eqn-repr-Aprime} 
	 	\A^\prime X \, = \, ( \la A_1^\prime,  X \ra, \dots, \la A_m,  X \ra)^{\top}. 
	 \end{equation}
	 
\end{Theorem}

To  build intuition, here we prove the ``easy" direction. 

{\bf Proof of $\Leftarrow$:} Let us assume the statement on the right hand side, and denote the dual of \eqref{p-prime} by $(D^\prime).$ 
Let $T$ be the product of all matrices used in the reformulation process. 
Then $T$ is admissible. 

We first need that $X^*$ is optimal in  \eqref{p-prime} and $Z^*$ is optimal in $(D^\prime).$ 
Indeed, $T^{- 1} X^* T^{- \top}  =  X^* $ and 
$T^{\top} Z^* T  =  Z^*, $ 
so this follows by Proposition \ref{prop-reform-invariance}.

We next prove that $X^*$ and $Z^*$ are  maximum rank optimal solutions in \eqref{p} and \eqref{d}, respectively. 
For that, let $X$ be an optimal solution in \eqref{p} and 
  $Z$ an optimal solution in \eqref{d},   and define 
 \begin{equation} \label{eqn-XXprime} 
 	\begin{array}{rcl}
 		X^\prime & := & T^{- 1 } X T^{- \top} \\ 
 		Z^\prime & := & T^{\top} Z T. 
 	\end{array}
 \end{equation}
Proposition \ref{prop-reform-invariance} implies that 
 $X^\prime$ is optimal in \eqref{p-prime}. 
 But $X^\prime$ has $0$ inner product with $Z^*, A_1^\prime, \dots, A_k^\prime.$ 
 Thus  
 $X^\prime $ is in 
  $\left(\scriptstyle
 \begin{smallmatrix}
 	0 & 0 \\
 	0 & \oplus_r
 \end{smallmatrix}
 \right)$
  hence so is $X.$ So indeed, $X^*$ is a maximum rank optimal solution in \eqref{p}.
   An analogous argument proves the statement for $Z^*.$ 
  \qed 

\begin{Example} \label{example-sturm-3} (Examples \ref{example-sturm} and  \ref{example-sturm-2} continued) 
The SDP \eqref{problem-sturm-2} does not need to be reformulated, as it is already in the normal form stated in Theorem \ref{thm-main}.
\end{Example}

The following remarks and Corollary \ref{Corollary-all-partitions} 
explain some details  in Theorem \ref{thm-main}.

First, Theorem \ref{thm-main} proves  a  stronger result than 
$X^*$ having maximum rank: it shows that when $X^*$ given in \eqref{eqn-Xstar-Zstar} 
has maximum {\em rank}, it also has "maximum support." 
That is, any optimal solution is in 
$\left(\scriptstyle
\begin{smallmatrix}
	0 & 0 \\
	0 & \oplus_r
\end{smallmatrix}
\right)$
This notion of "maximum support" differs 
from the  usual one: it does not mean that 
an optimal solution in \eqref{p} has nonzeros only where $X^*$ does.

Second, we may assume that the positive definite blocks in 
$A_1^\prime, \dots, A_k^\prime$ and in all $Y_j$ are nonempty; 
otherwise, we can simply remove the corresponding 
$A_i^\prime$, $b_i^\prime$, and $Y_j$.

Third, suppose \eqref{p-prime} satisfies the conclusions of Theorem 
\ref{thm-main}. Consider the index sets corresponding to 
the positive definite blocks in $A_1^\prime, \dots, A_k^\prime$ 
and the index sets corresponding to 
the positive definite blocks in $Y_1, \dots, Y_\ell.$ 
Both these index sets 
partition 
 the set that 
we informally call the "gap set"
\begin{equation} \label{eqn-define-G} 
G := \{s+1, \dots, n - r \}.
\end{equation} 
This argument also works in the other direction: any two partitions of $G$ give rise 
to an SDP that lacks strict complementarity. We state this result precisely:

\begin{Corollary} \label{Corollary-all-partitions}
	Suppose $X^*$ and $Z^*$ are given in \eqref{eqn-Xstar-Zstar} and 
	$(P_1, \dots, P_k)$ and $( Q_1, \dots, Q_\ell)$ partition the set $G$ 
	defined in \eqref{eqn-define-G}. 
		Let 
	$$
	P_0 = \{1, \dots, s\} \; \text{and} \; Q_0 = \{  n - r  +1, \dots, n \, \}.   
	$$
		Then there is an SDP of the form \eqref{p-prime}  and 
	$Y_1, \dots, Y_\ell$ so the following holds: 
	
				\begin{enumerate}
				\item  \label{Corollary-all-partitions-1}    \eqref{p-prime} with 	$Y_1, \dots, Y_\ell$ 
				satisfies the conclusions of Theorem \ref{thm-main}; 
				\item  \label{Corollary-all-partitions-2} the 
				structure of $(Z^*, A_1^\prime, \dots, A_k^\prime)$ is $(P_0, P_1, \dots, P_k).$ 
				%
				\item  \label{Corollary-all-partitions-3} the structure of $(X^*, Y_1, \dots, Y_\ell)$ is $(Q_0, Q_1, \dots, Q_\ell).$
			\end{enumerate}
			
			\qed
	\end{Corollary}
		Corollary \ref{Corollary-all-partitions} means that for a fixed rank of $X^*$ and $Z^*$ 
	all SDPs that fail strict complementarity can be identified with two partitions of the gap set $G.$ Thus, loosely speaking, 
	there are only finitely many 
	"representative" SDPs that fail strict complementarity.

We prove Corollary \ref{Corollary-all-partitions} in Subsection 
\ref{subsection-generate-nonstrict-SDP-nonSlater}.

\section{Generating SDPs that lack strict complementarity}
\label{section-generate-nonstrict-SDP} 

In this section we use the normal form of Theorem \ref{thm-main} to 
generate any SDP that fails strict complementarity. The main idea is as follows. We fix 
$X^*$ and $Z^*$ as in \eqref{eqn-Xstar-Zstar}. 
In the SDPs we construct we just write $A_i $ and $C$ for the matrices, i.e., we omit the 
prime notation.  
We first construct a minimal  SDP which fails strict complementarity, and it is
in the normal form of Theorem \ref{thm-main}. By "minimal" we mean $k=m.$ 

We also construct 
the accompanying $Y_j$ and ensure the following 
{\em base equations} hold:
\begin{equation} \label{eqn-base} \tag{$\mathrm{BASE}$} 
		\la A_i, Y_j \ra = 0 \, \text{for} \; i=1, \dots, k; j=1, \dots, \ell. 
	\end{equation} 
Next we find $A_{k+1}, \dots,  A_m$ to have zero inner product with all the $Y_j$ 
 and set $b$ and $C$ suitably. Thus, afterwards 
$(Z^*, A_1, \dots, A_k)$ still certify maximum rank of $X^*$ and still certify maximum rank of
$(X^*, Y_1, \dots, Y_\ell)$ of $Z^*.$

We give two algorithms. Algorithm \ref{algo-nonstrict-SDP} in Subsection \ref{subsection-generate-nonstrict-SDP-nonSlater} 
outputs an SDP that fails strict complementarity. 
Further, this algorithm has {\em any} such SDP among its outputs.
Accordingly, the generated SDPs may not satisfy Slater's condition on the primal or dual side (or both).

 Algorithm \ref{algo-nonstrict-Slater-SDP} in Subsection \ref{subsection-generate-Slater-nonstrict-SDP} outputs an SDP which fails 
strict complementarity, but satisfies Slater's condition on both the primal and dual sides.
This algorithm has {\em any} such SDP among its outputs.

\subsection{Without guaranteeing Slater's condition} 
\label{subsection-generate-nonstrict-SDP-nonSlater}

The main generating algorithm now follows:

	\begin{algorithm}[H]
		\caption{Construct SDP that fails strict complementarity} 
		\label{algo-nonstrict-SDP}
		\begin{algorithmic}[1]
			\Statex  \label{algo-nonstrict-SDP-input}  {\bf Input:} Integers $n \geq 3, \;  m \geq 2.$
			\Statex  \label{algo-nonstrict-SDP-output} {\bf Output:} 
				\begin{itemize}
				\item SDP in the normal form of \eqref{p}
				which fails strict complementarity with \eqref{d}, 
				and $Y_1, \dots, Y_\ell$ such that
								 \eqref{p} and $Y_1, \dots, Y_\ell$
					satisfy   the conclusions of Theorem \ref{thm-main}.
												\item A maximum rank  optimal solution $X^*$ of \eqref{p} and
				$Z^*$ of \eqref{d} of the form  \eqref{eqn-Xstar-Zstar}.
			\end{itemize}

			\State  \hspace{.05cm} Choose $\; r, s \, $ positive integers such that $r + s < n $ and $ k, \ell \in \{1, \dots, n-r-s \}.$ 
	\State \label{algo-nonstrict-SDP-choose-X*-Z*} Choose $X^*$ and $Z^*$ of the form 
	\eqref{eqn-Xstar-Zstar}, where $\Lambda$ and $\Gamma$ are positive definite, 
	of order $r$ and  $s, \,$ respectively.
	\State \label{algo-nonstrict-SDP-choose-Ai-Yj} Choose $A_1, \dots, A_k$ and $Y_1, \dots, Y_\ell$ such that 
	\begin{itemize}
		\item  $(Z^*, A_1, \dots, A_{k})$ is a regular facial reduction sequence which reduces $\psd{n}$ to 
		$\bpx 0 & 0 \\ 0 & \oplus_r \epx. $ 
		\item  $(X^*, Y_1, \dots, Y_{\ell})$ is a regular facial reduction sequence which reduces 
		$\psd{n}$ to 
		$\bpx \oplus_s  & 0 \\ 0 & 0 \epx. $ 
	\end{itemize}   
	\State \label{algo-nonstrict-SDP-adjust-Ai-Yj} Use Algorithm \ref{algo-base-eqn}, the 
	Base Equations Algorithm,  to make sure 
	the  \eqref{eqn-base} equations are satisfied.
		\State  \label{algo-nonstrict-SDP-extend}  Choose $A_{k+1},\dots,A_m$ so they have zero inner product 
		with $Y_1,\dots,Y_{\ell}.$
		\State \label{algo-nonstrict-SDP-set-bi} Set $b_i:= \la A_i,  X^* \ra $ for $i \in \{1, \dots, m\}, \, $  $C := Z^* + \sum_{i=1}^m y_i A_i$ for some $y \in \rad{m}.$ 
				\end{algorithmic}
	\end{algorithm}

Next we state the Base Equations Algorithm: 

	\begin{algorithm}[H]
	\caption{Base Equations Algorithm} 
	\label{algo-base-eqn} 
	\begin{algorithmic}[1]
		\Statex {\bf Input:} The sequences $(Z^*, A_1, \dots, A_k)$ and $(X^*, Y_1, \dots, Y_\ell)$
				chosen in Steps \ref{algo-nonstrict-SDP-choose-X*-Z*} and \ref{algo-nonstrict-SDP-choose-Ai-Yj} of Algorithm \ref{algo-nonstrict-SDP}.
		\Statex {\bf Output:} The same sequences, with some blocks adjusted so the $A_i$ and $Y_j$ 
		 satisfy the \eqref{eqn-base} equations.  
		 \State Let $(P_0, P_1, \dots, P_{k})$  be the structure of 
		 $(Z^*, A_1, \dots, A_{k})$ 
		  \State Let  $(Q_0, Q_1, \dots, Q_{\ell})$ be the structure of $(X^*, Y_1, \dots, Y_{\ell}). \, $ 
			\For{$i = 1:k$}
		\State \label{step-adjust-Ai} Adjust the $(P_{i-1},Q_0)$ block of $A_i, \, Y_1, \dots, Y_\ell$ to satisfy 
		$$
		\la A_i, Y_j \ra = 0 \; \text{for all} \; j=1, \dots, \ell.
		$$
		\EndFor
	\end{algorithmic}
\end{algorithm}
We will use the following facts. Given how the sequences 
$(Z^*, A_1, \dots, A_{k})$ and $(X^*, Y_1, \dots, Y_{\ell})$ are constructed, we have 
\begin{equation} \label{eqn-P0-all_Qj} 
	\ba{rcl} 
	Q_0 & = & [n] \setminus P_{0:k}, \\
	P_0 & = & [n] \setminus Q_{0:\ell}.
	\ena
\end{equation}
\begin{Example} \label{example-sturm-3} (Example \ref{example-sturm} and \ref{example-sturm-2} continued) 
	In this example we show how Algorithm \ref{algo-nonstrict-SDP} constructs the SDP \eqref{problem-sturm-2}. 

Suppose Algorithm \ref{algo-nonstrict-SDP} in 
steps \ref{algo-nonstrict-SDP-choose-X*-Z*} and 
\ref{algo-nonstrict-SDP-choose-Ai-Yj} chose the regular facial reduction sequences 
	\begin{equation} \label{problem-sturm-3} 
		\begin{aligned}
			Z^* = \begin{pmatrix}
				1 & 0 & 0 & 0\\
				0 & 0 & 0 & 0\\
				0 & 0 & 0 & 0\\
				0 & 0 & 0 & 0
			\end{pmatrix}, A_1 &= \begin{pmatrix}
				0 & 0 & 0 & \alpha_1 \\
				0 & 2 & 0 & 0 \\
				0 & 0 & 0 & 0 \\
				\alpha_1  & 0 & 0 & 0
			\end{pmatrix},
			\quad
			A_2 = \begin{pmatrix}  
				0 & 0 & 0 & 0 \\
				0 & 0 & 0 & \alpha_2 \\
				0 & 0 & 2 & 0 \\
				0 & \alpha_2  & 0 & 0
			\end{pmatrix},
			\quad \\ \\
					X^* = \begin{pmatrix}
					0 & 0 & 0 & 0\\
					0 & 0 & 0 & 0\\
					0 & 0 & 0 & 0\\
					0 & 0 & 0 & 1
				\end{pmatrix}, \, Y_1 &= \begin{pmatrix}
				0 & 0 & 0 & \beta_{11}   \\
				0 & 0 & 0 & \beta_{12}  \\
				0 & 0 & 1 & 0 \\
				\beta_{11}  & \beta_{12}  & 0 & 0
			\end{pmatrix},
			\quad
			Y_2 = \begin{pmatrix}
				0 & 0 & 0 & \beta_{21} \\
				0 & 1 & 0 & \beta_{22}  \\
				0 & 0 & 0 & 0 \\
				\beta_{21}  & \beta_{22}  & 0 & 0
			\end{pmatrix}.
		\end{aligned}
	\end{equation}
Their  structures are 
\begin{equation}
	\ba{rcl}
	(P_0, P_1, P_2) & = & ( \{ 1 \}, \,  \{ 2 \}, \, \{ 3 \}), \\
	(Q_0, Q_1, Q_2) & = & ( \{ 4 \}, \,  \{ 3 \}, \, \{ 2 \}), 
	\ena
\end{equation}
respectively.
In \eqref{problem-sturm-3} we denote by  $\alpha_i$ and $\beta_{ij}$ 
the entries in
$(P_{i-1}, Q_0)$ for $i=1,2$ that Algorithm \ref{algo-base-eqn} will set.

The goal of Algorithm \ref{algo-base-eqn} in iteration 1 
is to satisfy $\la A_1, Y_1 \ra = \la A_1, Y_2 \ra = 0.$ 
 For that, it  solves the system 
$$
\bpx \beta_{11} \\ \beta_{21} \epx \alpha_1 = \bpx 0 \\ -1 \epx.
$$	
A possible solution is $\alpha_1 = -1, \, \beta_{11} = 0, \, \beta_{21} = 1.$ 
In iteration 2  the goal is to satisfy $\la A_2, Y_1 \ra = \la A_2, Y_2 \ra = 0.$ 
For that, Algorithm \ref{algo-base-eqn} solves the system 
$$
\bpx \beta_{12} \\ \beta_{22} \epx \alpha_2 = \bpx -1  \\ 0 \epx.
$$	
A possible solution is $\alpha_2 = -1, \, \beta_{12}=1, \, \beta_{22} = 0.$ 

If $m=2, \,$ then we stop here. If $m m> 2, \,$ then we run Step \ref{algo-nonstrict-SDP-extend}. If we 
run this step, then it may choose 
$$
A_3 = \begin{pmatrix}
	1 & 0 & 0 & 0\\
	0 & 0 & 0 & 0\\
	0 & 0 & 0 & 0\\
	0 & 0 & 0 & 1
\end{pmatrix}.
$$
Step \ref{algo-nonstrict-SDP-set-bi} may choose $y=0.$ 
Thus, the final  output of Algorithm \ref{algo-nonstrict-SDP} is the SDP \eqref{problem-sturm-2}, with the equation 
$\la A_3, X \ra = 1$ possibly added; $X^*, Z^*$ as given in 
\eqref{problem-sturm-3}; and $Y_1, Y_2$ given in 
\eqref{eqn-Y1Y2-sturm}.

\end{Example} 

Next we prove correctness of our algorithm.
\begin{Theorem} \label{thm-alg-nonstrict-correct} The following hold:
	\begin{enumerate}
		\item 	\label{thm-alg-nonstrict-correct-1} Algorithm \ref{algo-nonstrict-SDP} can be implemented, using Algorithm \ref{algo-base-eqn} as a subroutine, so it i) succeeds in every run, and ii) its output is correct.

			\item \label{thm-alg-nonstrict-correct-2}   Suppose \eqref{p} is an SDP with $m$ constraints over $\psdn$ that  fails 
			strict complementarity with \eqref{d}. 
		Then \eqref{p} is obtained 
		as  a reformulation of some output of 	Algorithm \ref{algo-nonstrict-SDP}.
	\end{enumerate}
	\end{Theorem}
{\bf Proof} We first prove \eqref{thm-alg-nonstrict-correct-1}.
Suppose Algorithm \ref{algo-nonstrict-SDP} in Steps \ref{algo-nonstrict-SDP-choose-X*-Z*} and \ref{algo-nonstrict-SDP-choose-Ai-Yj} chooses $(Z^*, A_1, \dots, A_k)$ and 
$(X^*, Y_1,  \dots, Y_\ell), $ with structure 
 $(P_0, \dots, P_k)$ and 
$(Q_0, \dots, Q_\ell), \, $ respectively.

We prove that Step \ref{step-adjust-Ai} in  Algorithm \ref{algo-base-eqn} can be carried out for all $i.$ For that, let us fix $i \in \{1, \dots, k \}.$ 
 We first set  
the $(P_{i-1}, Q_0)$ block of $A_i$ and of all $Y_j$ to zero and define 
$$
d := - \tfrac{1}{2} \bigl( \la A_i, Y_1 \ra, \dots, \la A_i, Y_\ell \ra \bigr)^\top. 
$$
To compute the $(P_{i-1}, Q_0)$ block of $A_i$ and all $Y_j$ we solve the system 
\begin{equation} \label{eqn-vecYi-1-Q0} 
\bpx \myvec{Y_1(P_{i-1}, Q_0)}^\top \\
  \vdots \\
   \myvec{Y_\ell(P_{i-1}, Q_0)}^\top \epx \bpx \myvec{A_i(P_{i-1}, Q_0)}  \epx = d.
\end{equation} 
To see how to do this, let $M$ denote the constraint matrix on the left hand side, and let 
$x$ denote 
$\myvec{A_i(P_{i-1}, Q_0)}.$ We solve the system $M x = d$ in an unusual setup:
only the right hand side is fixed, and we solve for both  $M$ and 
$x.$  

By elementary linear algebra  all $M$ and $x$ that satisfy $Mx =d$ arise as possible outputs
 of the following algorithm: 
\begin{enumerate}
	\item If $d=0$, let $M$ be arbitrary and choose $x$ such that $Mx=d$.
	
	\item If $d\neq 0$, let $M=[M',d]$, where $M'$ is arbitrary, and
	$x=(0,\dots,0,1)^\top$.  
	\item Choose an invertible matrix $T$ and set $M:=MT$, $x:=T^{-1}x$.
\end{enumerate}

Next we claim that after iteration $i$ of Algorithm~\ref{algo-base-eqn} the previously satisfied equations
$$
\la A_t, Y_j \ra = 0 \; \text{for} \; j=1, \dots, \ell
$$
remain true for all $t < i.$
Indeed,  these equations are not affected by adjusting the $(P_{i-1}, Q_0)$ block of $A_i.$ 
We claim that they are also not affected by adjusting the same block of $Y_j$ for any 
$j \in \{1, \dots, \ell \}. $ To see why, we use that 
\begin{equation} \label{eqn-AtPi-1-Q0} 
	A_t(P_{i-1}, Q_0) = 0 \; \text{for all} \; t < i;
\end{equation}
and \eqref{eqn-AtPi-1-Q0} follows from Definition \ref{definition-regfr}.

Next we prove that Algorithm \ref{algo-nonstrict-SDP} correctly outputs an 
SDP of the form \eqref{p} that fails strict complementarity with \eqref{d}.
Indeed, $X^*$ is feasible in \eqref{p}, and $Z^*$ in \eqref{d}, and they have zero inner product. Thus they are optimal in \eqref{p} and \eqref{d}, respectively. 
Also, the proof of the $\Leftarrow$ direction in Theorem \ref{thm-main} shows that both $X^*$ and $Z^*$ are 
maximum rank solutions in \eqref{p} and \eqref{d}, respectively. This implies our claim.

We next  prove \eqref{thm-alg-nonstrict-correct-2}. For that, suppose \eqref{p} is an 
 SDP that fails  strict complementarity with \eqref{d}. 
  Let $X^*$ be a maximum rank optimal solution in \eqref{p} and $Z^*$ be such a solution in \eqref{d}. We claim that after a suitable rotation we may  assume 
that $X^*$ and $Z^*$ have the form as stated in \eqref{eqn-Xstar-Zstar}. Indeed, one rotation, say by $T_1$  brings $X^*$ into this form. To keep $X^*$ feasible, and keep its zero 
inner product with $Z^*, \,$ we also rotate all $A_i, \, C, \, $ and $Z^*$ by $T_1^{- \top}.$ 
Afterwards,  $\la X^*, Z^* \ra = 0$ implies that the last $r$ rows and columns of $Z^*$ are zero. Then we rotate $Z^*$ by a matrix of the form 
$T_2 := \left( \begin{smallmatrix*}
	U & 0 \\ 0 & I_r  \end{smallmatrix*} \right)$ 
 and also all $A_i$ and $C$ by $T_2.$ 
 This finishes he proof of our claim.

After we made sure that $X^*$ and $Z^*$ are in the right form, we still denote the data 
of \eqref{p} by $A_i$ for $i=1, \dots, m$  and by $b$ and $ C.$  
Then \eqref{p} has a reformulation as described in Theorem \ref{thm-main}. We next claim
that  
$A_1, \dots, A_k$ and $Y_1, \dots, Y_\ell$ are possible outputs of Steps 
\ref{algo-nonstrict-SDP-choose-Ai-Yj}--\ref{algo-nonstrict-SDP-adjust-Ai-Yj} of Algorithm \ref{algo-nonstrict-SDP}. Indeed, 
the blocks outside $(P_{i-1}, Q_0)$ for $i=1, \dots, k$ 
can be  chosen directly, while  the 
$(P_{i-1}, Q_0)$ blocks can be  chosen by solving \eqref{eqn-vecYi-1-Q0} for $i=1, \dots, k.$ 

The other $A_i$ 
are possible outputs of Step \ref{algo-nonstrict-SDP-extend}; and $b$ and $C$ are possible outputs of Step \ref{algo-nonstrict-SDP-set-bi}. 
\qed 

The next two remarks explain details of Algorithms \ref{algo-nonstrict-SDP}
and \ref{algo-base-eqn}.

First, Algorithm \ref{algo-base-eqn} solves the 
nonlinear system of \eqref{eqn-base} equations by a direct algorithm that does not require  backtracking. Indeed, we claim 
that this algorithm can be 
interpreted as solving a triangular system of equations. To prove that, we  first define the  index set 
	$R \subseteq [n] \times [n]$ as 
	$$
	R :=  (P_{1:k-1}, Q_0)  \cup  (Q_0, P_{1:k-1}) 
	$$ 
	Also, for a matrix $M \in \rad{n \times n}$ we write $\bar{M}(R)$ the matrix obtained by 
	setting all elements {\em not} in $R$ to zero. 
	
	Let us fix $j \in \{1, \dots, \ell \}.$ 
	For this $j$ we claim that 
	the \eqref{eqn-base} equations  can be written as the triangular system of equations: 
	\begin{small} 
		\begin{equation}
			\begin{aligned}
				\langle A_1(P_0,Q_0), Y_j(P_0,Q_0) \rangle
				&\phantom{{}+ \langle A_2(P_1,Q_0), Y_j(P_1,Q_0) \rangle}
				&= -  \langle \bar{A_1}(R), \bar{Y_j}(R) \rangle, \\
				\langle A_2(P_0,Q_0), Y_j(P_0,Q_0) \rangle
				&+ \langle A_2(P_1,Q_0), Y_j(P_1,Q_0) \rangle
				&= - \langle \bar{A_2}(R), \bar{Y_j}(R) \rangle, \\
				\vdots \\
				\langle A_k(P_0,Q_0), Y_j(P_0,Q_0) \rangle
				&+ \langle A_k(P_1,Q_0), Y_j(P_1,Q_0) \rangle
				+ \cdots + \langle A_k(P_{k-1},Q_0), Y_j(P_{k-1},Q_0) \rangle
				\!\!\!\!\!\!\!	&= - \langle \bar{A_k}(R), \bar{Y_j}(R) \rangle.
			\end{aligned}
		\end{equation}
	\end{small} 
	Indeed, our claim  follows, since for all $i \in \{1, \dots, k \}$ the relations in  
	\eqref{eqn-AtPi-1-Q0} hold. 	
	Thus, the $i$th step of Algorithm \ref{algo-base-eqn} computes the diagonal block 
	$A_i(P_{i-1}, Q_0)$ and 	$Y_j(P_{i-1}, Q_0)$ for all $j.$

 We next comment on Step \ref{algo-nonstrict-SDP-extend} of Algorithm \ref{algo-nonstrict-SDP}. 
In the previous steps we have chosen  $A_1, \dots, A_k. \,$ As we discussed after the statement of Theorem \ref{thm-main}, we may assume that  
their  positive definite blocks are nonempty, so these matrices are  linearly independent.
Ideally, we would like to choose $A_{k+1}, \dots, A_m$ so that i) they are orthogonal to all the $Y_j$ and that ii) all the $A_i$ are linearly independent.

For that, define  
$$ L_1 := \{Y_1, \dots, Y_\ell \}^\perp, \, L_2 := \lin \{ A_1, \dots, A_k \}.$$ 

Suppose $A_{k+1}, \dots, A_m \in L_1.$ Since 
$L_2 \subseteq L_1, \, $ all the $A_i$ are are linearly independent if and only if 
$$
\{ A_1, \dots, A_k, A_{k+1}^\prime, \dots, A_m^\prime \} 
$$
are linearly independent. Here $A_i^\prime \in L_1 \cap L_2^\perp$  is the projection 
of $A_i$ to $L_2^\perp$ for $i=k+1, \dots, m.$ 
Thus, to find the $A_i$ for $i=k+1, \dots, m$ we  first choose linearly independent 
matrices 
$A_i^\prime \in L_1 \cap L_2^\perp;$ 
then we define the $A_i$ by adding  linear combinations of $A_1, \dots, A_k$ to the $A_i^\prime.$ 


We conclude with the 

\emph{Proof of Corollary \ref{Corollary-all-partitions}:} 
Suppose 	$(P_1, \dots, P_k)$ and $( Q_1, \dots, Q_\ell)$ satisfy the assumptions of Corollary \ref{Corollary-all-partitions}.
Also suppose Algorithm \ref{algo-nonstrict-SDP} chooses $A_1, \dots, A_k$ and $Y_1, \dots, Y_\ell$ in Step 
\ref{algo-nonstrict-SDP-choose-Ai-Yj} so with $X^*$ and $Z^*$ they satisfy conclusions
\eqref{Corollary-all-partitions-2} and \eqref{Corollary-all-partitions-3}.
Then Algorithm \ref{algo-nonstrict-SDP} will successfully construct the SDP so that 
 conclusion \eqref{Corollary-all-partitions-1} in Corollary \ref{Corollary-all-partitions} will hold.
\qed

\subsection{Guaranteeing Slater's condition}  
\label{subsection-generate-Slater-nonstrict-SDP} 

In this subsection we modify the arguments in Subsection \ref{subsection-generate-nonstrict-SDP-nonSlater} to generate SDPs of the form  \eqref{p}
 that fail strict complementarity with \eqref{d},  
 while ensuring that both \eqref{p} and \eqref{d} are strictly feasible.
 
 \co{
 The output of Algorithm \ref{algo-nonstrict-Slater-SDP} consists of the data of \eqref{p}, together with 
 optimal solutions $X^*$ and $Z^*$ of \eqref{p} and 
 \eqref{d}, respectively, having the form 
 \eqref{eqn-Xstar-Zstar}.  The output also includes $\bX$ strictly feasible in 
 \eqref{p}$,$ and $\bZ$ strictly feasible in \eqref{d}.
}
 
 \co{
 
 The output will be the data of \eqref{p}, together with  $X^*$ and $Z^*$ optimal solutions 
 in \eqref{p} and \eqref{d}, respectively of the form 
 \eqref{eqn-Xstar-Zstar}, as well as $\bX$ strictly feasible in \eqref{p} and 
 $\bZ$ strictly feasible in \eqref{d}. 
}

Suppose that \eqref{p} has the form given in Theorem \ref{thm-main},
$X^*$ is optimal in \eqref{p} and $Z^*$ in \eqref{d}. Also suppose that 
$\bX$ strictly feasible in \eqref{p} and 
$\bZ$ strictly feasible in \eqref{d}.

Then 
		\begin{equation} \label{sosickofthis} 
	\la \bX - X^*, \bZ - Z^* \ra = 0.
	\end{equation}
Also, setting $\DX := \bX - X^*, \, \DZ := \bZ - Z^*, $
	the orthogonality equations must hold: 
	\begin{equation} 	\label{eqn-orth} \tag{$\mathrm{ORTH}$} 
	\la A_i, \DX \ra = \la Y_j, \DZ \ra = \, 0 \, \text{for} \, \, i=1, \dots, k; j=1, \dots, \ell.
\end{equation} 
And we also must have 
\begin{equation} \label{eqn-DeltaZ-in-A1-Am}
	\Delta Z \in \lin \{ A_1, \dots, A_m \}.
\end{equation}

In this subsection we use the $"\Diag"$ notation as follows. For a matrix 
$M \in \rad{n \times n}$ we write $\Diag(M)$ for the
the vector in $\radn$ that contains the diagonal elements of $M.$

Our generating algorithm is stated in Algorithm \ref{algo-nonstrict-Slater-SDP}.
It is intriguing that in Step \ref{step:one}\ref{step:one:a} 
we must choose $k$ to be smaller than $m$ in order to produce an SDP with the required features.

	\begin{algorithm}
	\caption{Construct SDP that fails strict complementarity, but both primal and dual satisfy Slater's condition} 
	\label{algo-nonstrict-Slater-SDP}
	\begin{algorithmic}[1]
		\Statex{\bf Input:} Integers $n \geq 3, \;  m \geq 2.$
		\Statex{\bf Output:} 
		\begin{itemize}
			\item SDP in the normal form of \eqref{p}
			which fails strict complementarity with \eqref{d}, 
			 and $Y_1, \dots, Y_\ell$ such that
			 \begin{itemize}
			 	\item \eqref{p} and $Y_1, \dots, Y_\ell$
			 	satisfy   the conclusions of Theorem \ref{thm-main}.
			 	\item  Both \eqref{p} and \eqref{d} satisfy Slater's condition. 
			 \end{itemize} 
			\item A maximum rank  optimal solution $X^*$ of \eqref{p} and
			$Z^*$ of \eqref{d} of the form  \eqref{eqn-Xstar-Zstar}.
			\item A positive definite matrix $\bX$ feasible in \eqref{p}, and a  
			positive definite matrix $\bZ$ feasible in \eqref{d}. 
		\end{itemize}
		\State   \label{step:one}Choose 
\begin{enumerate}[label=(\alph*)]
	\item  \label{step:one:a} positive integers $r$ and $s \,$ such that $r + s < n, \, $ $ k, \ell \in \{1, \dots, n-r-s \}, $ and $k < m.$ 
	\item  \label{step:one:b} $\bX, \bZ$ positive definite matrices. 
\end{enumerate}

	\State \label{step:two}Choose $X^*$ and $Z^*$ 	of the form  \eqref{eqn-Xstar-Zstar}, 
		 where $\Lambda$ and $\Gamma$ are positive definite, and of order $r$ and  $s, \,$ respectively. Make sure they satisfy \eqref{sosickofthis}.
			Define
		\begin{equation} \nonumber 
			\Delta X := \bX - X^*, \, \Delta Z := \bZ - Z^*.
		\end{equation}
		\State \label{algo-nonstrict-Slater-SDP-choose-Ai-Yj}Choose $A_1, \dots, A_k$ and $Y_1, \dots, Y_\ell$ such that 
		\begin{itemize} 
			\item  $(Z^*, A_1, \dots, A_{k})$ is a regular facial reduction sequence which reduces $\psd{n}$ to 
			$\bpx 0 & 0 \\ 0 & \oplus_r \epx. $ 
			\item  $(X^*, Y_1, \dots, Y_{\ell})$ is a regular facial reduction sequence which reduces 
			$\psd{n}$ to 
			$\bpx \oplus_s  & 0 \\ 0 & 0 \epx. $
		\end{itemize} 
		 	\State \label{step:three}This step ensures \eqref{eqn-base} and  \eqref{eqn-orth}:
		 			 	 \begin{enumerate}[label=(\alph*)]
		 	 		 	 	\item \label{step:three-b} Use Algorithm \ref{algo-base-eqn} to make sure 
		 	 	the  \eqref{eqn-base} equations are satisfied.
		 	 		\item 	\label{step:three:a} 	 Adjust $A_i(P_0)$ for all $i$ and $Y_j(Q_0)$ for all $j$  so afterwards  the \eqref{eqn-orth} equations hold.
		 	  \end{enumerate} 
		 	 	\State  \label{algo-nonstrict-Slater-SDP-extend}Let $A_{k+1} := \Delta Z.$ Then 
choose $A_{k+2},\dots,A_m$ so  they all  have zero inner product 
with $\Delta X, Y_1,\dots,Y_{\ell}. $ 
	 	 	\State \label{algo-nonstrict-Slater-SDP-set-bi}Set $b_i:= \la A_i, X^* \ra$ for $i \in \{1, \dots, m\}, \, $  $C := Z^* + \sum_{i=1}^m y_i A_i$ for some $y \in \rad{m}.$ 
		\end{algorithmic}
	\end{algorithm}

\begin{Example} \label{example-sturm-4} (Examples \ref{example-sturm}, \ref{example-sturm-2}, \ref{example-sturm-3} continued) 
	This example shows how Algorithm \ref{algo-nonstrict-Slater-SDP} constructs a variant 
	of the SDP \eqref{problem-sturm} (also stated in \eqref{problem-sturm-2}), 
	which satisfies Slater's condition both on the primal and dual side.
	
We can see  that even the original SDP \eqref{problem-sturm} satisfies Slater's condition, but the identity $I$ is not feasible in it. However, the dual does not satisfy Slater's condition. 
	For better intuition, we next construct a variant, in which 
	$\bX=I$ is strictly feasible in the primal, and $\bZ=I$ is strictly feasible in the dual.
	
	So suppose that in step \ref{step:one}\ref{step:one:b} Algorithm
	\ref{algo-nonstrict-Slater-SDP} chooses $\bX = \bZ = I.$ 
	Thus, when   in step \ref{step:two} it chooses 
	$X^*$ and $Z^*, \,$ it chooses their nonzero element to be $2$ to make sure 
	\eqref{sosickofthis} holds. 
	Hence  after step \ref{step:two} we  have 
	\begin{equation} \label{eqn-Slater-X*-Z*} 
		X^* = \begin{pmatrix}
		0 & 0 & 0 & 0\\
		0 & 0 & 0 & 0\\
		0 & 0 & 0 & 0\\
		0 & 0 & 0 & 2
	\end{pmatrix}, \, 
		Z^* = \begin{pmatrix}
		2 & 0 & 0 & 0\\
		0 & 0 & 0 & 0\\
		0 & 0 & 0 & 0\\
		0 & 0 & 0 & 0 
	\end{pmatrix},	
	\DX = \begin{pmatrix}
		1 & 0 & 0 & 0 \\
		0 & 1 & 0 & 0 \\
		0 & 0 & 1 & 0 \\
		0 & 0 & 0 & -1
	\end{pmatrix}, \, \DZ = \begin{pmatrix}
		-1 & 0 & 0 & 0 \\
		0 & 1 & 0 & 0 \\
		0 & 0 & 1 & 0 \\
		0 & 0 & 0 & 1
	\end{pmatrix}.
\end{equation} 
	After steps \ref{algo-nonstrict-Slater-SDP-choose-Ai-Yj} and ~\ref{step:three}\ref{step:three-b} we have (just like in Algorithm \ref{algo-nonstrict-SDP})
\begin{equation}
		\begin{aligned}
		A_1 &= \begin{pmatrix}
			0 & 0 & 0 & -1 \\
			0 & 2 & 0 & 0 \\
			0 & 0 & 0 & 0 \\
			-1 & 0 & 0 & 0
		\end{pmatrix},
		\quad
		A_2 = \begin{pmatrix}
			0 & 0 & 0 & 0 \\
			0 & 0 & 0 & -1 \\
			0 & 0 & 2 & 0 \\
			0 & -1 & 0 & 0
		\end{pmatrix}, \\ \\
		Y_1 &= \begin{pmatrix}
			0 & 0 & 0 & 0  \\
			0 & 0 & 0 & 1 \\
			0 & 0 & 1 & 0 \\
			0 & 1 & 0 & 0
		\end{pmatrix},
		\quad
		Y_2 = \begin{pmatrix}
			0 & 0 & 0 & 1 \\
			0 & 1 & 0 & 0 \\
			0 & 0 & 0 & 0 \\
			1 & 0 & 0 & 0
		\end{pmatrix}.
	\end{aligned} 
	\end{equation}
	Step \ref{step:three}\ref{step:three:a} ensures the \eqref{eqn-orth} equations by 
	setting the 
	$(1,1)$ element of both $A_i$ and the $(4,4)$ element of both $Y_j$ to obtain
	\begin{equation} \label{eqn-Slater-all-Ai-Yj} 
		\begin{aligned}
			A_1 &= \begin{pmatrix}
				-2 & 0 & 0 & -1 \\
				0 & 2 & 0 & 0 \\
				0 & 0 & 0 & 0 \\
				-1 & 0 & 0 & 0
			\end{pmatrix},
			\quad
			A_2 = \begin{pmatrix}
				2 & 0 & 0 & 0 \\
				0 & 0 & 0 & -1 \\
				0 & 0 & 2 & 0 \\
				0 & -1 & 0 & 0
			\end{pmatrix}, \\ \\
			Y_1 &= \begin{pmatrix}
				0 & 0 & 0 & 0  \\
				0 & 0 & 0 & 1 \\
				0 & 0 & 1 & 0 \\
				0 & 1 & 0 & -1 
			\end{pmatrix},
			\quad
			Y_2 = \begin{pmatrix}
				0 & 0 & 0 & 1 \\
				0 & 1 & 0 & 0 \\
				0 & 0 & 0 & 0 \\
				1 & 0 & 0 & -1 
			\end{pmatrix}.
		\end{aligned} 
	\end{equation}
	Step \ref{algo-nonstrict-Slater-SDP-extend}  chooses $A_3  := \DZ.$ 
	Step \ref{algo-nonstrict-Slater-SDP-set-bi} may choose $y=0.$ 
	So the final SDP has the data given in \eqref{eqn-Slater-X*-Z*}, \eqref{eqn-Slater-all-Ai-Yj}, and 
	\begin{equation}
		\begin{aligned}
		A_3 = \begin{pmatrix}
				-1 & 0 & 0 & 0 \\
				0 & 1 & 0 & 0 \\
				0 & 0 & 1 & 0 \\
				0 & 0 & 0 & 1
			\end{pmatrix}, 
			\quad
			b = (0,0,2)^\top, C &= \begin{pmatrix}
				1 & 0 & 0 & 0 \\
				0 & 0 & 0 & 0 \\
				0 & 0 & 0 & 0 \\
				0 & 0 & 0 & 0
			\end{pmatrix}.
		\end{aligned}
	\end{equation}
\end{Example}

\begin{Theorem} \label{thm-alg-nonstrict-Slater-correct} The following hold:

\begin{enumerate}

\item 	\label{thm-alg-nonstrict-Slater-correct-1}  
Algorithm \ref{algo-nonstrict-SDP} can be implemented, using Algorithm \ref{algo-base-eqn} as a subroutine, so it i) succeeds in every run, and ii) its output is as stated. 
\co{correctly outputs an SDP in the form of \eqref{p}.
The output \eqref{p} has the following properties:  it  fails strict complementarity with \eqref{d};  and 
both \eqref{p} and \eqref{d} satisfy Slater's condition.
}

\item \label{thm-alg-nonstrict-Slater-correct-2}   Suppose \eqref{p} is an SDP with $m$ constraints over $\psdn$ which fails 
strict complementarity, but both \eqref{p} and \eqref{d} 
satisfy  Slater's condition.  
Then \eqref{p} is obtained 
as  a reformulation of some output of 	Algorithm \ref{algo-nonstrict-Slater-SDP}.

\end{enumerate}
\end{Theorem}
{\bf Proof} We first prove \eqref{thm-alg-nonstrict-Slater-correct-1}. 
We will need some other subroutines besides Algorithm \ref{algo-base-eqn}. Since these are quite simple, we just describe  them embedded in the proof.

We show how to implement Step  \ref{step:two} so \eqref{sosickofthis}  
holds. Since we want  $X^* $ and $Z^*$ to be in the form given in 
\eqref{eqn-Xstar-Zstar}, \eqref{sosickofthis} 
amounts to 
\begin{equation} \label{eqn-Xstar-Xbar-Zstar-Zbar-3} 
	\la \bZ, X^* \ra + \la \bX, Z^* \ra  =  \la \bX, \bZ \ra.
\end{equation}
In this equation the right hand side is positive. 
So, to satisfy it, we i) sequentially set  the first $r+s-1$ nonzero diagonal elements of 
$X^*$ and $Z^*$ so the right hand side is not exceeded, then ii) 
set the $(r+s)$th element, so the left and right hand sides are equal. 
This algorithm produces any possible $X^*$ and $Z^*$ which satisfy 
\eqref{eqn-Xstar-Xbar-Zstar-Zbar-3}. 

Steps   \ref{algo-nonstrict-Slater-SDP-choose-Ai-Yj} 
and \ref{step:three}\ref{step:three-b} are carried out as in Algorithm 
\ref{algo-nonstrict-SDP}. 

Next we show how to carry out Step \ref{step:three}\ref{step:three:a}. Let $(P_0, P_1, \dots, P_k)$ the 
structure of $(Z^*, A_1, \dots, A_k).$ 
Let us fix $i \in \{1, \dots, k \}.$ We will show how to ensure 
\begin{equation} \label{eqn-AiDeltaX=0} 
	\la A_i, \Delta X \ra = 0. 
\end{equation} 
For that, we observe that i) since all nonzero entries  of $X^*$ are in 
$[n] \setminus P_{0:k}, \, $ 
and $\bX$ is positive definite, the diagonal elements of $\Delta X(P_{0})$ are positive; and 
ii) the 
diagonal of $A_i(P_{0})$ can be nonzero.

Thus, to ensure \eqref{eqn-AiDeltaX=0}, we first set  
the diagonal of $A_i(P_{0})$  to zero and let 
$$
d := -  \la A_i, \Delta X \ra.   
$$
Then we solve  the equation
\begin{equation} \label{seinfeld} 
\Diag(\Delta X(P_{0}))^\top y = d,	
\end{equation}
for $y$ and set the diagonal of $A_i(P_{0})$ to $y.$ 
We solve \eqref{seinfeld} as follows: we first choose  the first $|P_{0}|-1$ 
elements of $y$ arbitrarily, then we choose the last element  to satisfy 
\eqref{seinfeld}. Since $Y_j(P_0)=0$ for all $j \in \{1, \dots, \ell \}, \,$ 
this operation keeps the \eqref{eqn-base} equations satisfied.
 We similarly enforce  $\la Y_j, \DZ \ra = 0$ for all $j$ by setting the 
diagonal elements of $Y_j(Q_0).$  This  last operation also leaves the 
\eqref{eqn-base} equations satisfied, since $A_i(Q_0)=0$ for all $i.$ 

Carrying out Step \ref{algo-nonstrict-Slater-SDP-extend} is straightforward. 

Suppose Algorithm \ref{algo-nonstrict-Slater-SDP} constructed \eqref{p}.
The proof that \eqref{p} fails strict complementarity with its dual is essentially the same as the proof for Algorithm \ref{algo-nonstrict-SDP} in Theorem \ref{thm-alg-nonstrict-Slater-correct}. 
Since the $A_i$ all have zero inner product with 
$\Delta X, \,$ by the choice of the $b_i$ in Step 
\ref{algo-nonstrict-Slater-SDP-set-bi} we deduce that $\bX$ is feasible in \eqref{p}.
Further, since \eqref{eqn-DeltaZ-in-A1-Am} holds, by the choice of 
$C$ in Step 
\ref{algo-nonstrict-Slater-SDP-set-bi} we deduce that $\bZ$ is feasible in \eqref{d}.

We next prove \eqref{thm-alg-nonstrict-Slater-correct-2}.   
For that, suppose \eqref{p} is an 
SDP that lacks strict complementarity, but both \eqref{p} and \eqref{d} satisfy Slater's condition.
Let $X^*$ be a maximum rank optimal solution in \eqref{p} and $Z^*$ be such a solution in \eqref{d}. Also suppose  $\bX$ is strictly feasible in 
\eqref{p} and $\bZ$ in \eqref{d}. Just like in the proof of Theorem \ref{thm-alg-nonstrict-correct},
we rotate $X^*, \, Z^*, \,$ and \eqref{p} to ensure that $X^*$ and $Z^*$ are in the form 
\eqref{eqn-Xstar-Zstar}. These rotations keep $\bX$ and $\bZ$ positive definite.
Thus, given how we satisfied 
\eqref{sosickofthis}, the matrices $\bX, \, \bZ, X^*, \, Z^*$ are possible outputs of Steps 
\ref{step:one} and \ref{step:two}.

Thus \eqref{p} has a reformulation as described in Theorem \ref{thm-main}, with the accompanying
$Y_1, \dots, Y_\ell.$ Let 
$(P_0, P_1, \dots, P_k)$ be the structure of $(Z^*, A_1, \dots, A_k)$ and 
$(Q_0, Q_1, \dots, Q_\ell)$ be the structure of $(X^*, Y_1, \dots, Y_\ell).$ 
  
We next prove that $A_1, \dots, A_k$ and $Y_1, \dots, Y_\ell$ are  possible outputs of Steps 
\ref{algo-nonstrict-Slater-SDP-choose-Ai-Yj} and 
\ref{step:three}. We only do this for the $A_i$ since the proof for the $Y_j$ is analogous.
First, the blocks of all $A_i$ outside the diagonal of block 
 $A_i(P_{0})$ and 
$A_i(P_{i-1}, Q_0)$  can be chosen by Step  \ref{algo-nonstrict-Slater-SDP-choose-Ai-Yj}.
The blocks $A_i(P_{i-1}, Q_0)$  can be chosen by step \ref{step:three}\ref{step:three-b}.
And the diagonal elements of $A_i(P_{0})$  can be chosen  by step \ref{step:three}\ref{step:three:a}: 
for that, we  need that 
the algorithms described above can output any possible solution of the 
\eqref{eqn-orth} equations. 

To proceed, we claim 
\begin{equation} \label{eqn-DeltaZ-not-in-A1-Ak}
	\Delta Z \not \in \lin \, \{ A_1, \dots, A_k \}.
\end{equation}
Indeed, to prove \eqref{eqn-DeltaZ-not-in-A1-Ak}, define $P_{k+1} := [n] \setminus P_{1:k}.$ 
Then $\Delta Z(P_{k+1})$ is positive definite, hence it is nonzero; 
but $A_i(P_{k+1})=0$ for $i=1, \dots, k, \, $ so \eqref{eqn-DeltaZ-not-in-A1-Ak}
follows.

We then combine 
 \eqref{eqn-DeltaZ-not-in-A1-Ak} with \eqref{eqn-DeltaZ-in-A1-Am} and deduce 
 $k < m.$ We also deduce that when writing $\Delta Z$ as a linear combination of 
 all the $A_i, \,$ the coefficient of at least one of the $A_i$ for 
 $i \in \{ k+1, \dots, m \}$ is nonzero. 
 Thus, we can reformulate \eqref{p} by doing elementary row operations only on the last $m-k$ equations to ensure $A_{k+1}=\Delta Z.$ 
 
Thus, $A_{k+1}, \dots, A_m$ can be chosen by Step  \ref{algo-nonstrict-Slater-SDP-extend}; 
and $b$ and $C$ can be chosen by Step \ref{algo-nonstrict-SDP-set-bi}.
With this the proof is complete.
\qed 

As we remarked before, we typically would like to ensure that all the constructed $A_i$ are linearly independent.
This can be ensured for $A_1, \dots, A_k$ just by choosing their positive definite 
blocks to be nonempty. By \eqref{eqn-DeltaZ-not-in-A1-Ak}, and since $A_{k+1} = \Delta Z, \,$ 
also $A_1, \dots, A_k, A_{k+1}$ are linearly independent. 
Finally, we can choose $A_{k+2}, \dots, A_m$ 
so that all $A_i$ are linearly independent, as we discussed after 
the proof of Theorem \ref{thm-alg-nonstrict-Slater-correct}.

\section{The singularity degree}
\label{section-the-singularity-degree}

In this section we discuss the difficulty of the SDPs generated by Algorithms
\ref{algo-nonstrict-SDP} and \ref{algo-nonstrict-Slater-SDP}. We will use 
the concept of singularity degree introduced by Sturm \cite{Sturm:00}. Before we 
formally define it, we give some motivation. Thus, suppose 
Algorithm \ref{algo-nonstrict-SDP} generates an SDP in which $A_1, \dots, A_k$ each have just one nonzero, which is a $1, \,$ on the main diagonal. 
Intuitively it is clear that 
this SDP will not exhibit the pathology we saw in Example \ref{example-sturm}.
Also, from it we can create an SDP with just $k=1,$
 by aggregating the first $k$ equations.

To define the singularity degree, we need an  
 intermediate definition: 
\begin{Definition} \label{definition-RR-form} 
	We say that a semidefinite system 
	\begin{equation} \label{problem-Mi-X} 
		\ba{rcl} 
		\la M_i, X \ra & = & d_i \; (i=1, \dots, m) \\
		X  & \in  & \psdn
	\end{array} 
\end{equation}
is in rank revealing (RR) form, if the following hold for some $k \in \{0, \dots,  m \}$ and 
$r \in \{0, \dots,  n \}: $
\begin{enumerate}
	\item $(M_1, \dots, M_k)$ is a regular facial reduction sequence which reduces 
	$\psdn$ to 
	$$ \begin{pmatrix}
		0  & 0 \\
		0 & \oplus_r  \end{pmatrix}.$$
	\item $d_1 = \dots = d_k = 0.$
	\item there is an $X$ feasible in \eqref{problem-Mi-X} of the form 
	\begin{equation} 
		X^* = \bpx 0  & 0  \\
		0 & \Lambda  \epx,   
	\end{equation}
	where $\Lambda$ is order $r$ and positive definite.
\end{enumerate}
We also say that the first $k$ equations in \eqref{problem-Mi-X} 
certify the maximum rank of a feasible solution.
\end{Definition}

As it was shown in \cite{LiuPataki:15} 
every feasible semidefinite system can be 
reformulated  into RR form. Further, the RR form permits us to generate any semidefinite system
in which the rank of the maximum rank solution is fixed: for that we just need to create a system in RR form, then reformulate it. 

Definition \ref{definition-RR-form} naturally connects to Theorem \ref{thm-main} as follows.
Consider the feasible set of \eqref{p} with the equation $\la Z^*, X \ra = 0$ attached:
\begin{equation} \label{p-Zstar} 
\begin{array}{rcl}
	\la Z^*, X \ra & = & 0 \\
	\la A_i, X \ra & = & 0 \; \text{for} \, i=1, \dots, m  \\
	X  & \succeq & 0,
\end{array} \tag{$P_{Z^*}$} 
\end{equation} 
Theorem \ref{thm-main} shows that \eqref{p-Zstar} can be brought into RR form. Theorem \ref{thm-main} of course 
shows much 
more: it proves that after a reformulation a  regular facial sequence  certifies that $X^*$ has a maximum rank, 
{\em and} another such sequence {\em simultaneously}  certifies  that $Z^*$ has maximum rank.

The following definition of a singularity degree is slightly different from the original one in  \cite{Sturm:00}. 
But the two definitions are equivalent: Sturm's definition  in \cite{Sturm:00} did not use elementary row operations, but this makes no difference.
\begin{Definition}
The singularity degree of the feasible set of a semidefinite system  is 
the smallest $k$ so the system can be reformulated into an RR form in which the first $k$ equations certify the maximum rank of a feasible solution.
\end{Definition}
\co{Usually we will just refer to the singularity degree of \eqref{p}, rather than spelling out 
the singularity degree of its feasible set.} 
Thus, the singularity degree of a semidefinite system is zero, exactly when it satisfies Slater's condition.

The following are  known about the singularity degree of a semidefinite system. 
First, the definition 
directly implies that it does not change if we reformulate it. 
Thus, asking "what is the singularity degree of a semidefinite system?" is meaningful, even when 
it is already in RR  form. Indeed, as we discussed above, in some cases 
the  equations that certify the maximum rank in an RR form could be aggregated to just one equation.

However, most questions about the singularity degree seem difficult.
Given a semidefinite system, the complexity status of computing its singularity degree is unknown both in the Turing model and the real number model \cite{Cuckeretal:98}.
Next we look at the complexity status of the decision version of the problem, stated as
\begin{itemize}
\item[] "Given a semidefinite system, and an integer $k, \, $ is the singularity degree 
$\leq k$? 
\end{itemize}
The complexity is unknown in the Turing model, and the same is true when we replace $\leq$ by $\geq, $ or by $=.$ These problems may be NP-hard, or may not even be in NP.

In the real number model
we claim that the $\leq$ problem is in NP. Indeed, a "yes" certificate is 
a sequence of elementary row operations and a rotation that brings our system into RR form, in which the number of equations that certify the maximum rank is at most $k.$ However, even in this model, the complexity of the $\geq$ or $=$ problems are unknown.

To obtain good test problems it is important to generate SDPs 
in which the singularity degree of 
\eqref{p-Zstar} is large. Recall that Algorithms \ref{algo-nonstrict-SDP} and 
\ref{algo-nonstrict-Slater-SDP} first generate a minimal SDP containing the equations that certify the maximum rank of $X^*, \,$ then generate the other equations. 

Theorem \ref{thm-singularity-degree}  below characterizes when the singularity degree of a system is the largest possible, i.e., when it equals the number of equations.
For that, we need one more definition.

\begin{Definition}
Suppose $(Y_1, \dots, Y_t)$ is a regular facial reduction sequence in which all $Y_i$ are in $\symn$ with structure 
$(P_1, \dots, P_t)$  and set 
$$
P_{t+1} := [n] \setminus P_{1:t}.
$$
The critical block in  $Y_i$ for $i=2, \dots, t$ is  \footnote{Recall that 
	$Y_i(P_{i-1}, P_{i:t+1})$ stands for $Y_i(P_{i-1}, P_i \cup \dots P_{t+1}).$ }.
$$
Y_i(P_{i-1}, P_{i:t+1}) \;  
	$$
\end{Definition}

\begin{Example}
	(Example \ref{example-sturm-2} continued) 
	Consider the feasible set of this problem with the 
	equation  $\la Z^*, X \ra = 0$ attached: 
	$$
	\begin{array}{rcl}
		\la Z^*, X \ra & = & 0 \\
		\la A_i, X \ra & = & 0 \; (i=1,2) \\
		X  & \succeq & 0. 
	\end{array} 
	$$
	This system is in RR form and $(Z^*, A_1, A_2)$ is a regular facial reduction sequence, 
	with structure 
	$$
	(P_0, P_1, P_2)  =  ( \{ 1 \}, \,  \{ 2 \}, \, \{ 3 \}).
	$$
	The entries in the critical blocks in $A_1$ and $A_2$ are circled:
	\[
Z^* =
\begin{pmatrix}
	1 & 0 & 0 & 0 \\
	0 & 1 & 0 & 0 \\
	0 & 0 & 1 & 0 \\
	0 & 0 & 0 & 1
\end{pmatrix},
\qquad
A_1 =
\begin{pmatrix}
	0 & 0 & \circledentry{0} & \circledentry{-1} \\
	0 & 2 & 0 & 0 \\
	\circledentry{0} & 0 & 0 & 0 \\
	\circledentry{-1} & 0 & 0 & 0
\end{pmatrix},
\qquad
A_2 =
\begin{pmatrix}
	0 & 0 & 0 & 0 \\
	0 & 0 & 0 & \circledentry{-1} \\
	0 & 0 & 2 & 0 \\
	0 & \circledentry{-1} & 0 & 0
\end{pmatrix}.
\]

\end{Example}

\begin{Theorem} \label{thm-singularity-degree} 
Supose the semidefinite system \eqref{problem-Mi-X} is in RR form.
Then its singularity degree is $m$ $\Leftrightarrow$ the following hold:
\begin{enumerate}
	\item \label{thm-singularity-degree-1} All of the $m$ equations certify the maximum rank in it. In particular, 
	$d=0.$
	\item \label{thm-singularity-degree-2} All the critical blocks in 
	$M_2, \dots, M_m$ are nonzero.
\end{enumerate}
\end{Theorem}
\pf{} Let $(P_1, \dots, P_m)$ be the structure of $(M_1, \dots, M_m)$ and 
let $p_i = | P_i|$ for all $i.$ To better visualize the arguments that follow, we  
assume without loss of generality that 
$P_1 = \{1, \dots, p_1\}, \, P_2 = \{p_1 + 1, \dots, p_1 + p_2 \}, \,$ and so on. Assume without loss of generality 
that $P_i $ is nonempty for $i=1, \dots, m.$ Let us also define $P_{m+1} := [n] \setminus P_{1:m}.$  

We first prove $\Leftarrow.$
Suppose \eqref{thm-singularity-degree-1} and \eqref{thm-singularity-degree-2} hold.
Also suppose  \eqref{elaine} below 
is  a reformulation  of \eqref{problem-Mi-X}, which is in RR form: 
\begin{equation} \label{elaine} 
\ba{rcl} 
\la M_i^\prime, X \ra & = & 0 \; (i=1, \dots, m). \\
X  & \in  & \psdn.
\end{array} 
\end{equation}
We assume that in \eqref{elaine} for some $k \leq m$ the first $k$ equations certify the maximum rank solution. Hence 
$(M_1^\prime, \dots, M_k^\prime)$ is a regular facial reduction sequence.
Suppose $(M_1^\prime, \dots, M_k^\prime)$ has structure $(P_1^\prime, \dots, P_k^\prime).$ 
Let $p_i^\prime = | P_i^\prime|$ for  $i \in \{1, \dots, k \}.$ 
We  
assume without loss of generality that 
$P_1^\prime = \{1, \dots, p_1^\prime\}, \, P_2 = \{p_1^\prime + 1, \dots, p_1^\prime + p_2^\prime \}, \,$ and so on. 

We will show that 
\begin{equation} \label{eqn-Pi-Piprime} 
	P_i = P_i^\prime \, \text{for } \, i=1, \dots, k,
\end{equation}
and this will imply $k=m, \,$ completing the proof.

To prove \eqref{eqn-Pi-Piprime}, let $T$ be the product of all rotation matrices used in the reformulation process.
The definition of a reformulation implies that for some 
$y \in \rad{m}, \,$ and $Z \in \psdn$  we have 
\begin{equation} \label{eqn-M1prime-Z-T} 
M_1^\prime = T^\top Z T, \, Z = \sum_{i=1}^m y_i M_i.
\end{equation}
Since $M_1^\prime$ is nonzero and psd, so is $Z.$ 

Since $Z$ is a linear combination of the $M_i$, we have $Z(P_{m+1})=0. \,$  Thus, since 
$Z \succeq 0$ we deduce $Z({P_{m+1}, [n]})  = 0 .$ Since $M_m(P_{m-1}, P_{m+1}) \neq 0, \,$ this in turn implies  
$y_m=0.$ 
Hence the diagonal block $Z(P_{m:m+1})=0$ is also zero.  
Continuing in like  fashion, we deduce  that 
$y_2 = \dots = y_m = 0, $ hence $y_1 > 0.$ 
Hence the rank of $M_1^\prime$ is $p_1, $ so we deduce $P_1 = P_1^\prime.$ 

If $m=1, \,$ then we are done. Otherwise, let us partition $T$ as 
\begin{equation} \label{eqn-T} 
	T = \bpx T_{11} & T_{12}  \\ T_{21} & T_{22} \epx,
\end{equation}
where $T_{11} \in \sym{p_1}$ and $T_{22} \in \sym{n-p_1}.$ 
We claim that $T_{12}=0.$ Indeed, 
let $\Lambda$ be the upper left order $p_1$ minor of $M_1.$ 
Then 
$$
T^\top M_1 T = \begin{pmatrix}
	T_{11}^\top \Lambda T_{11} & 	T_{11}^\top \Lambda T_{12} \\
		T_{12}^\top \Lambda T_{11} & 	T_{12}^\top \Lambda T_{12} 
\end{pmatrix}.
$$
Since $	T_{11}^\top \Lambda T_{11}$ is invertible, so is $T_{11}.$ Since $T_{11}^\top \Lambda T_{12}=0, \,$ we deduce $T_{12}=0.$ 

For later use we record a basic fact. Suppose a matrix, say $M \in \symn$ is partitioned as 
$$
M = \bpx A & B \\ B^\top & C \end{pmatrix}
$$
where $A \in \sym{p_1}, \, C \in \sym{n-p_1}.$ Then by a straightforward calculation the lower right 
order $n - p_1$ minor of $T^\top M T$ is $T_{22}^\top C T_{22}.$ 

Let us delete the first $p_1$ rows and columns from all $M_i$ and all $M_i^\prime$ and 
for simplicity, still call the resulting matrices $M_i$ and $M_i^\prime.$ Then 
 $(M_1, \dots, M_m)$ and  $(M_1^\prime, \dots, M_k^\prime)$ are still regular facial reduction sequences.
 (The first element of both these sequences is zero, but that does not matter). By the preceding argument,
  \eqref{problem-Mi-X} (with the reduced $M_i$) can be reformulated into 
  system  \eqref{elaine} (with the reduced $M_i^\prime$), and the product of all rotation matrices in the reformulation process can be chosen to be $T_{22}.$ 
 
 Thus we proceed in a similar manner to arrive at \eqref{eqn-Pi-Piprime}.

\co{

we know that for some $Z \in \symn$ and $y \in \rad{m}$ we have 
$$
M_2^\prime = T^\top Z T, \, Z = \sum_{i=1}^m y_i M_i.
$$
(This $Z$ and $y$ is different from the one in \eqref{eqn-M1prime-Z-T}, we just use the same 
symbols to save on notation. In particular, now $Z$ does not have to be psd.)
Partition
\begin{equation} \label{eqn-Z} 
Z = \bpx Z_{11} & Z_{12}  \\ Z_{12}^\top & Z_{22} \epx,
\end{equation}
where $Z_{11} \in \sym{p_1}, \, Z_{22} \in \sym{n-p_1}.$  Since $T_{12}=0, $ 
the lower right order $n - p_1$ corner of $M_2^\prime$ equals
$T_{22}^\top Z_{22} T_{22}.$ So $Z_{22} \succeq 0$ and an analogous argument shows 
$Z = y_1 M_1 + y_2 M_2$ for some $y_2 > 0$ and $P_2 = P_2^\prime.$  

Continuing, we deduce that $P_i = P_i^\prime$ for $i=1, \dots, k$ 
This completes  the proof.
}

Next we prove $\Rightarrow.$ Assume the statement on the left hand side.
In this proof we use ideas from \cite{pataki2024exponential}. 
Let us assume that the singularity degree of \eqref{problem-Mi-X} is $m.$ 
Since the singularity degree of any SDP with $m$ equations is at most $m,$ 
item \eqref{thm-singularity-degree-1} follows directly. 
To complete the proof we will show that all critical blocks in the $M_i$ are nonzero.
We argue by contradiction, so suppose $i \in \{2, \dots, m\}$ and 
$M_i(P_{i-1}, P_{i+1:m+1}) \neq 0.$ 

To help with  the proof, we show $M_{i}$ and $M_{i-1}$ in equation \eqref{bla}.
The empty blocks are zero, and the $\ti$ blocks are arbitrary. We mark by 
$\otimes$ the critical block of $M_i.$ 
We will prove that this  block is nonzero. 

\vspace{.1cm}
\beq \label{bla}    
M_{i-1} \, = \, \begin{pmatrix}[c|c|c|c]
\bovermat{$P_{1:(i-2)}$}{\mbox{$\,\,\,\, \times \,\,\,\,$}} 	& \bovermat{$P_{i-1}$}{\mbox{$\,\, \times\,\,$}}	& \bovermat{$ P_{i}$}{\mbox{$\,\,\times\,\,$}}	& \bovermat{$  P_{(i+1):(m+1)}$}{\mbox{$\,\,\,\,\,\,\,\,\,\,\,\times\,\,\,\,\,\,\,\,\,\,\,\,$}}	\\ \hline 
\ti & + &  &    \\ \hline
\ti &  & &    \\ \hline
\ti & \pha{0} &  &  
\end{pmatrix}, \,  
M_{i}  \, = \, \begin{pmatrix}[c|c|c|c]
\bovermat{$P_{1:(i-2)}$}{\mbox{$\,\,\,\, \times \,\,\,\,$}} 	& \bovermat{$P_{i-1}$}{\mbox{$\,\, \times\,\,$}}	& \bovermat{$ P_{i}$}{\mbox{$\,\,\times\,\,$}}	& \bovermat{$P_{(i+1):(m+1)}$ }{\mbox{$\,\,\,\,\,\,\,\,\,\,\,\,\,\,\,\times\,\,\,\,\,\,\,\,\,$}}	\\ \hline 
\ti & \ti & \ti & \,\,\,\,\,\otimes  \\ \hline
\ti & \ti & + &    \\ \hline
\ti & \otimes &  &  
\end{pmatrix}. 
\eeq 
For that, suppose the $\otimes$ blocks are  zero and let 
$M_{i-1}^\prime := \lambda M_{i-1} + M_{i}$ for some large $\lambda >0.$ Then by the Schur complement condition for positive definiteness, 
we find that $M_{i-1}^\prime(P_{i-1:i})$ is positive definite. 

So after a suitable rotation, we see that $(M_1, \dots, M_{i-2}, M_{i-1}^\prime, M_{i+1}, \dots, M_m)$ is a shorter regular facial reduction sequence. This sequence 
yields a reformulation 
with $m-1$ equations that certify the maximum rank of a feasible solution 
in \eqref{problem-Mi-X}. This contradiction completes the proof.
\qed  

How do we use Theorem \ref{thm-singularity-degree}? 
Recall that 
Algorithm \ref{algo-nonstrict-SDP} constructs an a semidefinite system of the form \eqref{p-Zstar}, 
\co{\begin{equation} \label{problem-generated-by-Alg1}
\begin{array}{rcl}
	\la Z^*,  X \ra & = & 0 \\
	\A X & = & b \\
	X & \succeq & 0,
\end{array}
\end{equation}}
in which the first $k+1$ equations certify that $X^*$ has maximum rank. 
Hence the singularity degree  is at most $k+1.$ 
 Theorem \ref{thm-singularity-degree} permits us to check that the singularity degree of the subsystem made up of the first $k+1$ equations is exactly $k+1:$ for that, we just need to verify that the critical blocks are all nonzero.
  
 The singularity degree of \eqref{p-Zstar} may end up being strictly less than $k+1, \,$ due to the other equations. Nevertheless, our computational study in Section 
 \ref{section-computational-results} will show that the difficulty of the generated SDPs increases with $k.$ 
 
 \co{
 We can similarly define the singularity degree of a semidefinite system in the form of \eqref{d}.
 We sketch the arguments for that next. We  rewrite this system in the  form \eqref{re-d} that constrains $Z$ by equations, i.e., in the primal form. The details are given in Subsection  \ref{subsection-proof-of-theorem-1-lemmas}.
Then we define the singularity degree of \eqref{d} as the singularity degree of the  system rewritten in the primal form.
 
  Following the arguments there, 
 this system can be rotated so the equations  $\la Y_j, Z \ra = 0$ for $j=1, \dots, \ell$ appear in it.
}
 
 \co{ It is desirable for this subsystem to have large singularity degree; and to achieve that, we only need to ensure that the critical blocks in the $Y_j$ are nonzero.
 We leave the details to the reader. 
}
 
 We can similarly define the singularity degree of \eqref{d} with the equation $\la X^*, Z \ra = 0$ attached.
We sketch the arguments for that next. We  rewrite this system in a form that constrains $Z$ by equations, i.e., in the primal form, using arguments  from Subsection  \ref{subsection-proof-of-theorem-1-lemmas}. It is desirable for this system to have large singularity degree.

Based on our current knowledge, this is impossible to achieve. However, we can achieve large singularity degree of a subsystem, as follows.
Following the arguments in  Subsection  \ref{subsection-proof-of-theorem-1-lemmas}, 
 this system can be rotated into RR form, so the equations 
 \begin{equation} \label{eqn-Xstar-Yj} 
 	\ba{rcl} 
 	\la X^*, Z \ra & = & 0 \\
 	\la Y_j, Z \ra & = & 0 \,\, \text{for} \, j=1, \dots, \ell
 	\ena
 	 \end{equation}
 	 in it certify the maximum rank of $Z^*.$ 
 We can ensure that the subsystem  \eqref{eqn-Xstar-Yj} has maximum singularity degree by 
 choosing nonzero critical blocks in all of the $Y_j.$ 

\section{Computational results}
\label{section-computational-results}

\co{
Settings for interval:
\begin{itemize}
	\item spread of svecA and svecY $\leq 10^6.$
	\item Condition number of svec A $\leq 10.$
	\item Condition number of svecY $\leq 1000.$ 
\end{itemize}

Settings for Sturm:
\begin{itemize}
	\item spread of svecA and svecY $\leq 10^6.$
	\item Condition number of svec A $\leq 20.$
	\item Condition number of svecY $\leq 1000.$ 
\end{itemize}
}

We implemented Algorithms \ref{algo-nonstrict-SDP} and \ref{algo-nonstrict-Slater-SDP} in Matlab, and ran 
extensive computational experiments. We first generated SDPs in the form of \eqref{p}
which are in the normal form described in Theorem \ref{thm-main}. Then we  rotated \eqref{p} by an orthonormal matrix $T.$ 
While Theorem \ref{thm-main} applies rotations by matrices, which are merely invertible, 
for the purpose of computing the  forward errors it is relevant to choose $T$ as orthonormal: see Subsection 
\ref{subsection-computational-results-backward-and-forward-errors}. 

\subsection{Controlling the condition number and spread}
\label{subsection-computational-results-condition-number} 

To describe our criteria of what makes a good SDP instance, we define (with some abuse of notation)
$$
\svec (\A) := \bpx \svec A_1 \\ \dots \\ \svec A_m \epx, \, \svec (Y) := \bpx \svec Y_1 \\ \dots \\ \svec Y_\ell \epx. 
$$
For a matrix $M$ we define its {\em spread} as 
$$
\spread(M) := \dfrac{\max \, \{  |m_{ij}| \} }{\min \{ |m_{ij}|: \, m_{ij} \neq 0 \} }.
$$  
We first generated instances which are in the normal form given in Theorem \ref{thm-main} with the accompanying $Y_j.$ 
Then we  rotated \eqref{p} by a random 
$T$ matrix with orthonormal columns. To keep the inner product $\la A_i, Y_j \ra$ the same, we also rotated the $Y_j$ by the same 
$T.$ 

Our goal was to generate instances in which (after the rotation) 
\begin{enumerate}
	\item \label{criterion1} The condition number of $\svec(\A)$ is small;
	\item \label{criterion2} The spread of  $\svec(\A)$ is moderate; and 
	\item \label{criterion3} The spread of  of $\svec(Y)$ is moderate. 
	\end{enumerate} 
The $Y_j$ actually do not even appear in the SDPs. Our reason to satisfy criterion \eqref{criterion3} is that we believe that 
a future  algorithm that computes a solution with small forward error may actually have to construct the $Y_j.$ 

We first discuss the relevance of these criteria.
The role of the condition number is well documented in the numerical analysis literature; see, for example, the excellent books \cite{trefethen2022numerical, overton2001numerical}.
The spread seems to have received less attention, and in fact we coined this terminology ourselves.
Optimization problems with large spread are often said to be ``badly scaled''; see, e.g., \cite{klotz2014identification} for a discussion of how bad
 scaling can negatively affect the accuracy of computed solutions in linear programming.

\co{
We first discuss the importance of these criteria. 
The importance of the condition number is well documented in the numerical analysis literature:
see for example, the excellent books \cite{trefethen2022numerical, overton2001numerical}.
The  spread seems to have received less attention; in fact, we coined this terminology ourselves. 
An optimization problem with large spread is often said to be "badly scaled:"
 see e.g. \cite{klotz2014identification} for discussions 
on how  large spread can negatively impact accuracy of computed solutions in linear programming.
}

To achieve criterion \eqref{criterion1}, we scaled each $A_i$ to make its Frobenius norm equal to $1.$

Next we discuss how we achieved criterion \eqref{criterion2}. Suppose Algorithm \ref{algo-nonstrict-SDP} in Step \ref{algo-nonstrict-SDP-choose-Ai-Yj} chose  
$A_1, \dots, A_k.$ 
In 
Step \ref{algo-nonstrict-SDP-extend} we  first generated {\em all} $A_i$ which were  orthogonal to 
$A_1, \dots, A_k$ and all $Y_j.$ Then we discarded all $A_i$ for which the 
spread of $T^\top A_i T$ was large. Finally, from the $A_i$ that were left, we randomly chose 
$A_{k+1}, \dots, A_m.$ This method is heuristic, since it does not 
even look at the spread of $\svec(\A)$ but it worked well.

Further, in all  the instances we tested, we set 
\begin{itemize}
	\item   an upper bound of $20$ on the condition number of $\svec(\A);$ and  
	\item  an upper bound of $10^6$ on the spread of both $\svec(\A)$ and $\svec(Y).$ 
\end{itemize}
If after all our precautions, these criteria were not met, then 
we just discarded  the instance we generated.

 To solve the SDPs we used SDPT3 \cite{toh1999sdpt3} as solver, with the Nesterov-Todd search direction.
We called SDPT3 from the YALMIP interface \cite{lofberg2004yalmip}. 

\subsection{How to compute the backward and forward errors?}
\label{subsection-computational-results-backward-and-forward-errors} 

\co{
LAPACK ERROR COMPUTING: 
\begin{verbatim*}
https://www.netlib.org/lapack/lug/node75.html
\end{verbatim*}
}
In this subsection we first discuss how we computed the backward errors.
The backward errors are also known as the DIMACS errors for SDPs.
They form a vector of six elements, and measure  (i) 
primal and dual violation of equality and semidefinite constraints; (ii)  the duality gap; and (iii) the 
measure of how far $X$ and $Z$ from satisfying the complementarity equation 
$\la X, Z \ra = 0.$ For a detailed description, we refer to \cite{mittelmann2003independent}.
Thus the DIMACS errors form a vector of six elements, and we report the infinity norm of this vector.

We also report on the $5$th DIMACS error separately, which is defined as
$$
\frac{\la C, X \ra - \la b, y \ra}{1+|\la C, X \ra| +|\la b, y \ra|}. 
$$
Occasionally this error turns out to be  negative. Although we have not seen a theoretical analysis of this situation, 
in such a case the solution pair may be quite inaccurate.
So, to be on the safe side, we report the most negative of these errors for both families of instances we tested.

Next we discuss how we computed the forward errors.
Suppose we generated an instance \eqref{p}, then rotated it by an orthonormal matrix $T$ to obtain the instance 
that we call 
$(P_T).$ Let $(D_T)$ denote the dual of $(P_T).$

First suppose $X$  is an approximate optimal solution of $(P_T).$ 
Proposition \ref{prop-reform-invariance} implies that for $Y \in \symn$ we have 
\begin{equation} \label{eqn-eqn-optimal-P-PT} 
	Y \, \text{is optimal in } \, (P_T) \; \Leftrightarrow \; T Y T^\top \, \text{is optimal in } \eqref{p}.
\end{equation}
Hence the distance of $X$ to the optimal set of $(P_T)$ is computed as 
\begin{equation} \label{eqn-compute-forward-error-P}
	\begin{array}{rcl} 
		\inf \, \{ \, \norm{X-Y} \, : \, Y \, \text{is optimal in } (P_T) \, \} & = & \inf \, \{ \, \norm{T(X-Y) T^\top} \, : \, Y \, \text{is optimal in } (P_T) \, \} \\
		& = & \inf \, \{ \, \norm{T X T^\top - TYT^\top } \, : \, Y \, \text{is optimal in } (P_T) \, \} \\
		& = & \inf \, \{ \, \norm{T X T^\top - Y^\prime } \, : \, Y^\prime  \, \text{is optimal in } \eqref{p} \,  \}.
	\end{array}
\end{equation}

Here the norm corresponds to the Frobenius norm of matrices.\co{ The first equation is by definition, the second follows 
since rotating by an orthonormal  $T$ preserves the Frobenius norm, the third is obvious, and the last is from 
\eqref{eqn-eqn-optimal-P-PT} (with $Y$ in place of $X.$)
}
The first equation follows 
since rotating by an orthonormal  $T$ preserves the Frobenius norm, the second  is obvious, and the last is from 
\eqref{eqn-eqn-optimal-P-PT} (with $Y$ in place of $X.$)

The last quantity in \eqref{eqn-compute-forward-error-P} is still impossible to compute exactly, but it is easy to bound from below as we next show.
In any optimal solution of \eqref{p} only the lower right order $r$ block can be nonzero.
Thus, we obtain a lower bound by setting the lower  right order $r$ block of $T X T^\top$ to zero, and 
compute the Frobenius norm of the resulting matrix.

Next suppose $Z \,$ is an approximate an optimal solution in $(D_T).$ 
Part \eqref{prop-reform-invariance-d} in  Proposition \ref{prop-reform-invariance} implies 
\begin{equation} \label{eqn-eqn-optimal-P-PT} 
	Z \, \text{is optimal in } \, (D_T) \; \Leftrightarrow \; T^{- \top} Z T^{-1} =  T Z T^{\top} \, \text{is optimal in } \eqref{d},
\end{equation}
where the equality comes from $T^{-1}=T^\top.$ 
Hence the distance of $Z$ to the optimal set of $(D_T)$ is computed as 
\begin{equation} \label{eqn-compute-forward-error-D}
	\begin{array}{rcl} 
		\inf \, \{ \, \norm{Z-Y} \, : \, Y \, \text{is optimal in } (D_T) \, \} & = &  \inf \, \{ \, \norm{T Z T^\top - Y^\prime } \, : \, Y^\prime  \, \text{is optimal in } \eqref{d} \,  \}.
	\end{array}
\end{equation}
The argument to show \eqref{eqn-compute-forward-error-D} is analogous to the one carried out in \eqref{eqn-compute-forward-error-P}.
For brevity, we omitted the intermediate steps.

To lower bound the quantity in \eqref{eqn-compute-forward-error-D}, we set the upper left order $s$ minor of 
$T Z T^\top$ to zero, and computed the Frobenius norm of the resulting matrix.

\subsection{Interval  partition structure} 

\label{subsection-computational-results-random}

In our first problem set we chose $X^*$ and $Z^*$ as in \eqref{eqn-Xstar-Zstar}, both ranks being two.
We then set 
$$
P_0 := \{1,2 \}, \, G := \{ 3,4, \dots, n-3, n- 2 \}, \, Q_0 := \{n-1, n \},
$$
then selected $k \in \{1, 2, \dots, |G| \}.$ We then chose $P_1, \dots, P_k$ to be an interval partition of $G:$ that is,
$P_1$ comprised the first few elements of $G; $ $P_2$ the next few; and so on.
Then we chose $A_i$ for $i=1, \dots, k$ so $(Z^*, A_1, \dots, A_k)$ was a regular facial reduction sequence with 
 structure $( P_0, P_1, \dots, P_k).$ 

To define the $Y_j$ we  set $\ell := k$ and set 
$$
Q_1 := P_k, Q_2 := P_{k-1}, \dots, Q_k := P_1.
$$
Then  we choose $Y_j$ for $j=1, \dots, k$ so $(X^*, Y_1, \dots, Y_k)$ was a regular facial reduction sequence with 
structure $( Q_0, Q_1, \dots, Q_k).$ 

For $i=1, \dots, k$ we  generated the diagonal block of all the $A_i$  to have entries in $(0,4].$ 
We also generated the dense block of the $A_i, \,$ i.e., the nonzero offdiagonal block   to have entries in $[-2,2].$  
We generated all the $Y_j$ in an analogous manner.

Thus, even before the rotation, the offdiagonal blocks of the $A_i$ and $Y_j$ were dense; and $A_{k+1},\allowbreak A_{k+2},\allowbreak \ldots,\allowbreak A_m$
 were usually 
completely
dense. See Figure \ref{fig:sparsity-interval} for 
the sparsity structure of the $A_i$ and $Y_j$ when $n=10, m=5, \, $ and $k = \ell=3.$ 

See Figure \ref{fig:sparsity-interval} for 
the sparsity structure of the $A_i$ and $Y_j$ when $n=10, m=5, \, $ and $k = \ell=3.$ 
\begin{figure}[h]
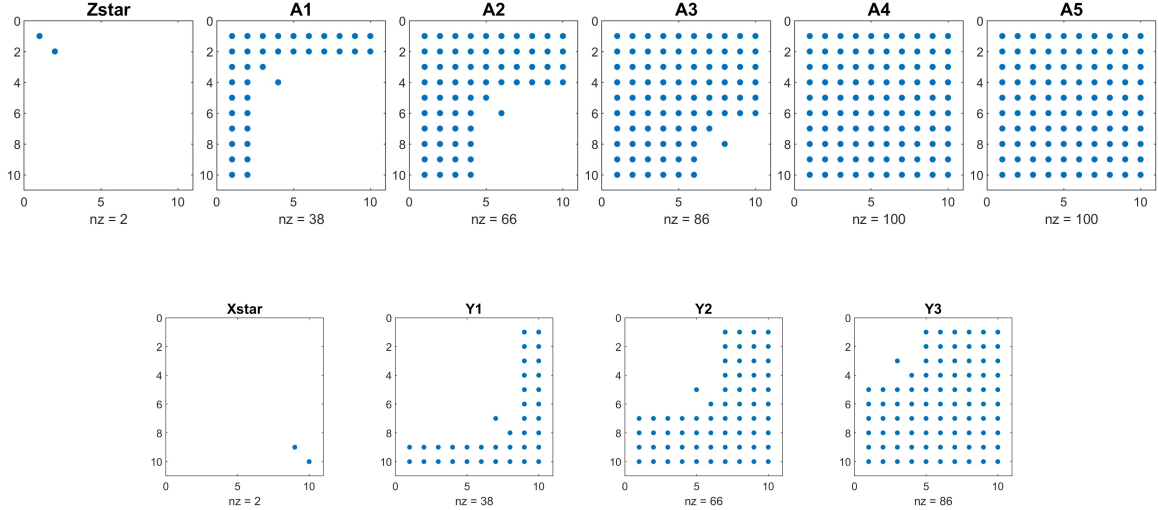

	\centering
	
	\includegraphics[width=\textwidth]{plotA.jpg}
	
	\vspace{1cm}
	
	\includegraphics[width=0.75\textwidth]{plotY.jpg}
	
	\caption{Sparsity structure of $A_i$ and $Y_j$ for the interval partition instances} 
\label{fig:sparsity-interval}
\end{figure}
(We did not actually tested  instances computationally with this choice of parameters, so this figure is just for illustration.)
After rotating by $T, \,$ all the $A_i$ and $Y_j$ were dense. 

We generated instances with $n=10$ and $n=20. $ For each $n$ we chose $m$ to start at $n, \,$ and by steps of $10$ go all the way to $3 n.$ 
For each $n$ and $m$ we chose $k$ to go from $1$ to $|G|.$ And for each $k$ we tested  $20$ instances.

We found that the $5$th DIMACS errors were all at least  $- 2 \times 10^{-6}.$
On $16$ instances SDPT3 did not return a $0$ termination code. 

To make the results accessible, we present them in Tables \ref{tab:worstdimacserror:random}, \ref{tab:worstforwarderror:random}, and \ref{tab:worstratio:random},
and Figures \ref{fig:avg-error-10-10} through \ref{fig:avg-ratio-20-60}. 

We first explain the tables. In each table the row "nonSlater" refers to instances that were generated by Algorithm \ref{algo-nonstrict-SDP}.
These are not guaranteed to be Slater on either the primal or dual side -- 
of course, accidentally they still may turn out to be Slater on one or both of the sides, but this is unlikely.
In each table the row "Slater" refers to instances that were generated by Algorithm  \ref{algo-nonstrict-Slater-SDP}.
 These are  guaranteed to be Slater on both sides.

Table \ref{tab:worstdimacserror:random} shows the worst DIMACS errors. 
For example, the entry $(30,4,1.0922\times 10^{-1})$ 
 in the upper left corner means: among the "nonSlater" instances when $n=10, \,$ 
the worst DIMACS error occurred with the choice  $m=30$ and $k=4,$ and it was equal to
$1.0922\times 10^{-1}.$  
   
Table \ref{tab:worstforwarderror:random} shows the worst forward  errors, and 
Table \ref{tab:worstratio:random} shows the worst, i.e., largest (forward error)/(DIMACS error) ratios. The entries in these tables 
have analogous meanings.

We glean the following from these tables:
\begin{enumerate}
	\item The DIMACS errors can be as large as $10^{-1}.$ Thus, some of these instances are challenging even when we 
	minimize  the 	error in the traditional sense.
	\item The forward errors and the ratio  of the forward errors to DIMACS errors is strikingly large:
	the first is six to seven orders of magnitude larger.
	
	For example, in the "nonSlater" instances the largest ratio is $4.9143 \times 10^6, \,$ which is achieved when $n=20, \, m = 40, \, $ and $k=16.$
	In that case the forward and DIMACS errors are  $1.6099 \times 10^1$ and $3.2760 \times 10^{-6}, \,$ respectively.
	Thus the huge forward error is not predicted by the reasonably small DIMACS error.
	
	\item Even when Slater's condition holds, the errors are still large, although in some cases they are a bit better than in the "nonSlater" case.
	
\end{enumerate}

\co{

\begin{table}[h]
	\centering
	\begin{tabular}{|c|c|c|}
		\hline
		& $n=10$ & $n=20$ \\
		\hline
		nonSlater & $(30,5,2.0390\times 10^{-1})$ & $(60,9,2.1934\times 10^{-1})$ \\
		\hline
		Slater & $(30,5,1.0789\times 10^{-1})$ & $(50,4,6.0146\times 10^{-2})$ \\
		\hline
	\end{tabular}
	
		\caption{Worst DIMACS errors, interval partition instances}
	\label{tab:worstdimacserror:random:old}
	
\end{table}

}

\begin{table}[h]
	\centering
	\begin{tabular}{|c|c|c|}
		\hline
		& $n=10$ & $n=20$ \\
		\hline
		nonSlater & $(30,4,1.0922\times 10^{-1})$ & $(60,9,7.6444\times 10^{-2})$ \\
		\hline
		Slater & $(20,5,7.5534\times 10^{-2})$ & $(60,7,2.4109\times 10^{-2})$ \\
		\hline
	\end{tabular}
	
	\caption{Worst DIMACS errors, interval partition instances}
	\label{tab:worstdimacserror:random}
	
\end{table}

\co{

\begin{table}[h]
	\centering
	\begin{tabular}{|c|c|c|}
		\hline
		& $n=10$ & $n=20$ \\
		\hline
		nonSlater & $(10,6,1.2652\times 10^{1})$ & $(20,2,4.3746\times 10^{1})$ \\
		\hline
		Slater & $(10,2,8.0539\times 10^{0})$ & $(40,2,1.6946\times 10^{1})$ \\
		\hline
	\end{tabular}
	
		\caption{Worst forward errors, interval partition instances}
	\label{tab:worstforwarderror:random:old}
	
\end{table}

}
\begin{table}[h]
	\centering
	\begin{tabular}{|c|c|c|}
		\hline
		& $n=10$ & $n=20$ \\
		\hline
		nonSlater & $(10,6,9.8458\times 10^{0})$ & $(20,2,3.7310\times 10^{1})$ \\
		\hline
		Slater & $(10,6,7.9106\times 10^{0})$ & $(30,12,2.0400\times 10^{1})$ \\
		\hline
	\end{tabular}
	
		\caption{Worst forward errors, interval partition instances}
	\label{tab:worstforwarderror:random}
	
\end{table}

\co{

\begin{table}[h]
	\centering
	\begin{tabular}{|c|c|c|}
		\hline
		& $n=10$ & $n=20$ \\
		\hline
		nonSlater & $(10,6,1.8165\times 10^{6})$ & $(20,15,1.0169\times 10^{7})$ \\
		\hline
		Slater & $(10,6,7.1105\times 10^{5})$ & $(40,16,3.6505\times 10^{6})$ \\
		\hline
	\end{tabular}

	\caption{Worst (forward error)/(DIMACS error) ratios, interval partition instances}
	\label{tab:worstratio:random:old}
	
\end{table}

}

	\begin{table}[h]
		\centering
		\begin{tabular}{|c|c|c|}
			\hline
			& $n=10$ & $n=20$ \\
			\hline
			nonSlater & $(10,6,1.1116\times 10^{6})$ & $(40,16,4.9143\times 10^{6})$ \\
			\hline
			Slater & $(10,5,4.9636\times 10^{5})$ & $(30,16,6.8643\times 10^{6})$ \\
			\hline
		\end{tabular}
		
			\caption{Worst (forward error)/(DIMACS error) ratios, interval partition instances}
		\label{tab:worstratio:random}
		
	\end{table}

	\FloatBarrier

We next visualize the results on Figures \ref{fig:avg-error-10-10} through \ref{fig:avg-ratio-20-60}.
We do so in a compact manner, and only show figures with $n=10$ and $n=20$ and 
$m \in \{ n, 2n, 3n \}.$ (So we left out figures for $n=20$ and $m=30,50.$)
\co{To explain these figures, we recall that $k$ is an upper bound on the singularity degree of 
\eqref{p} with the equation $\la Z^*, X \ra = 0$ attached; and $\ell$ is an upper 
bound on the singularity degree of 
\eqref{d} with the equation $\la X^*, Z \ra = 0$ attached.} In each figure that shows 
DIMACS and forward errors, we plotted the {\em average} DIMACS and forward errors over 
$100$ runs as a function of $k$ (and $k = \ell$).
 In each figure that shows ratios, we plotted 
  the {\em average} ratios of forward/DIMACS errors over $100$ runs.

From these figures we learn that: 
\begin{itemize}
	\item When $m$ is small, the forward errors generally increase as $k$ and $\ell$ increase. 
	 When $m$ is larger, the increase of the  forward errors is less pronounced as $k$ and $\ell$ grow.
	 In fact, there is often a "dip" in the forward errors for larger $k.$
	 \item Nevertheless, the {\em ratio} of the forward and DIMACS errors still increases as $k$ and $\ell$ increase.
	 (The only exception to this rule is  the nonSlater instances, when $n=10, $ and $m=30.$) We can phrase this as: the worse the true error gets,  the more difficult it is to detect.
\end{itemize}

\begin{figure}[H]
	\centering
	\begin{minipage}{0.48\textwidth}
		\centering
		\includegraphics[width=\textwidth]{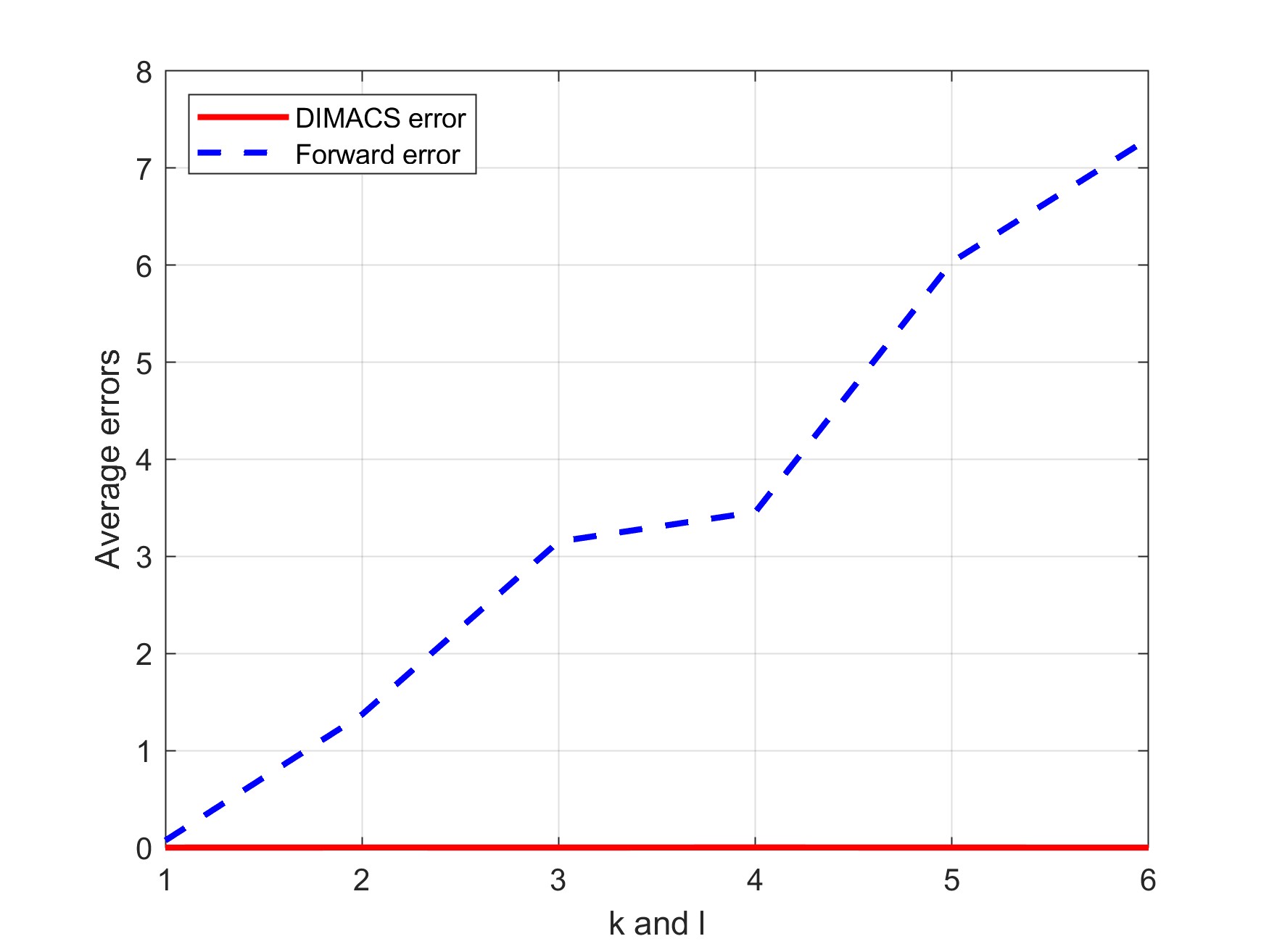}
	\end{minipage}
	\hfill
	\begin{minipage}{0.48\textwidth}
		\centering
		\includegraphics[width=\textwidth]{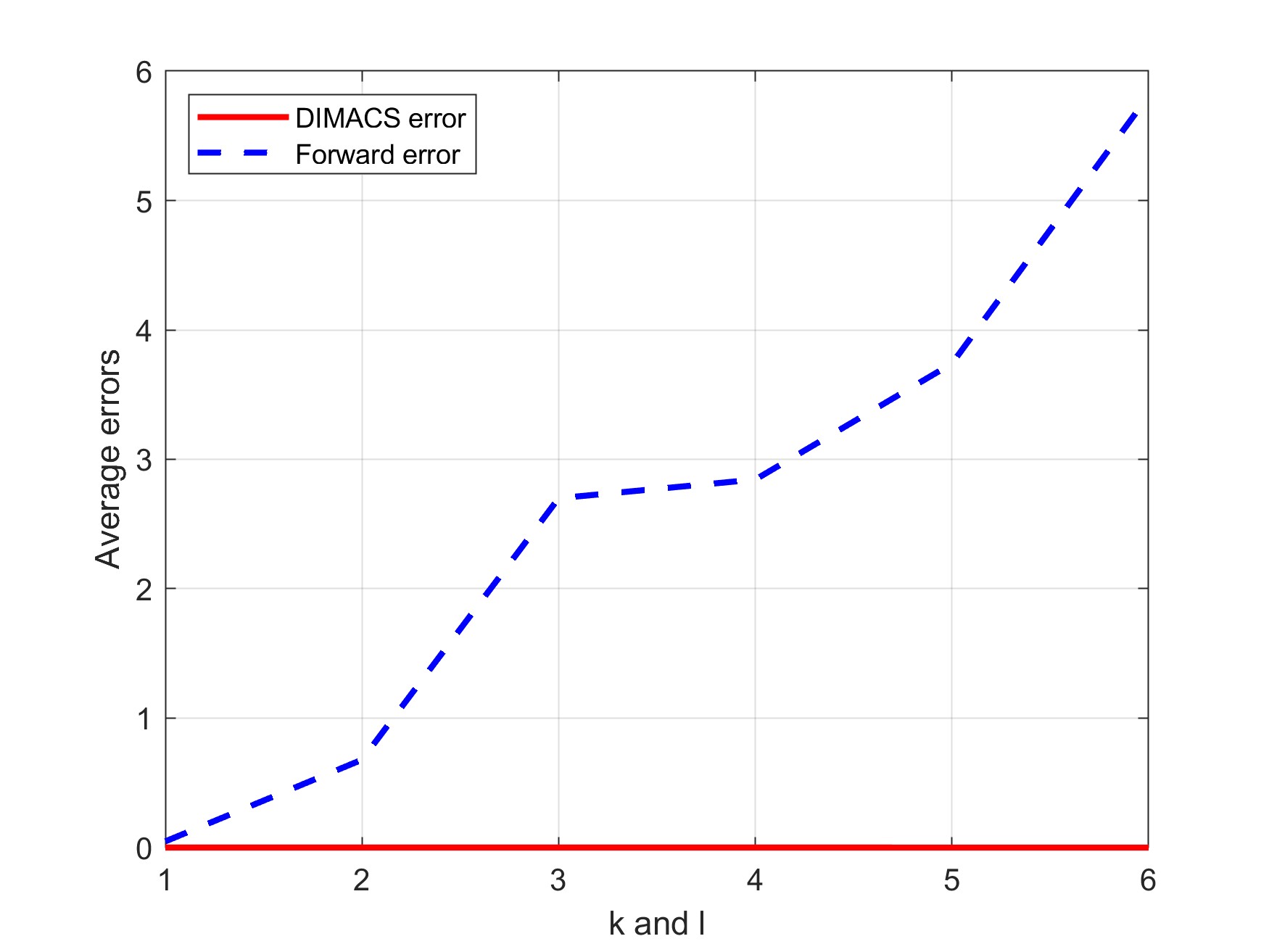}
	\end{minipage}
	\caption{Average errors, $n=10, m=10$: nonSlater (left) and Slater (right).}
	\label{fig:avg-error-10-10}
\end{figure}

\begin{figure}[H]
	\centering
	\begin{minipage}{0.48\textwidth}
		\centering
		\includegraphics[width=\textwidth]{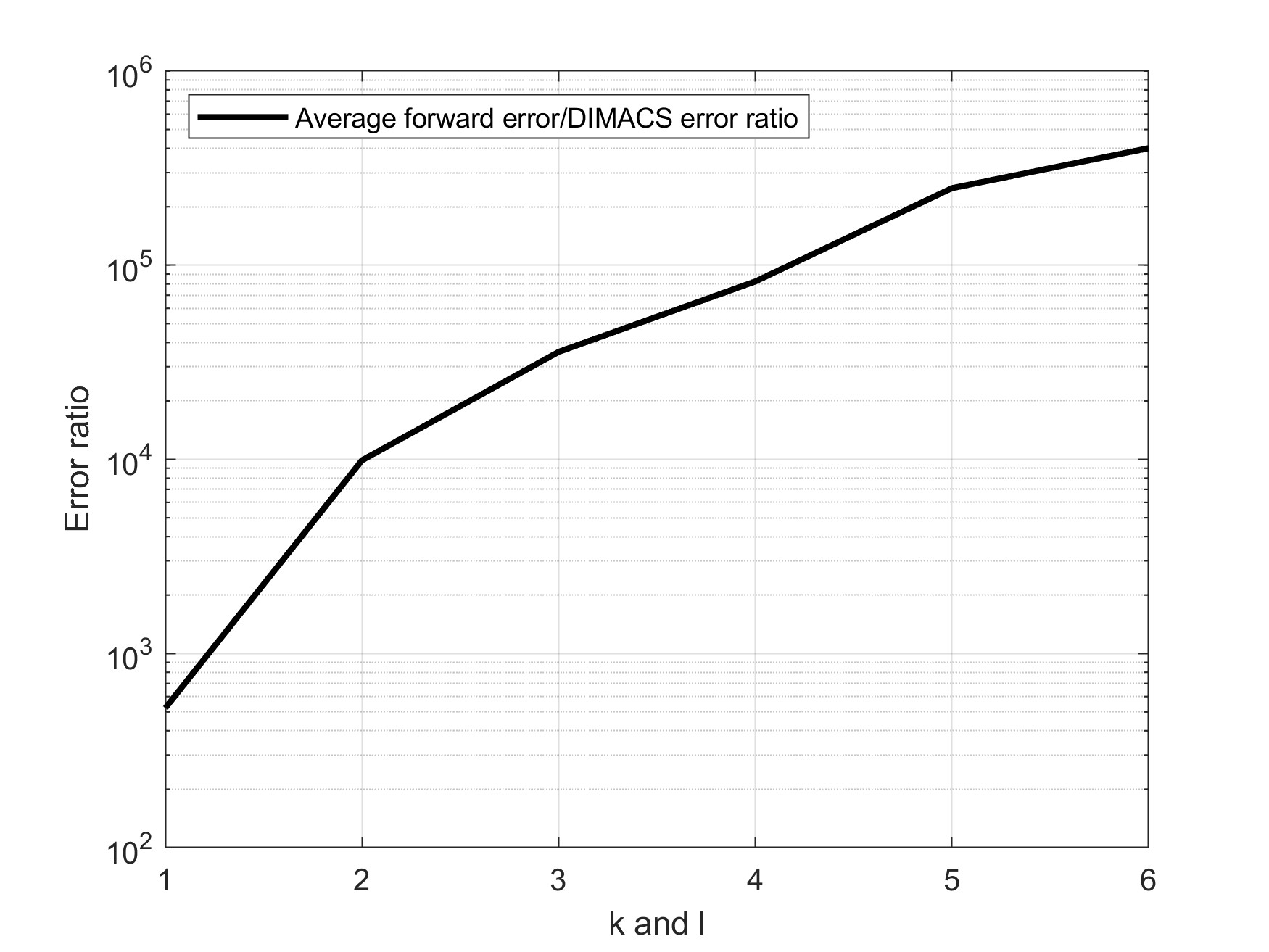}
	\end{minipage}
	\hfill
	\begin{minipage}{0.48\textwidth}
		\centering
		\includegraphics[width=\textwidth]{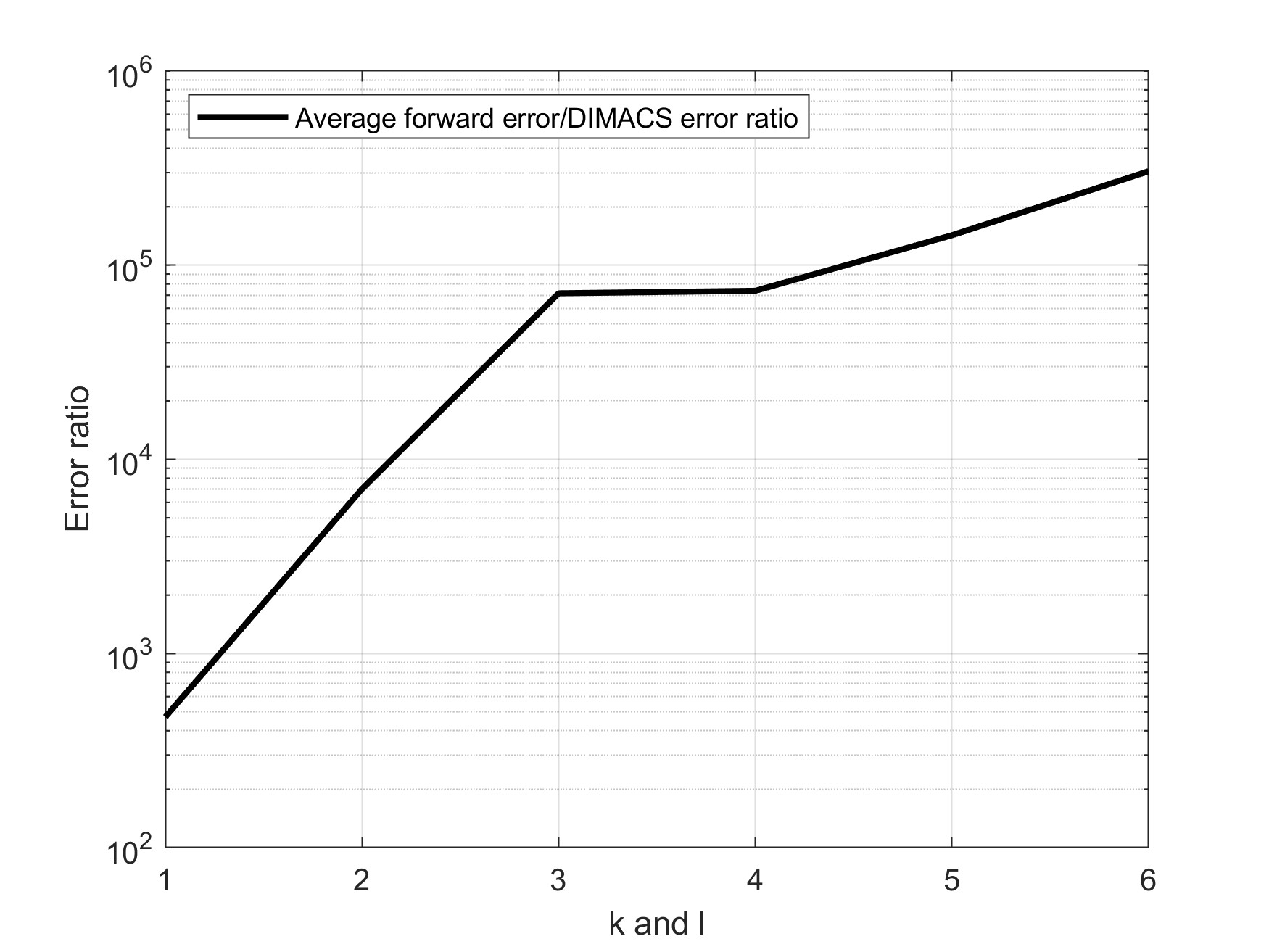}
	\end{minipage}
	\caption{Average ratios, $n=10, m=10$: nonSlater (left) and Slater (right).}
	\label{fig:avg-ratio-10-10}
\end{figure}

\begin{figure}[H]
	\centering
	\begin{minipage}{0.48\textwidth}
		\centering
		\includegraphics[width=\textwidth]{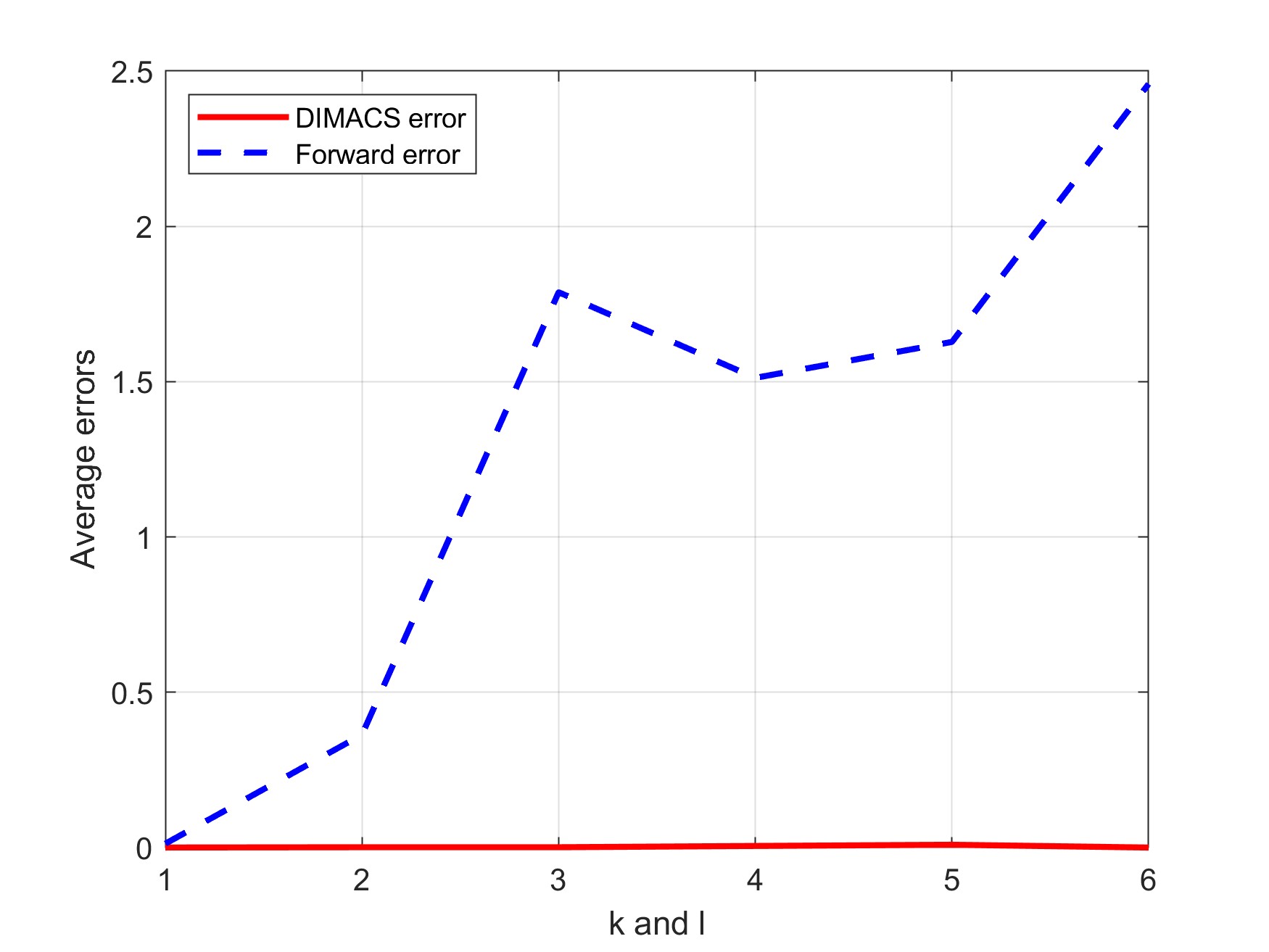}
	\end{minipage}
	\hfill
	\begin{minipage}{0.48\textwidth}
		\centering
		\includegraphics[width=\textwidth]{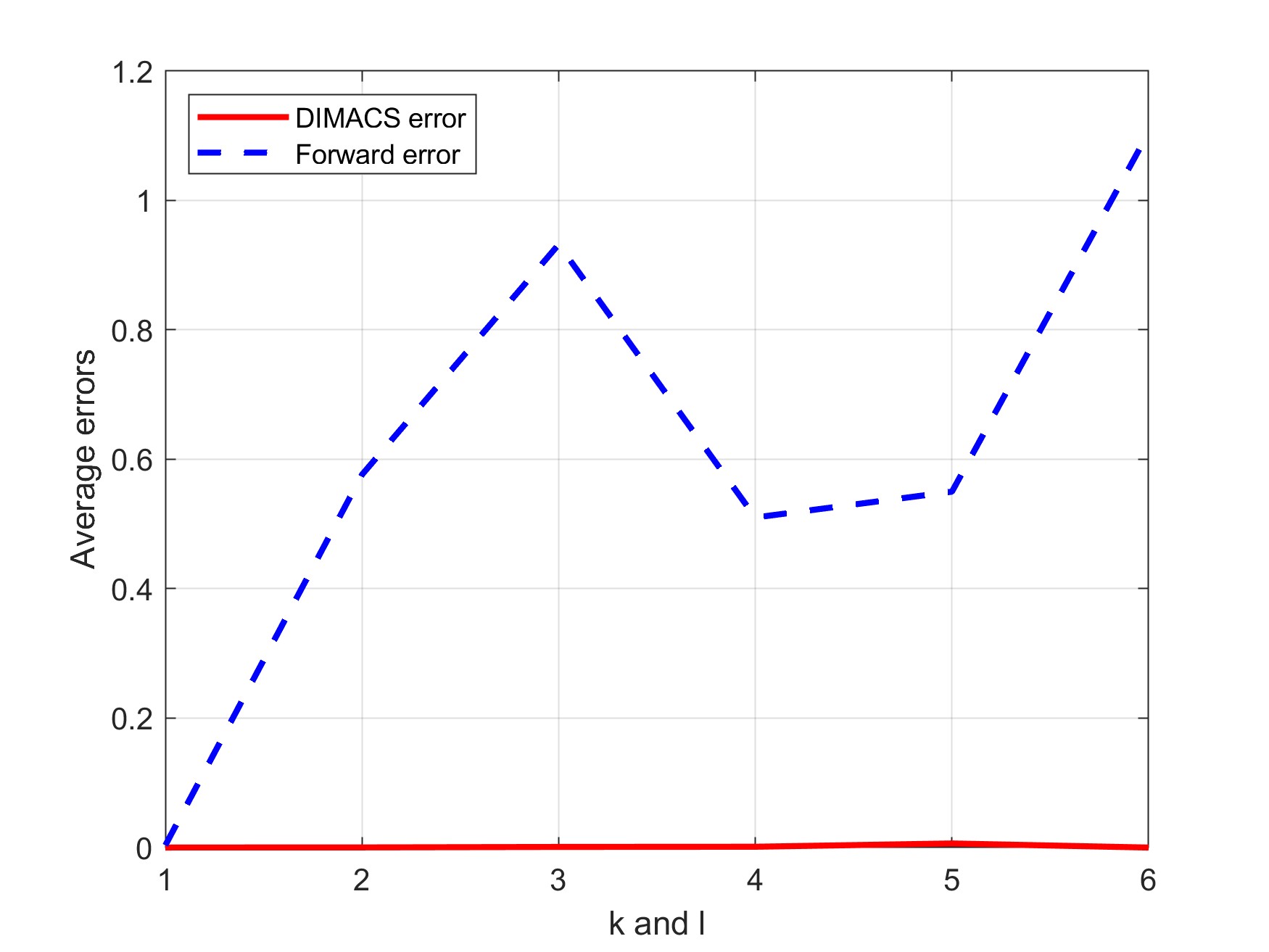}
	\end{minipage}
	\caption{Average errors, $n=10, m=20$: nonSlater (left) and Slater (right).}
	\label{fig:avg-error-10-20}
\end{figure}

\begin{figure}[H]
	\centering
	\begin{minipage}{0.48\textwidth}
		\centering
		\includegraphics[width=\textwidth]{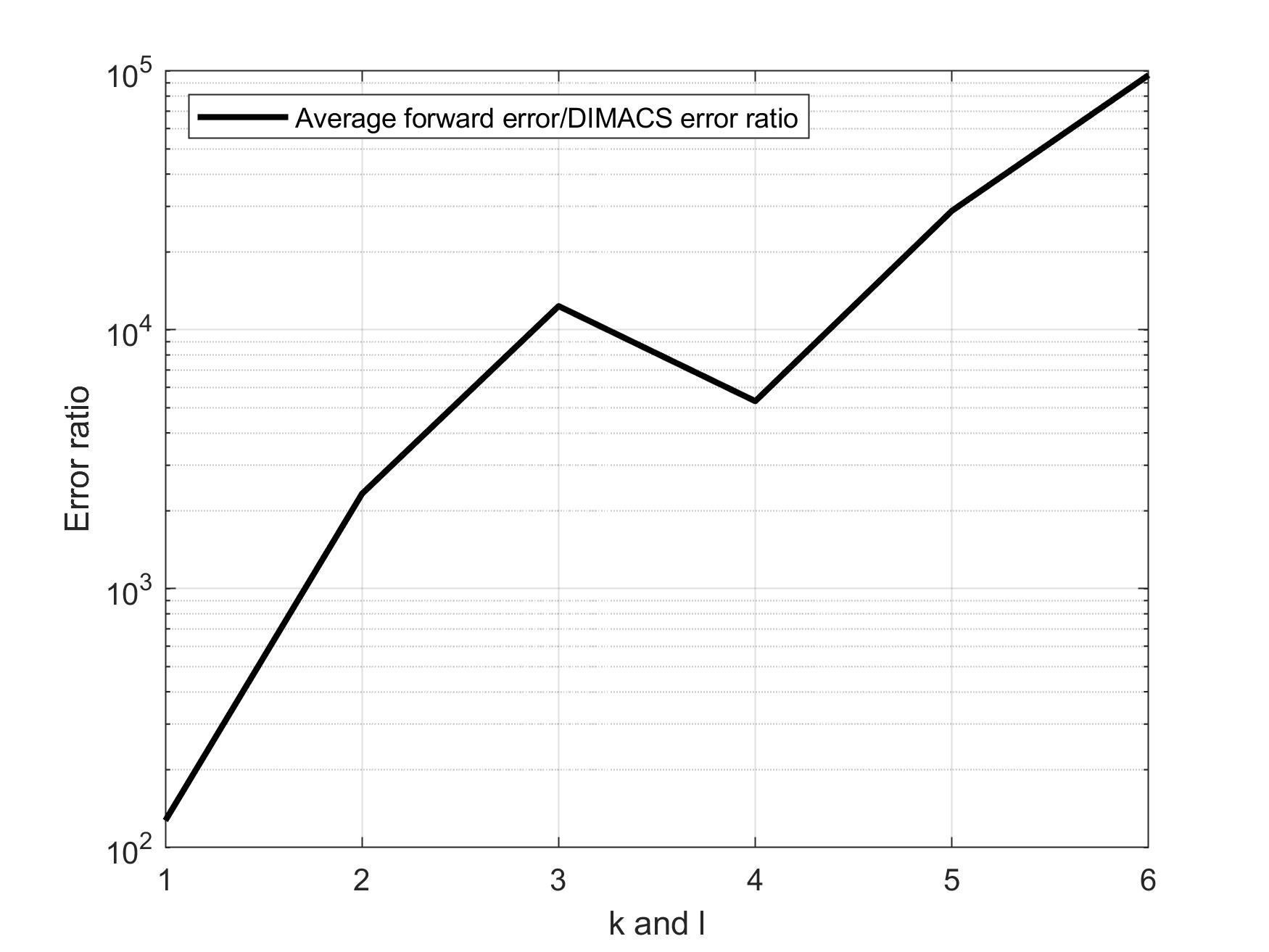}
	\end{minipage}
	\hfill
	\begin{minipage}{0.48\textwidth}
		\centering
		\includegraphics[width=\textwidth]{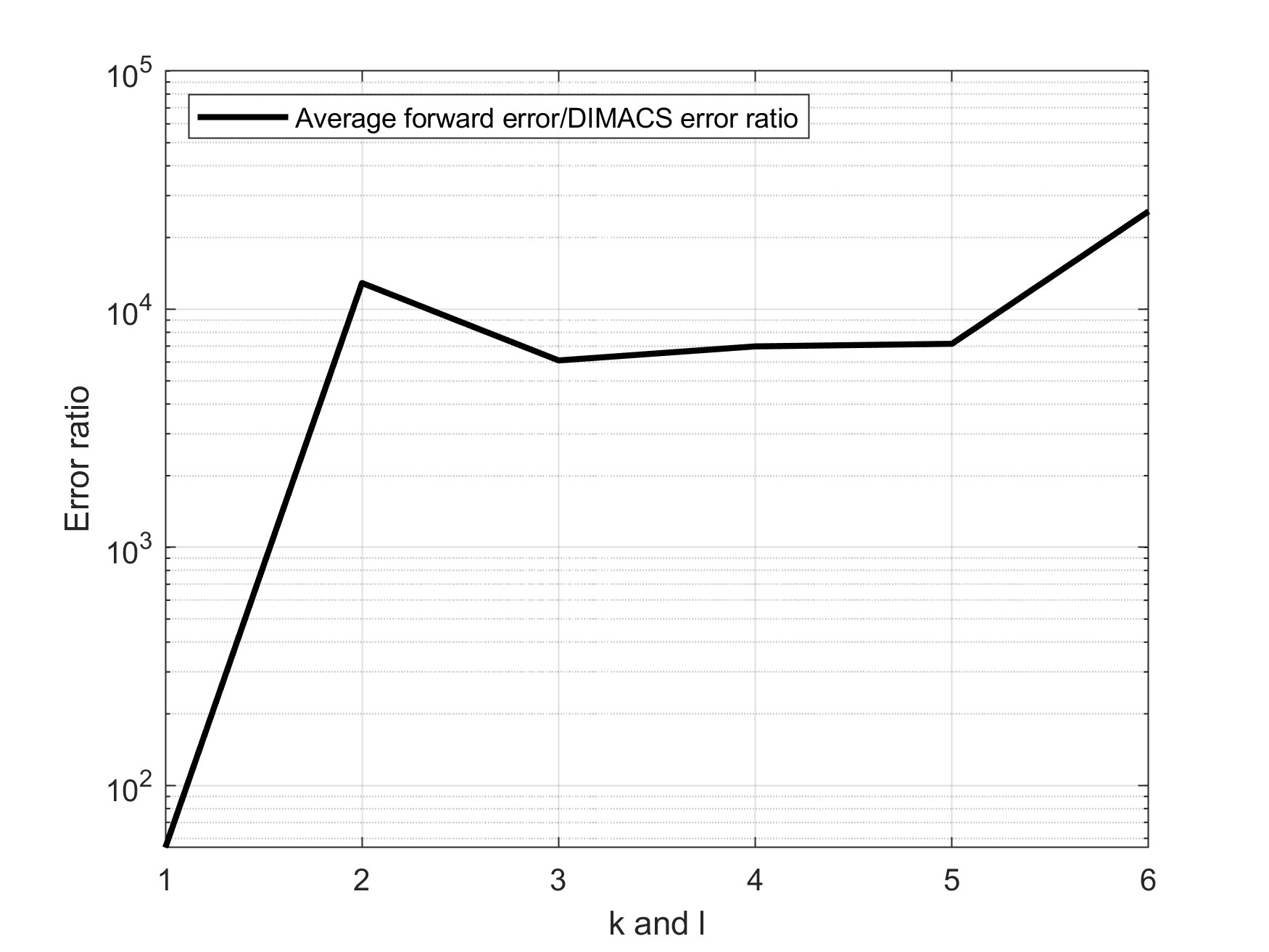}
	\end{minipage}
	\caption{Average ratios, $n=10, m=20$: nonSlater (left) and Slater (right).}
	\label{fig:avg-ratio-10-20}
\end{figure}

\begin{figure}[H]
	\centering
	\begin{minipage}{0.48\textwidth}
		\centering
		\includegraphics[width=\textwidth]{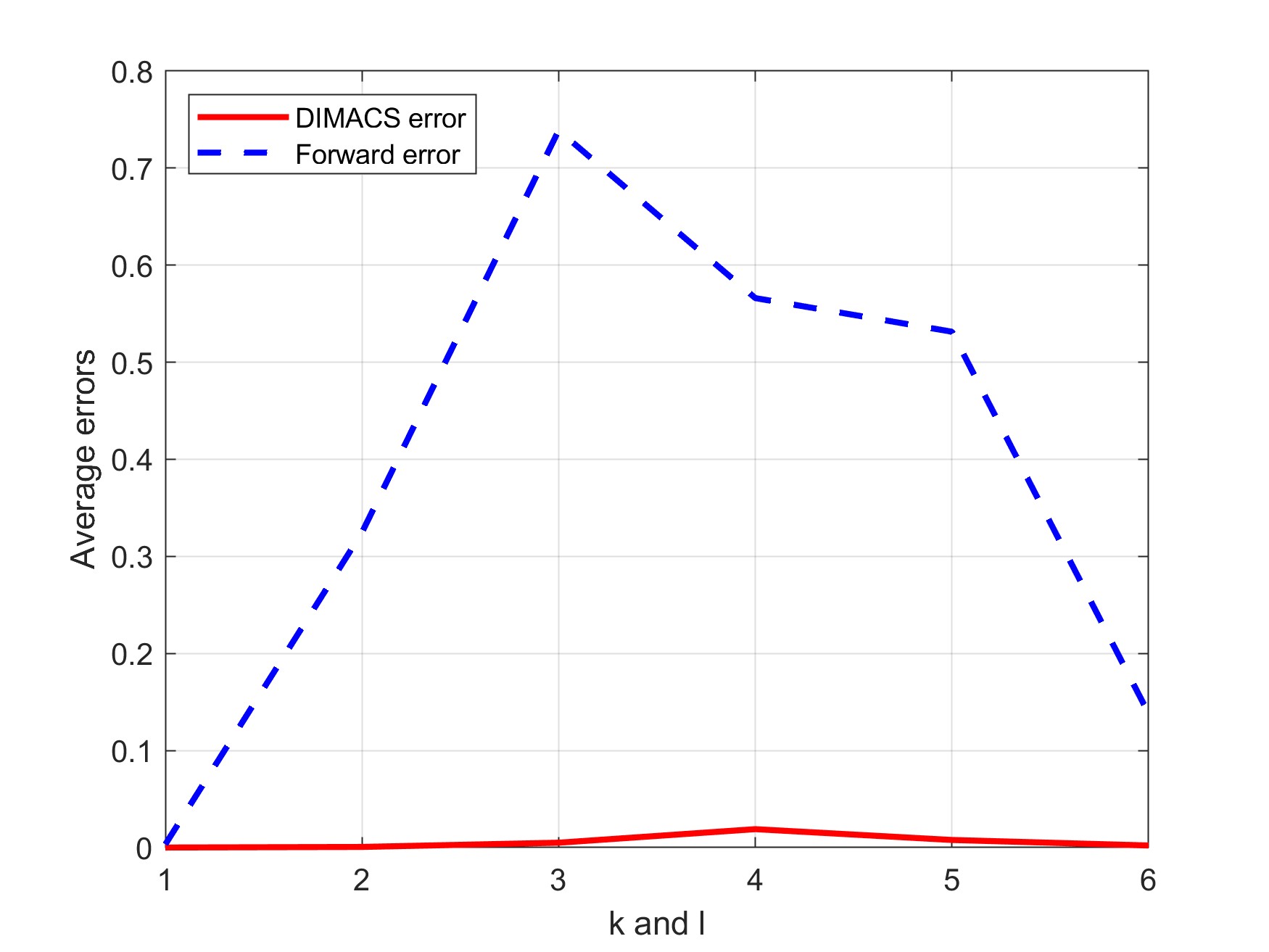}
	\end{minipage}
	\hfill
	\begin{minipage}{0.48\textwidth}
		\centering
		\includegraphics[width=\textwidth]{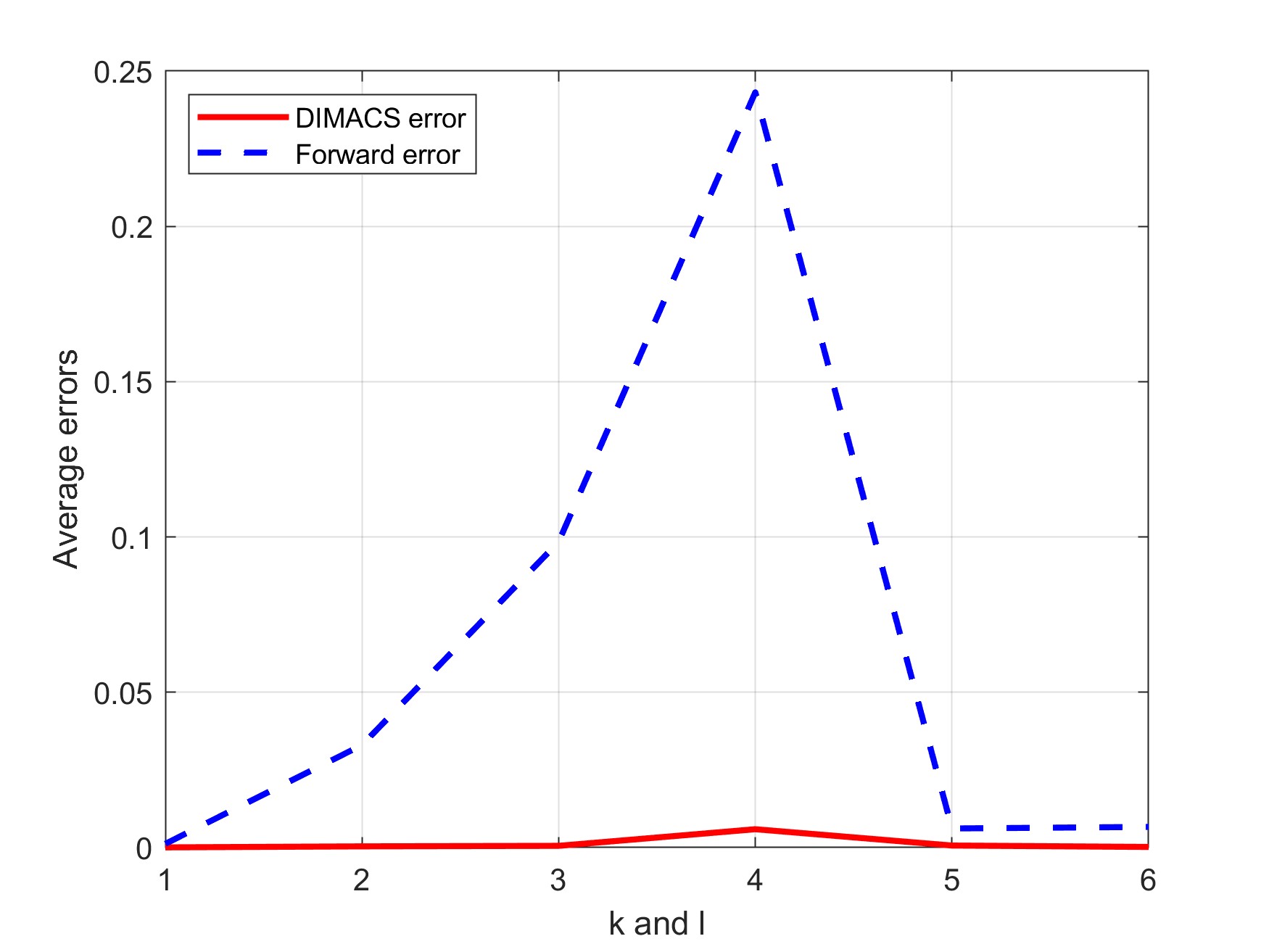}
	\end{minipage}
	\caption{Average errors, $n=10, m=30$: nonSlater (left) and Slater (right).}
	\label{fig:avg-error-10-30}
\end{figure}

\begin{figure}[H]
	\centering
	\begin{minipage}{0.48\textwidth}
		\centering
		\includegraphics[width=\textwidth]{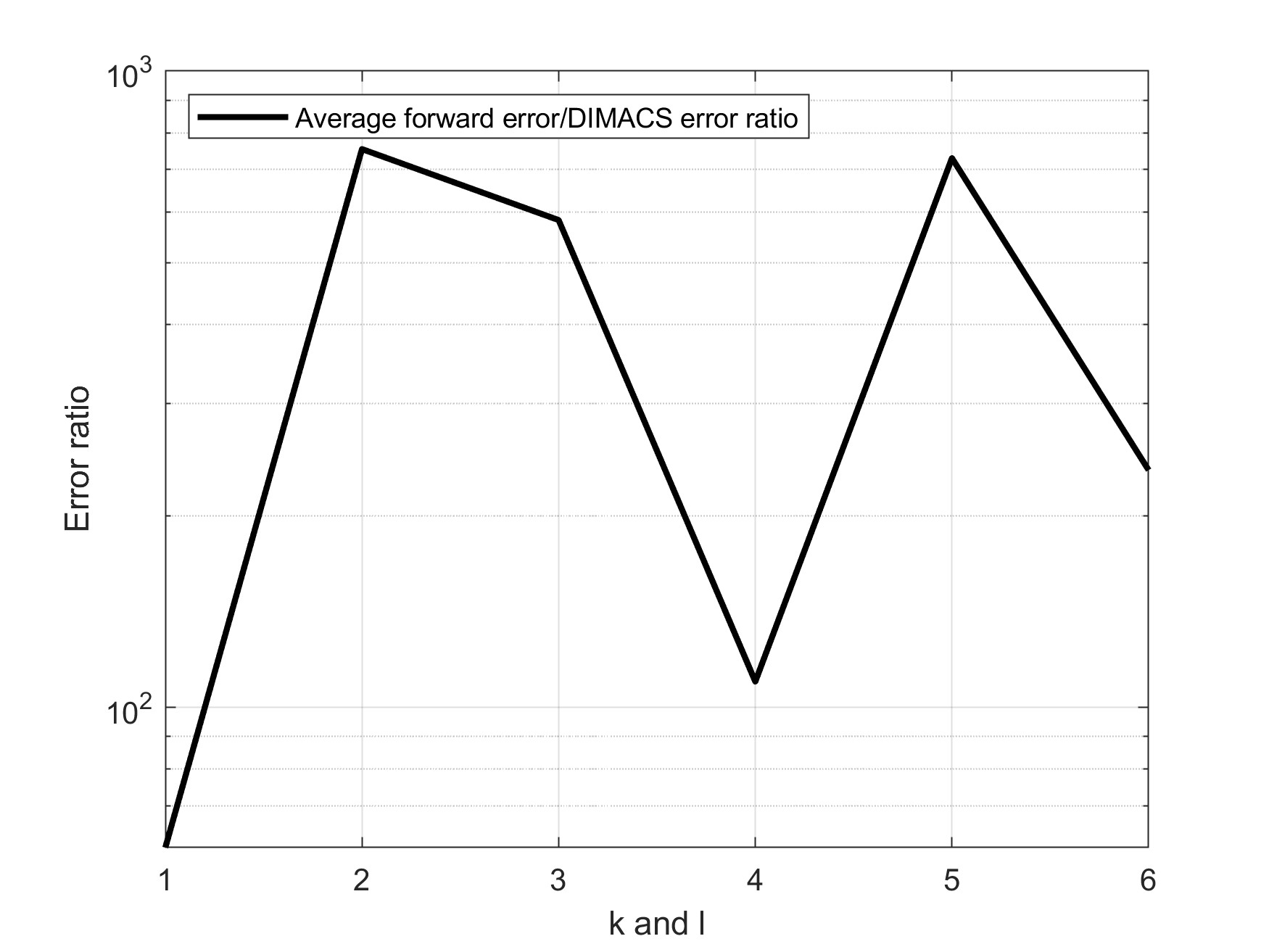}
	\end{minipage}
	\hfill
	\begin{minipage}{0.48\textwidth}
		\centering
		\includegraphics[width=\textwidth]{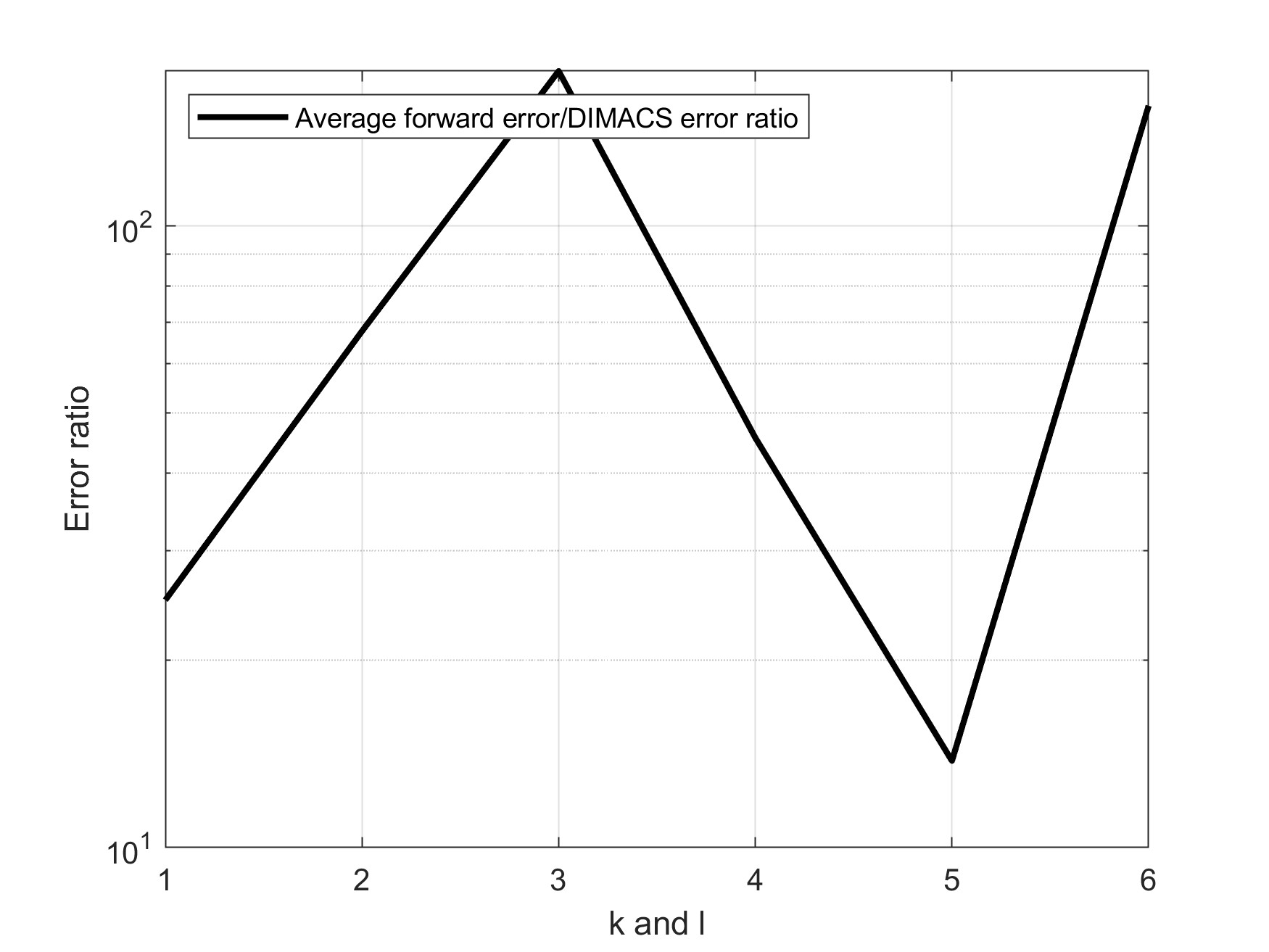}
	\end{minipage}
	\caption{Average ratios, $n=10, m=30$: nonSlater (left) and Slater (right).}
	\label{fig:avg-ratio-10-30}
\end{figure}

\begin{figure}[H]
	\centering
	\begin{minipage}{0.48\textwidth}
		\centering
		\includegraphics[width=\textwidth]{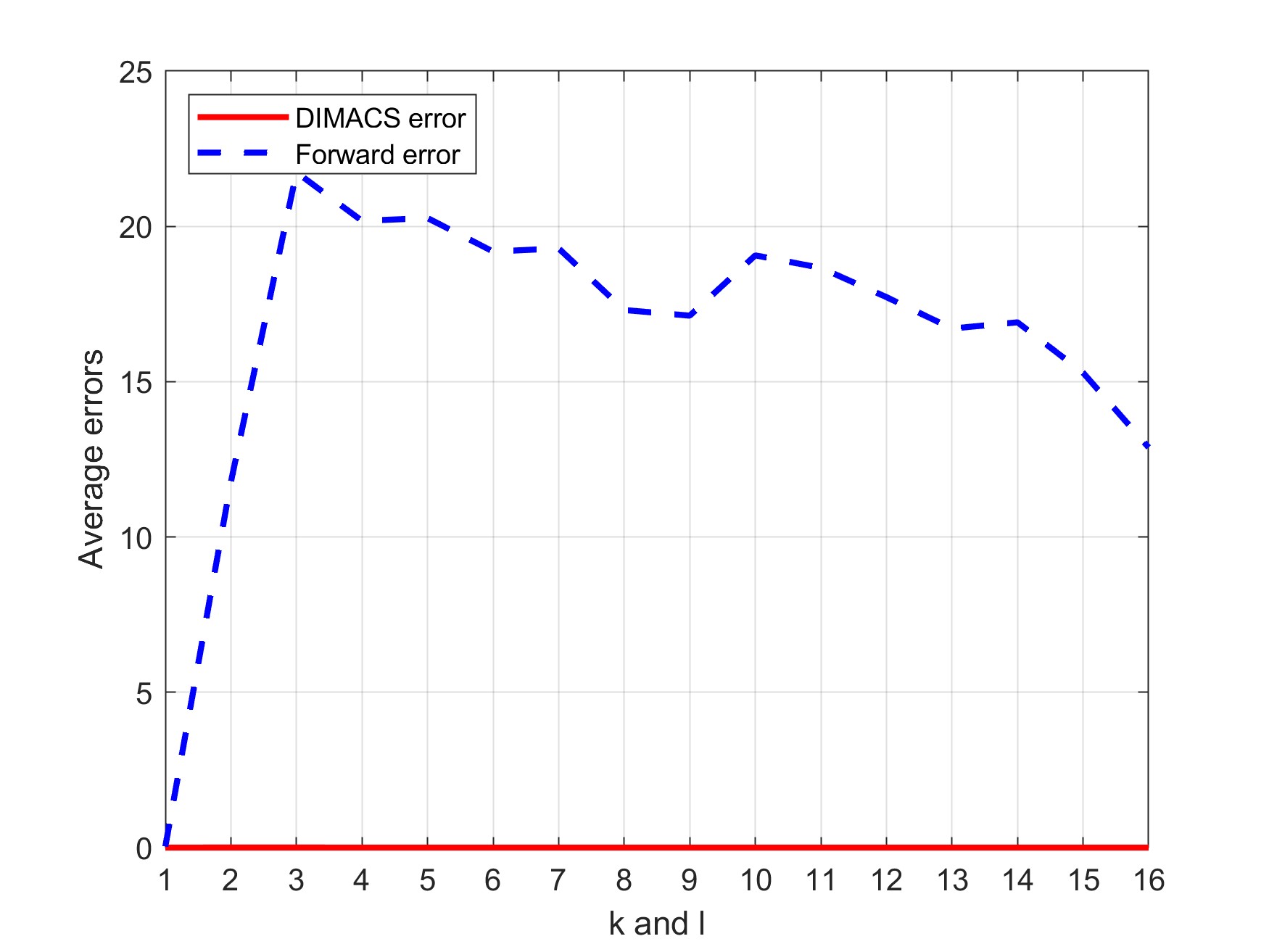}
	\end{minipage}
	\hfill
	\begin{minipage}{0.48\textwidth}
		\centering
		\includegraphics[width=\textwidth]{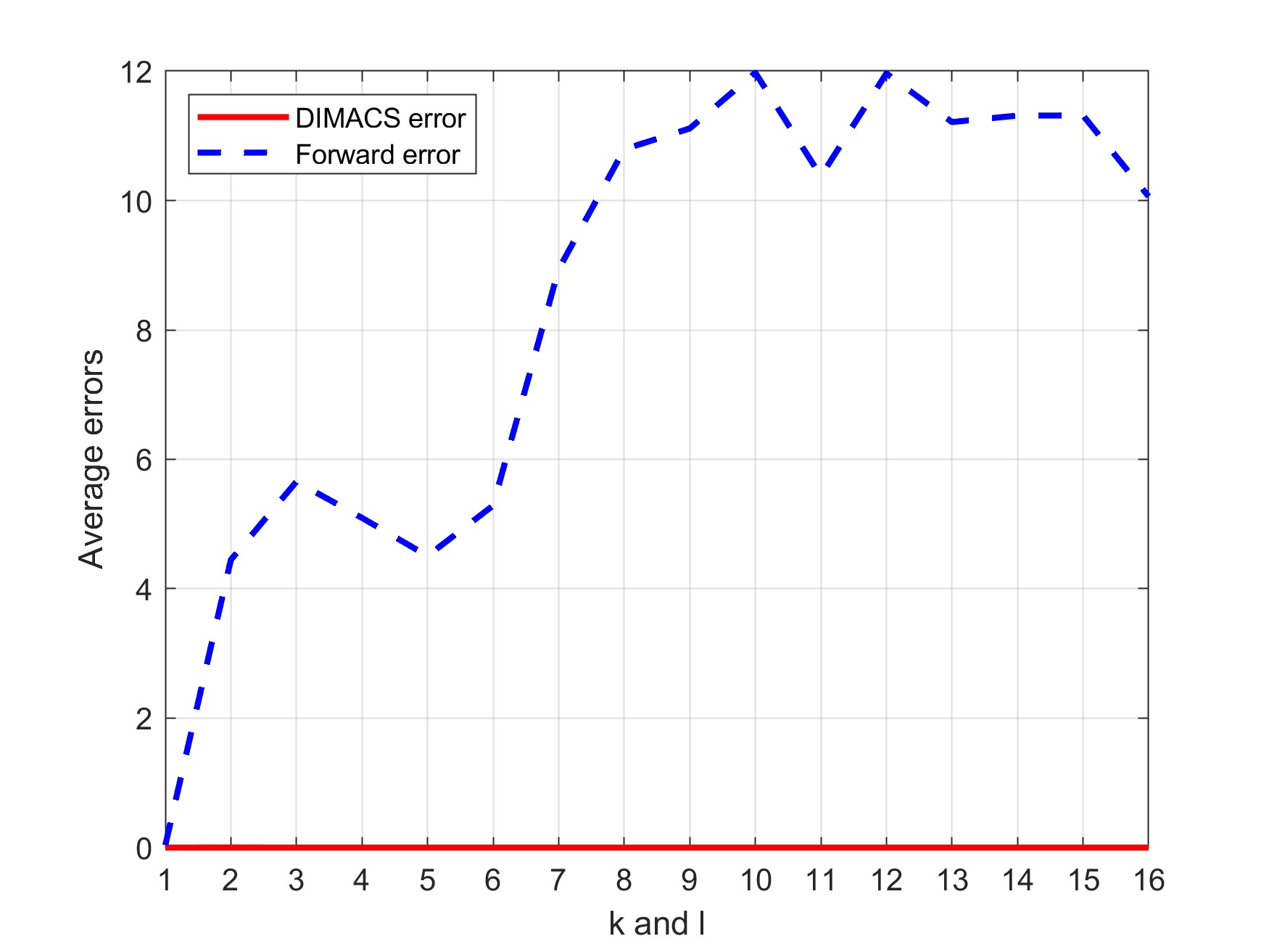}
	\end{minipage}
	\caption{Average errors, $n=20, m=20$: nonSlater (left) and Slater (right).}
	\label{fig:avg-error-20-20}
\end{figure}

\begin{figure}[H]
	\centering
	\begin{minipage}{0.48\textwidth}
		\centering
		\includegraphics[width=\textwidth]{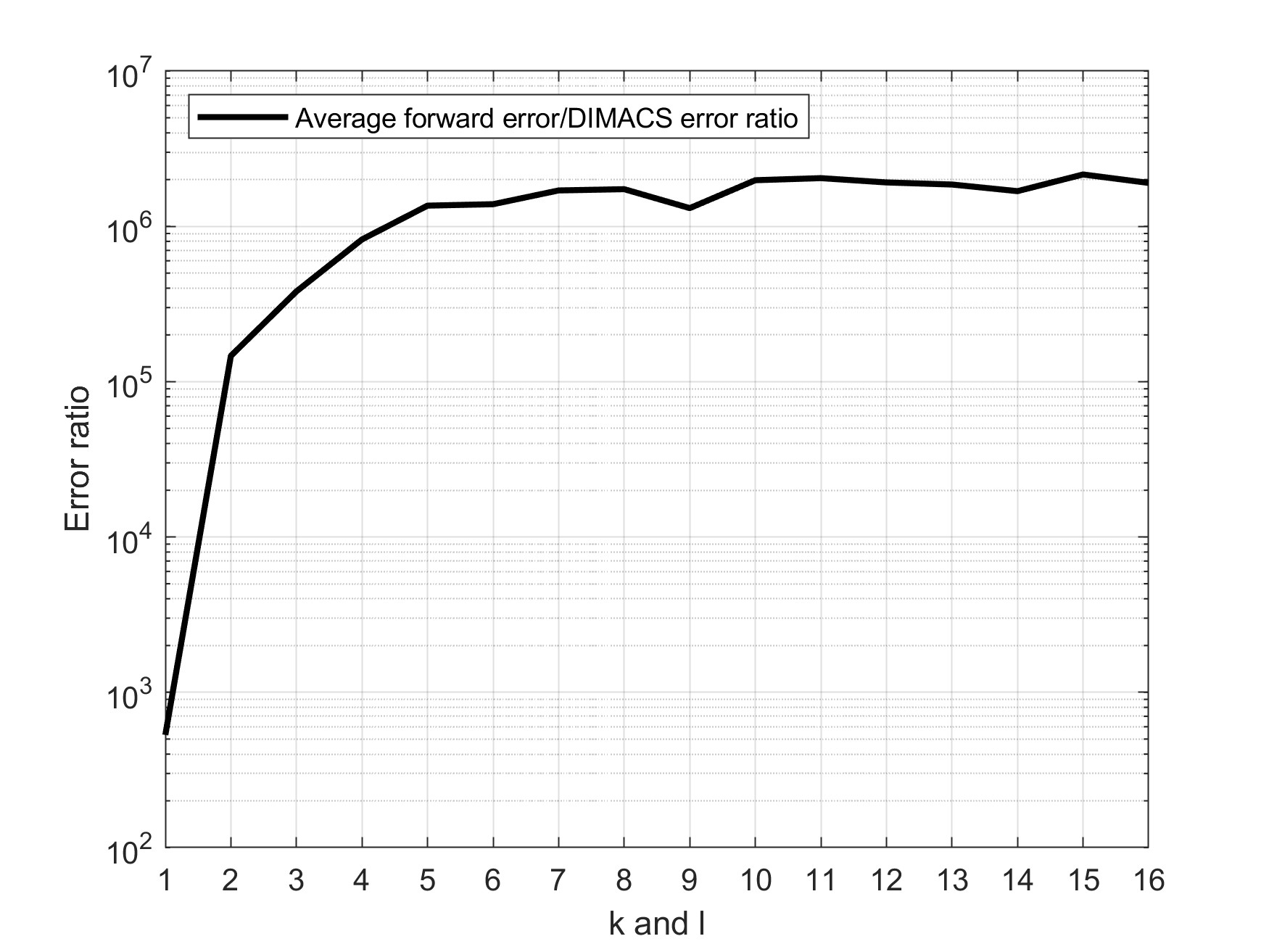}
	\end{minipage}
	\hfill
	\begin{minipage}{0.48\textwidth}
		\centering
		\includegraphics[width=\textwidth]{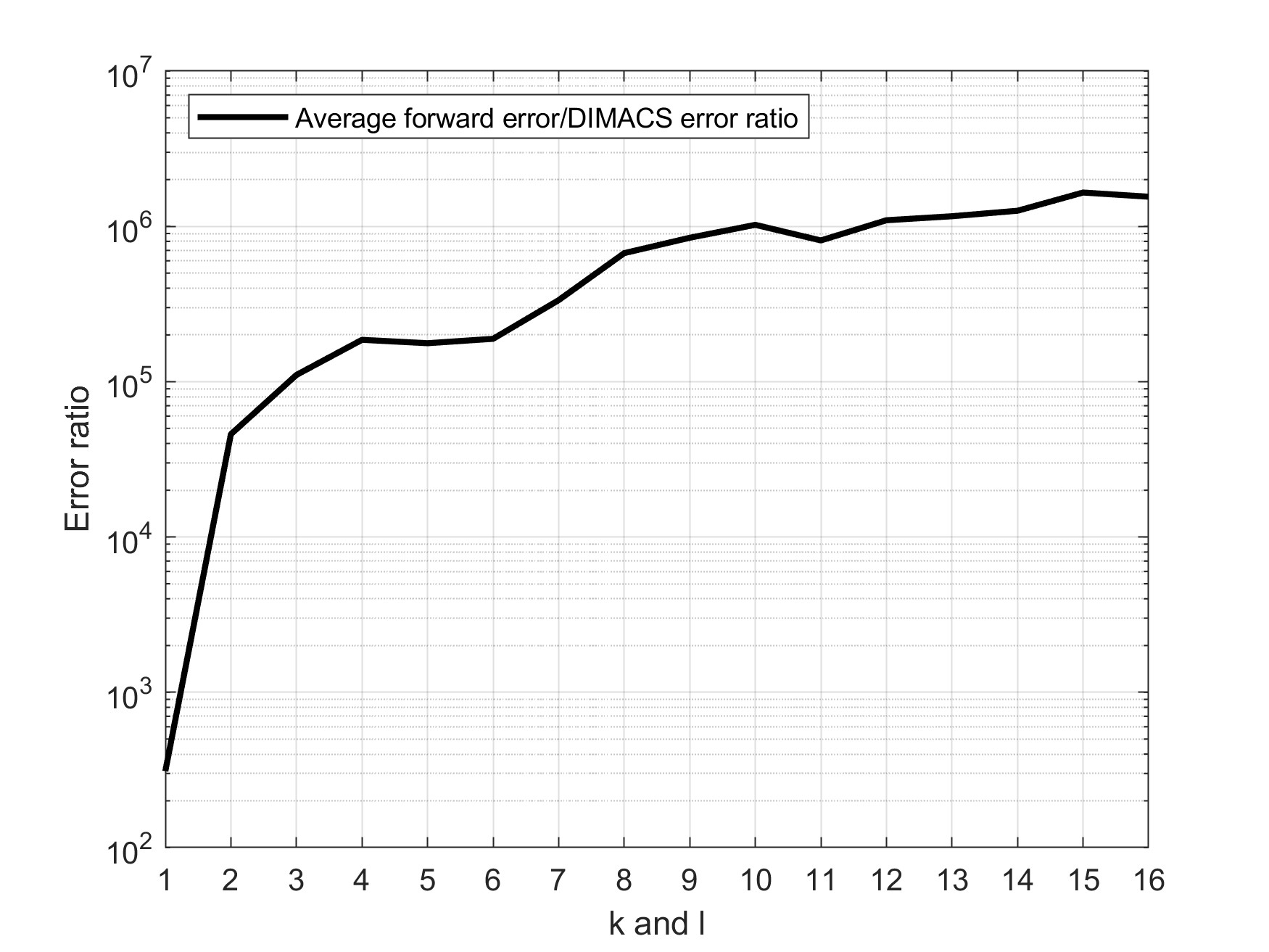}
	\end{minipage}
	\caption{Average ratios, $n=20, m=20$: nonSlater (left) and Slater (right).}
	\label{fig:avg-ratio-20-20}
\end{figure}

\co{
\newpage

\begin{figure}[H]
	\centering
	\begin{minipage}{0.48\textwidth}
		\centering
		\includegraphics[width=\textwidth]{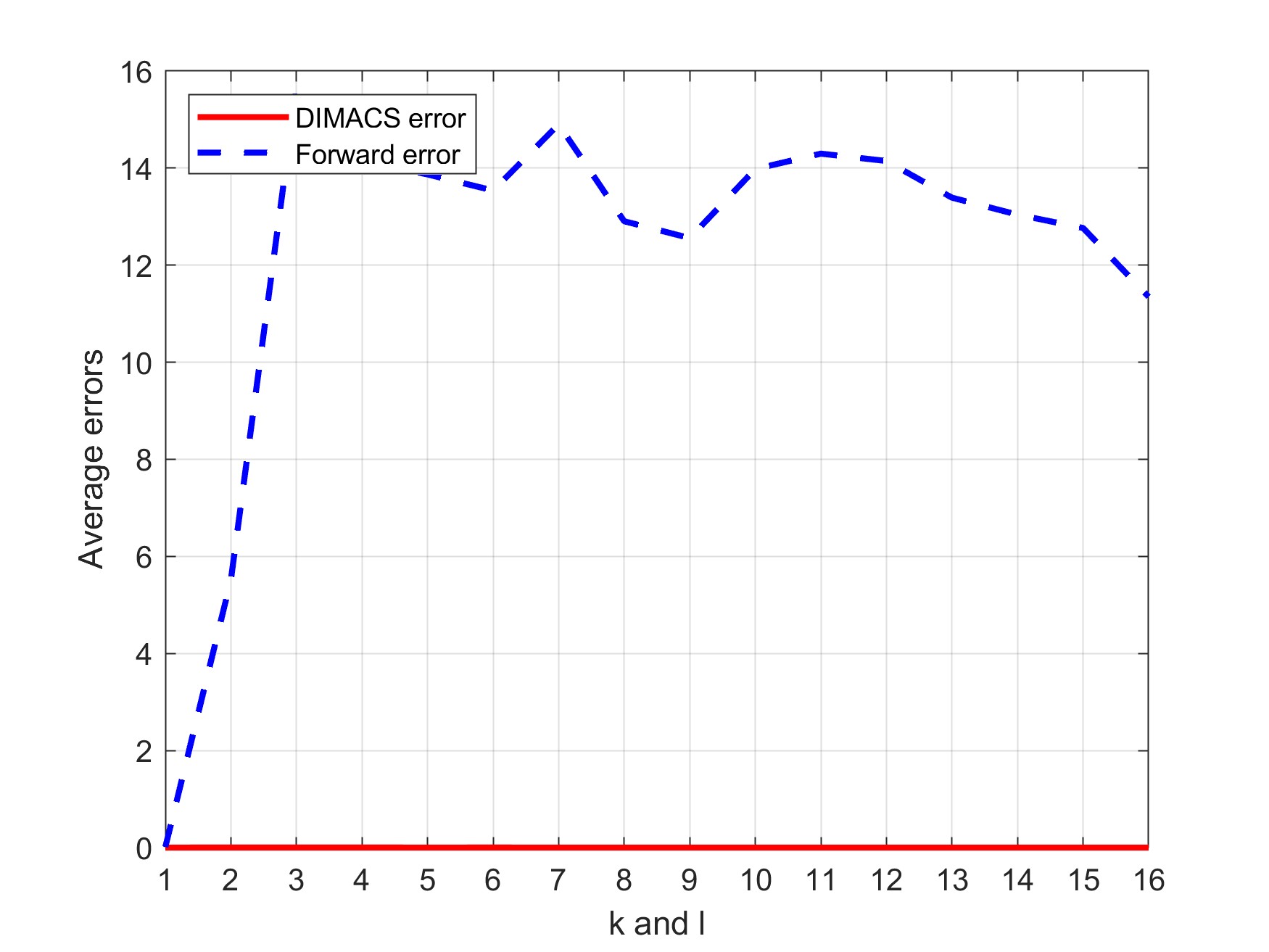}
	\end{minipage}
	\hfill
	\begin{minipage}{0.48\textwidth}
		\centering
		\includegraphics[width=\textwidth]{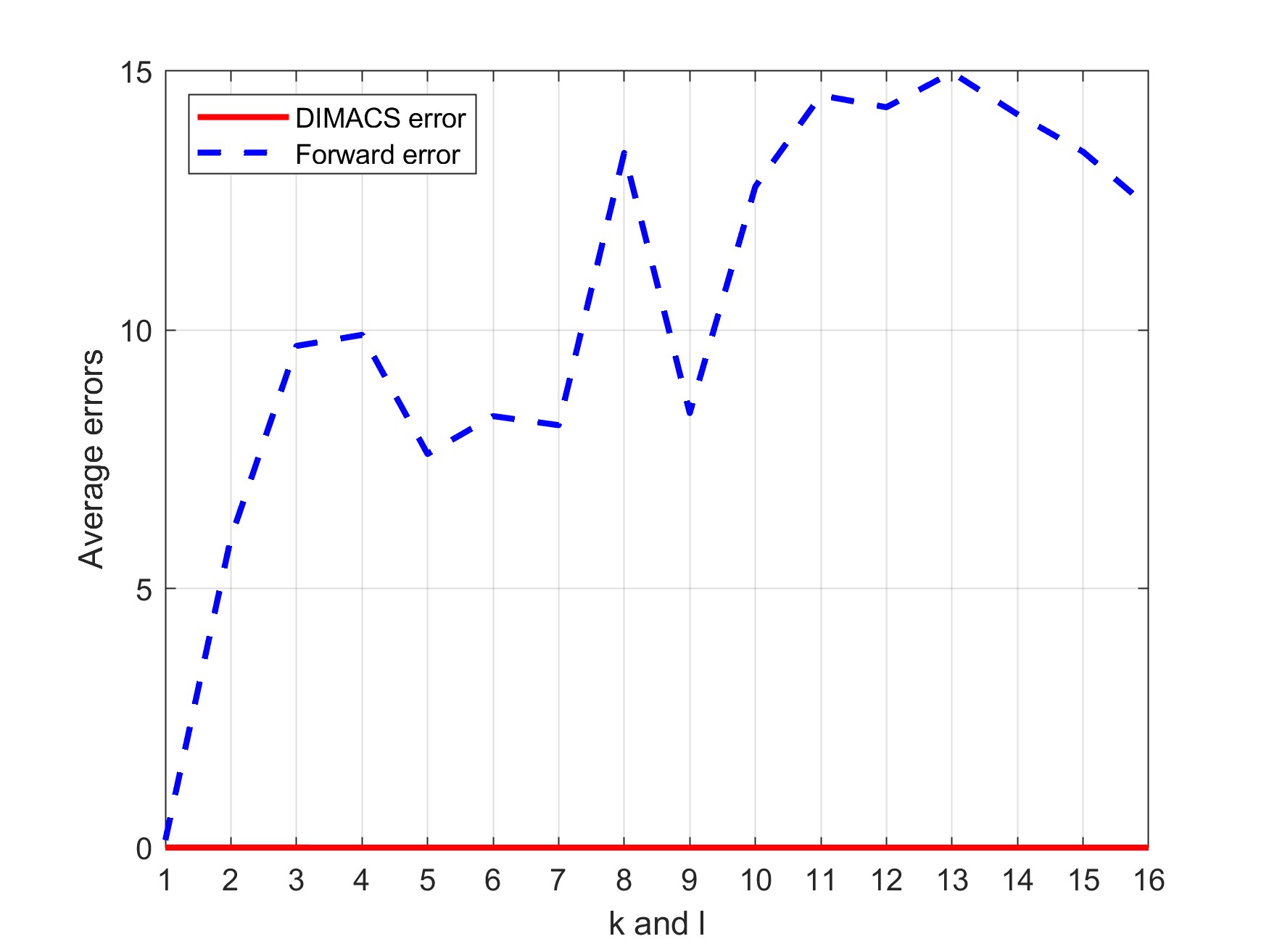}
	\end{minipage}
	\caption{Average errors, $n=20, m=30$: nonSlater (left) and Slater (right).}
	\label{fig:avg-error-20-30}
\end{figure}

\begin{figure}[H]
	\centering
	\begin{minipage}{0.48\textwidth}
		\centering
		\includegraphics[width=\textwidth]{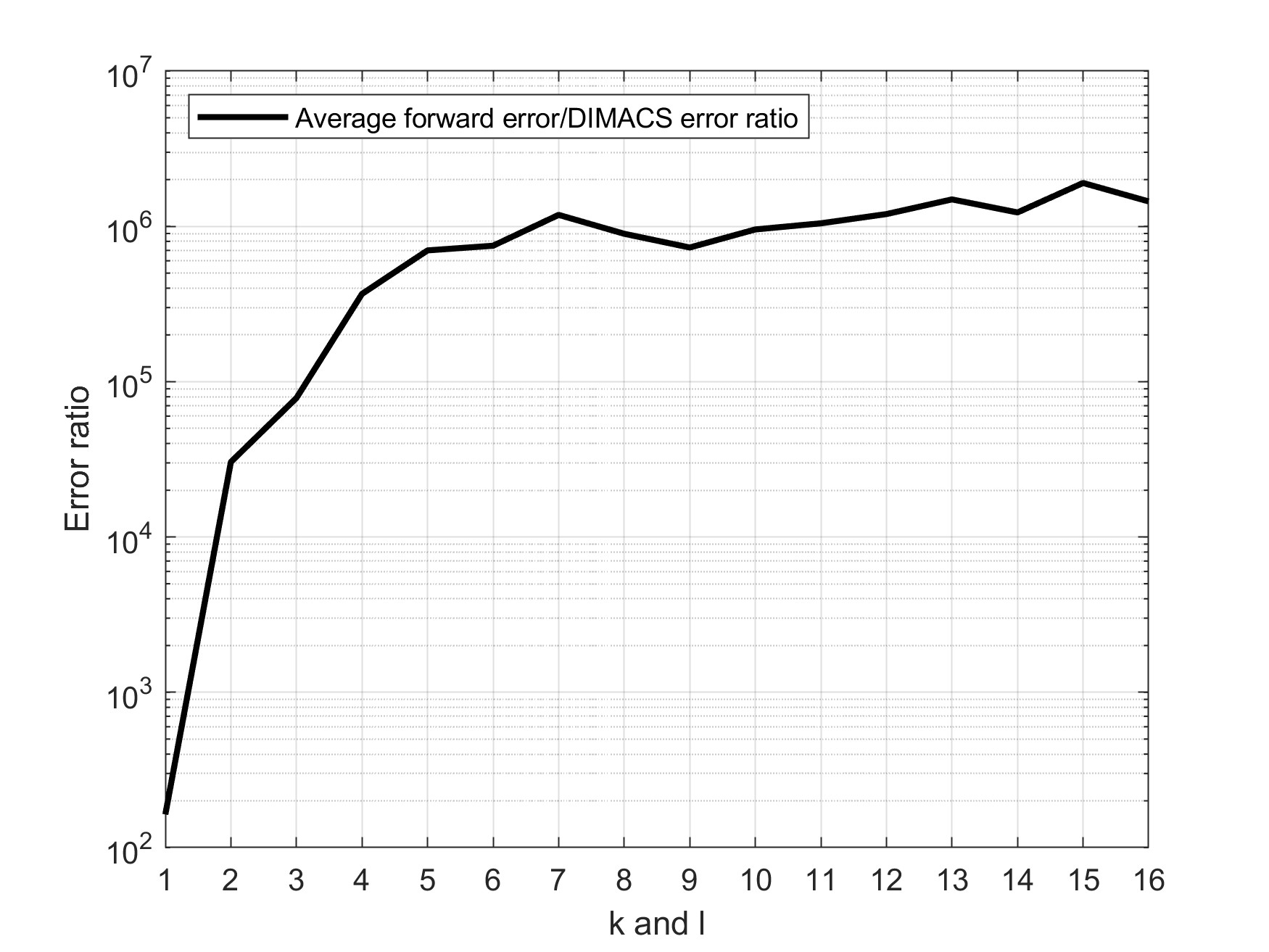}
	\end{minipage}
	\hfill
	\begin{minipage}{0.48\textwidth}
		\centering
		\includegraphics[width=\textwidth]{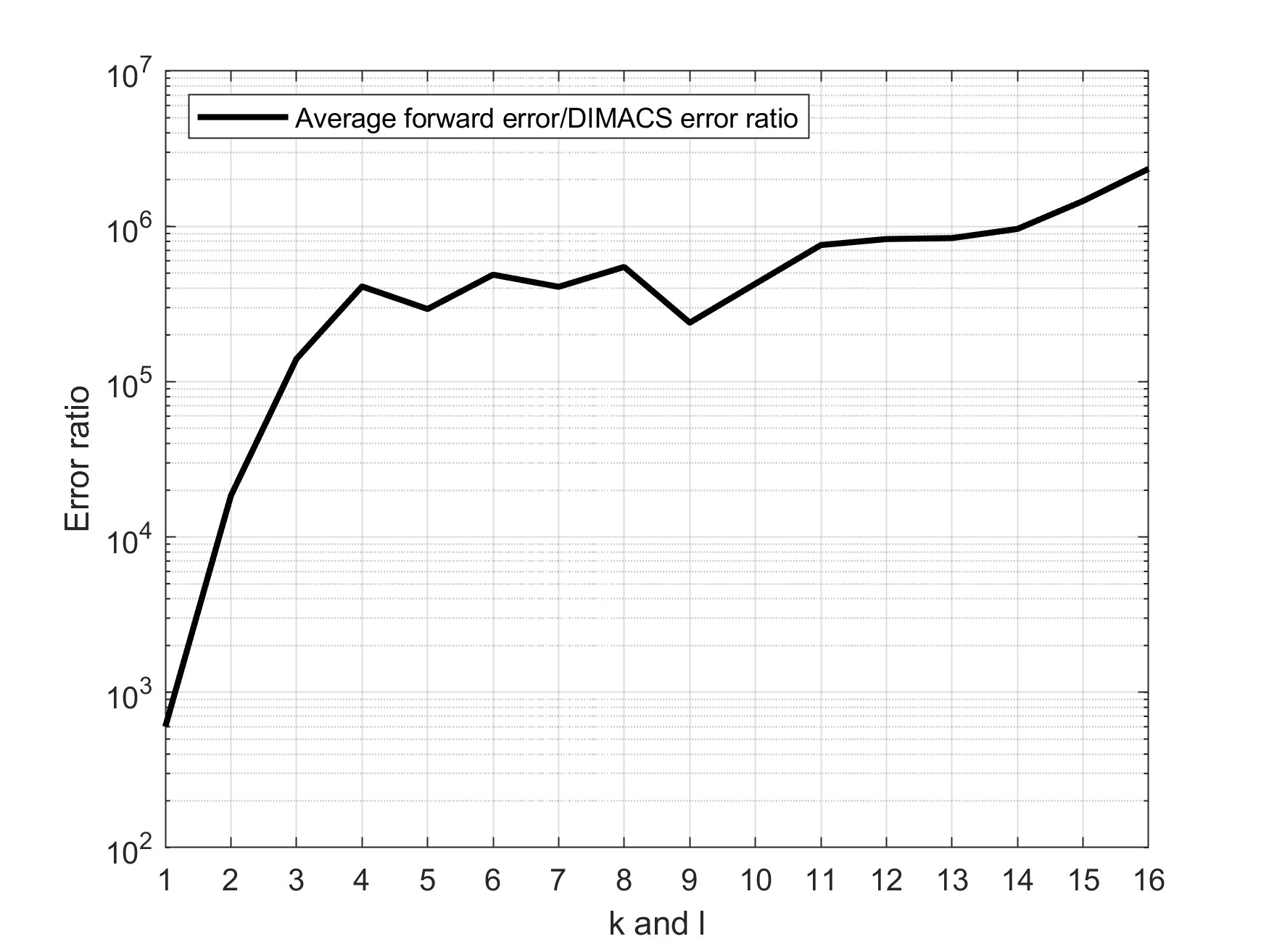}
	\end{minipage}
	\caption{Average ratios, $n=20, m=30$: nonSlater (left) and Slater (right).}
	\label{fig:avg-ratio-20-30}
\end{figure}
}

\begin{figure}[H]
	\centering
	\begin{minipage}{0.48\textwidth}
		\centering
		\includegraphics[width=\textwidth]{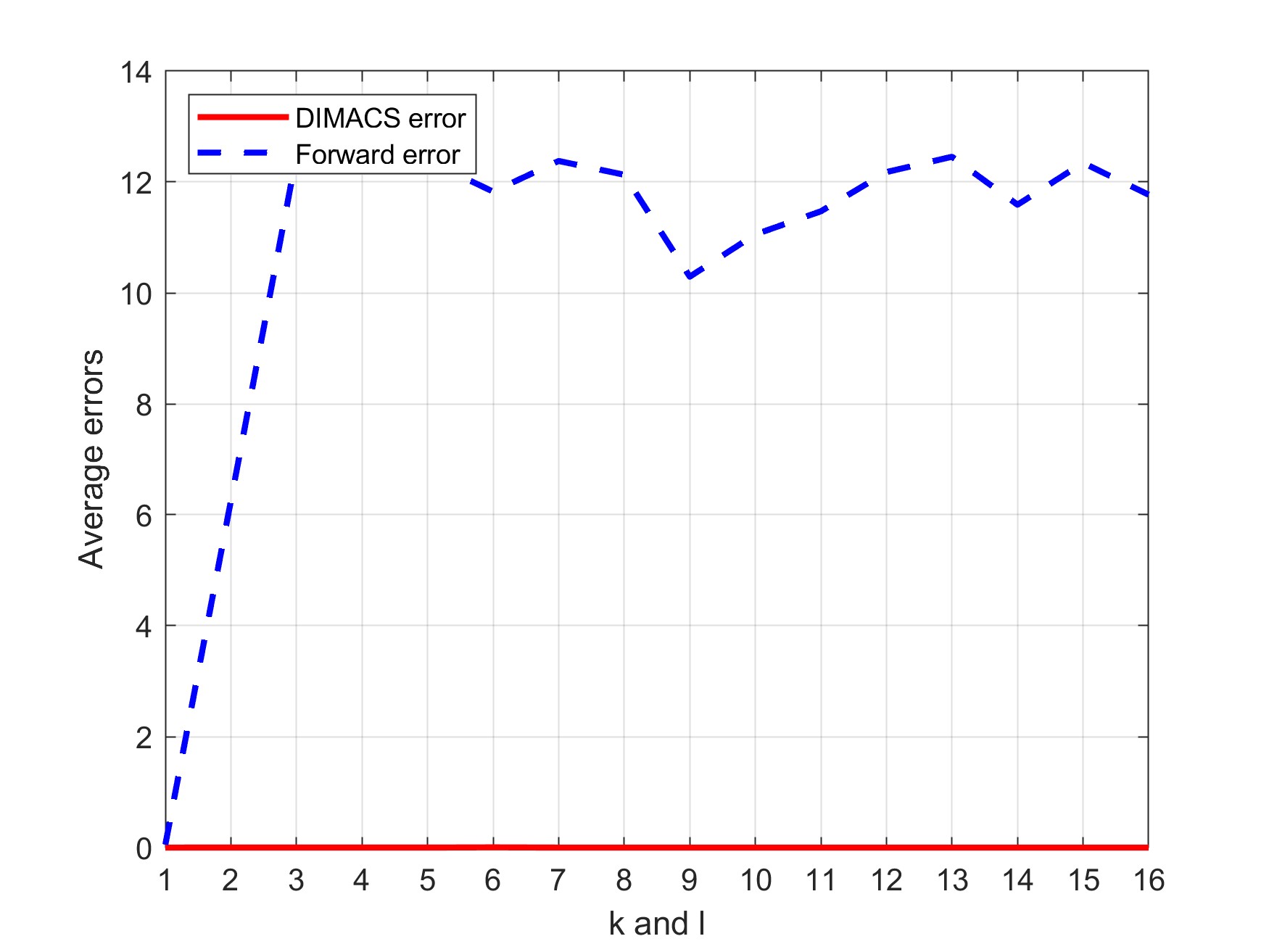}
	\end{minipage}
	\hfill
	\begin{minipage}{0.48\textwidth}
		\centering
		\includegraphics[width=\textwidth]{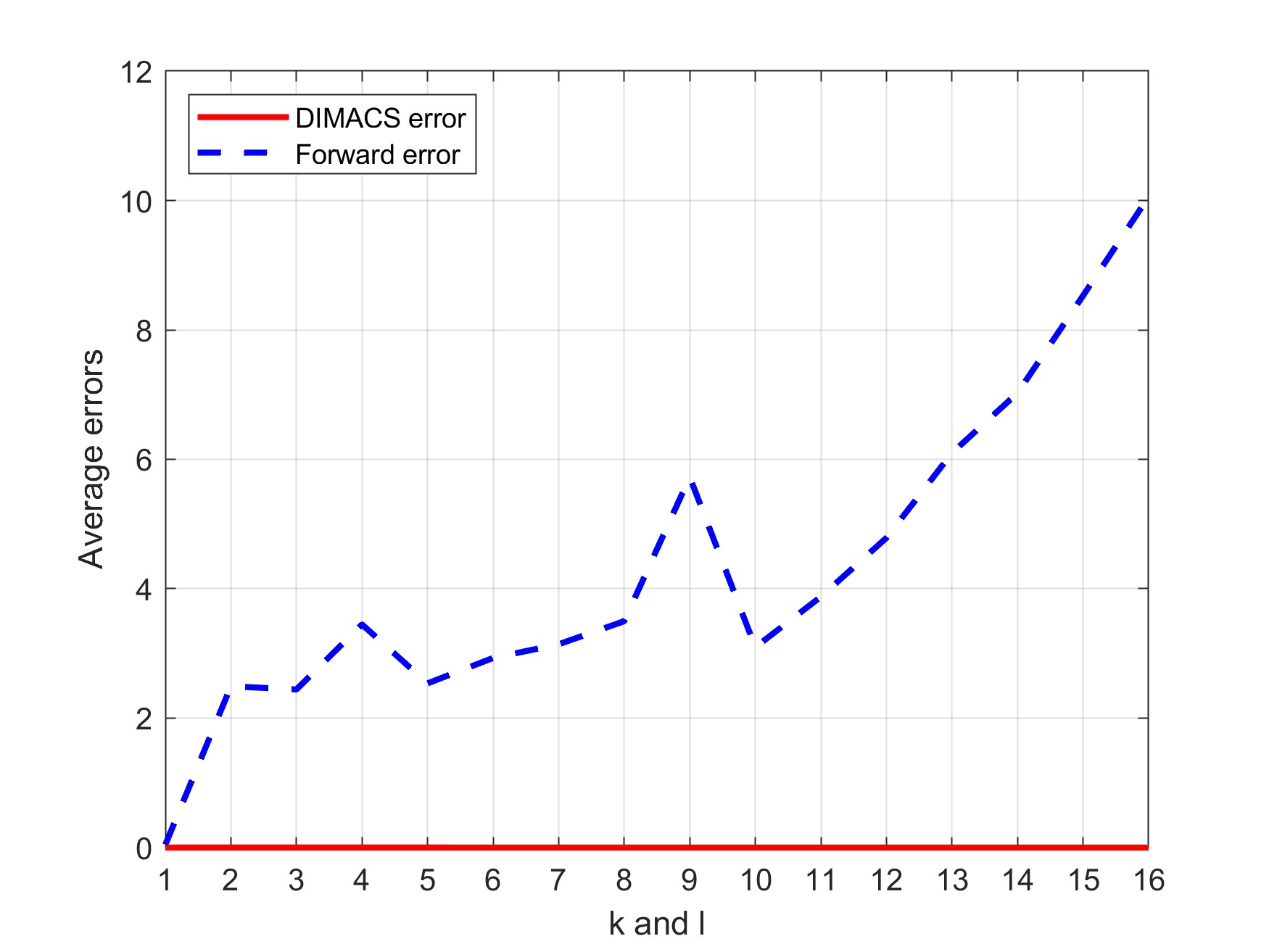}
	\end{minipage}
	\caption{Average errors, $n=20, m=40$: nonSlater (left) and Slater (right).}
	\label{fig:avg-error-20-40}
\end{figure}

\begin{figure}[H]
	\centering
	\begin{minipage}{0.48\textwidth}
		\centering
		\includegraphics[width=\textwidth]{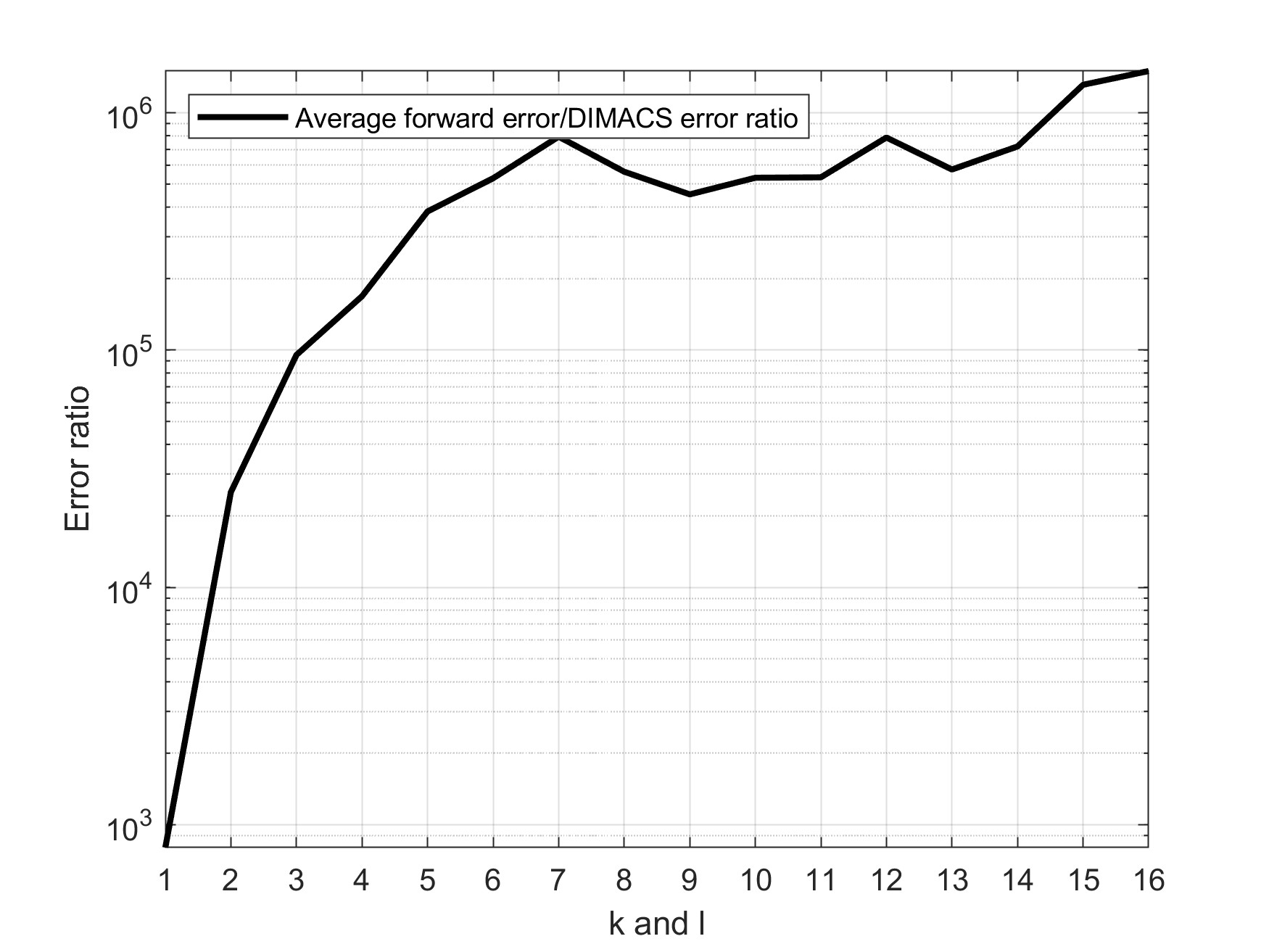}
	\end{minipage}
	\hfill
	\begin{minipage}{0.48\textwidth}
		\centering
		\includegraphics[width=\textwidth]{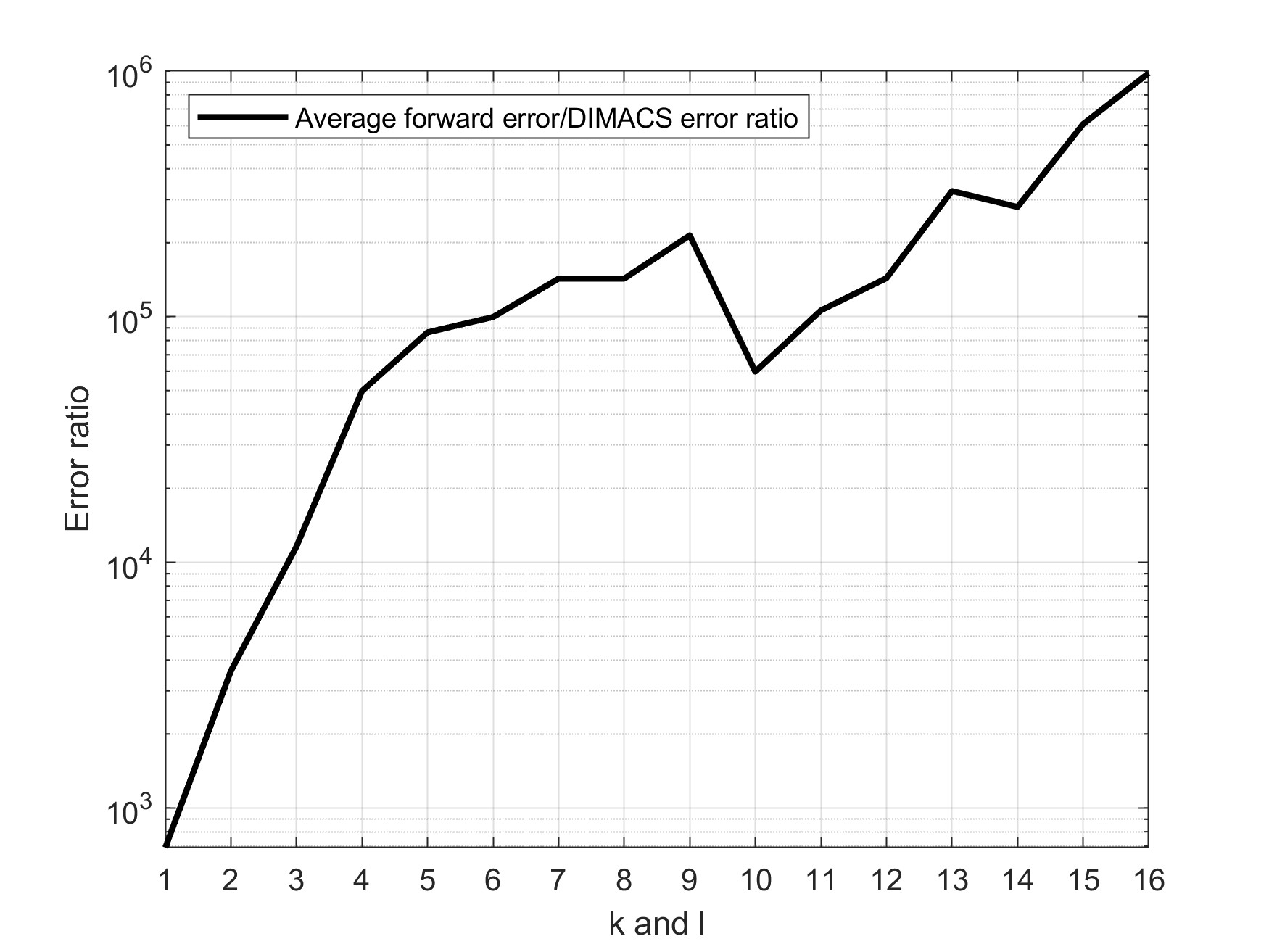}
	\end{minipage}
	\caption{Average ratios, $n=20, m=40$: nonSlater (left) and Slater (right).}
	\label{fig:avg-ratio-20-40}
\end{figure}

\co{
\newpage

\begin{figure}[H]
	\centering
	\begin{minipage}{0.48\textwidth}
		\centering
		\includegraphics[width=\textwidth]{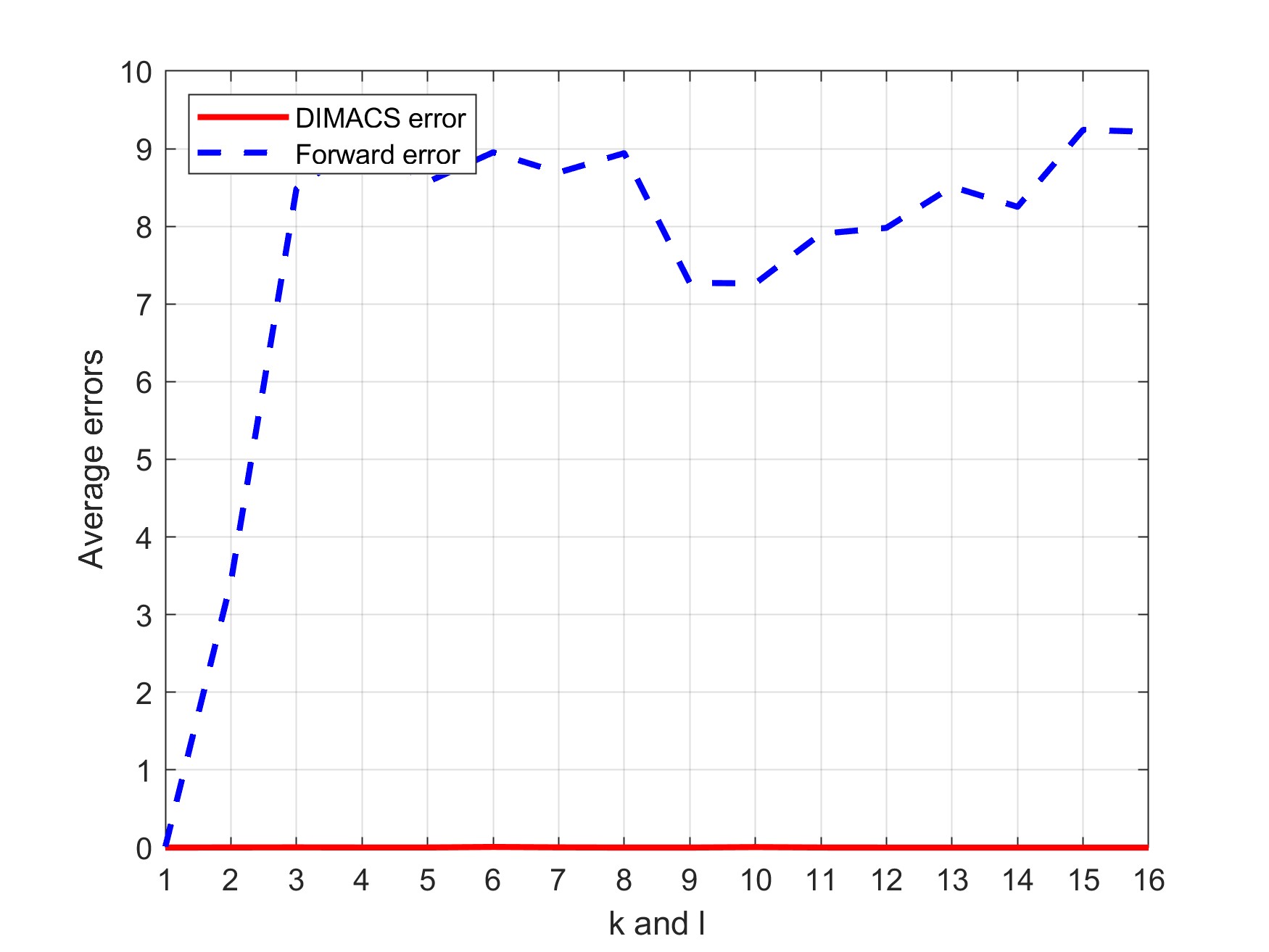}
	\end{minipage}
	\hfill
	\begin{minipage}{0.48\textwidth}
		\centering
		\includegraphics[width=\textwidth]{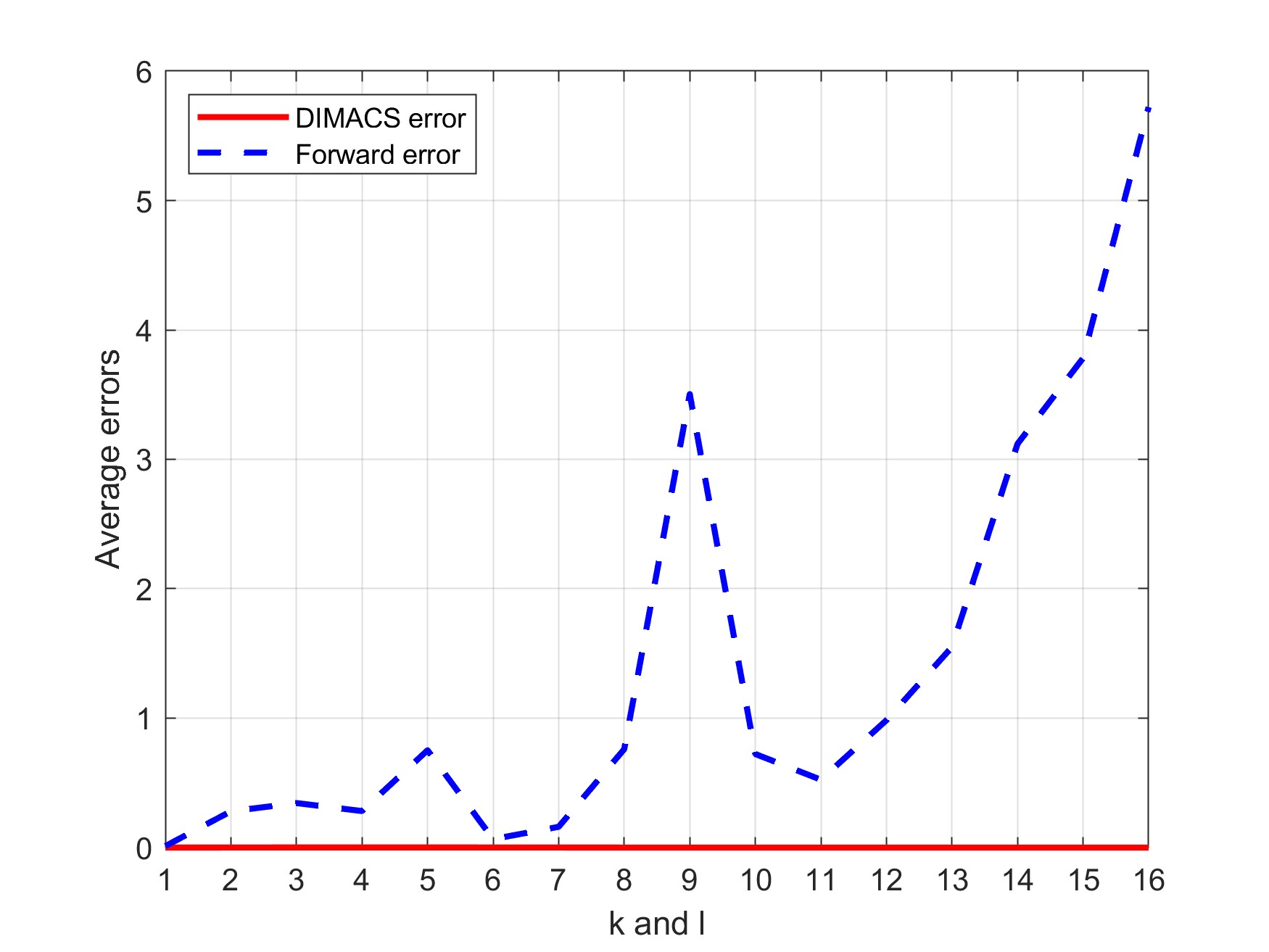}
	\end{minipage}
	\caption{Average errors, $n=20, m=50$: nonSlater (left) and Slater (right).}
	\label{fig:avg-error-20-50}
\end{figure}

\begin{figure}[H]
	\centering
	\begin{minipage}{0.48\textwidth}
		\centering
		\includegraphics[width=\textwidth]{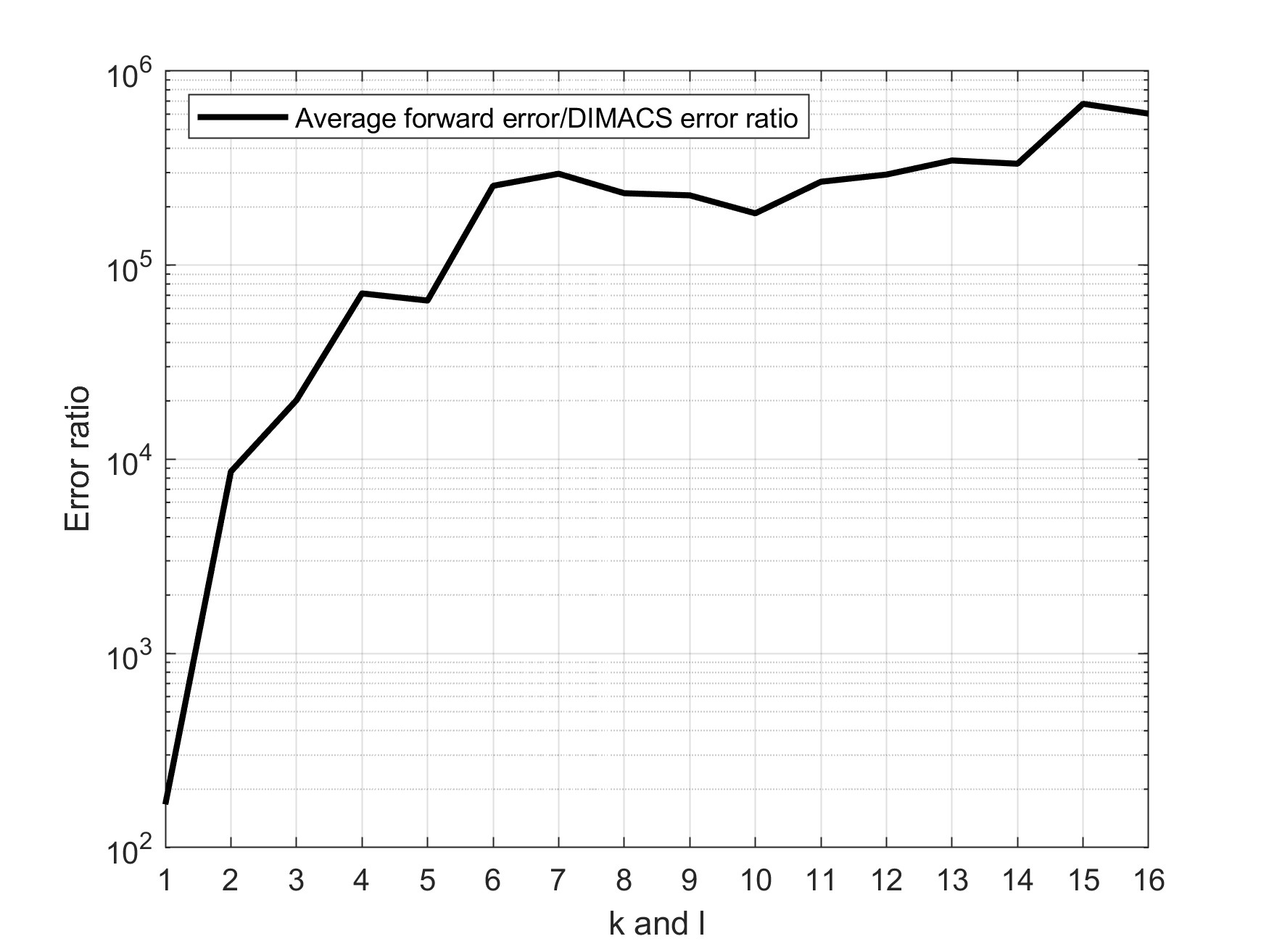}
	\end{minipage}
	\hfill
	\begin{minipage}{0.48\textwidth}
		\centering
		\includegraphics[width=\textwidth]{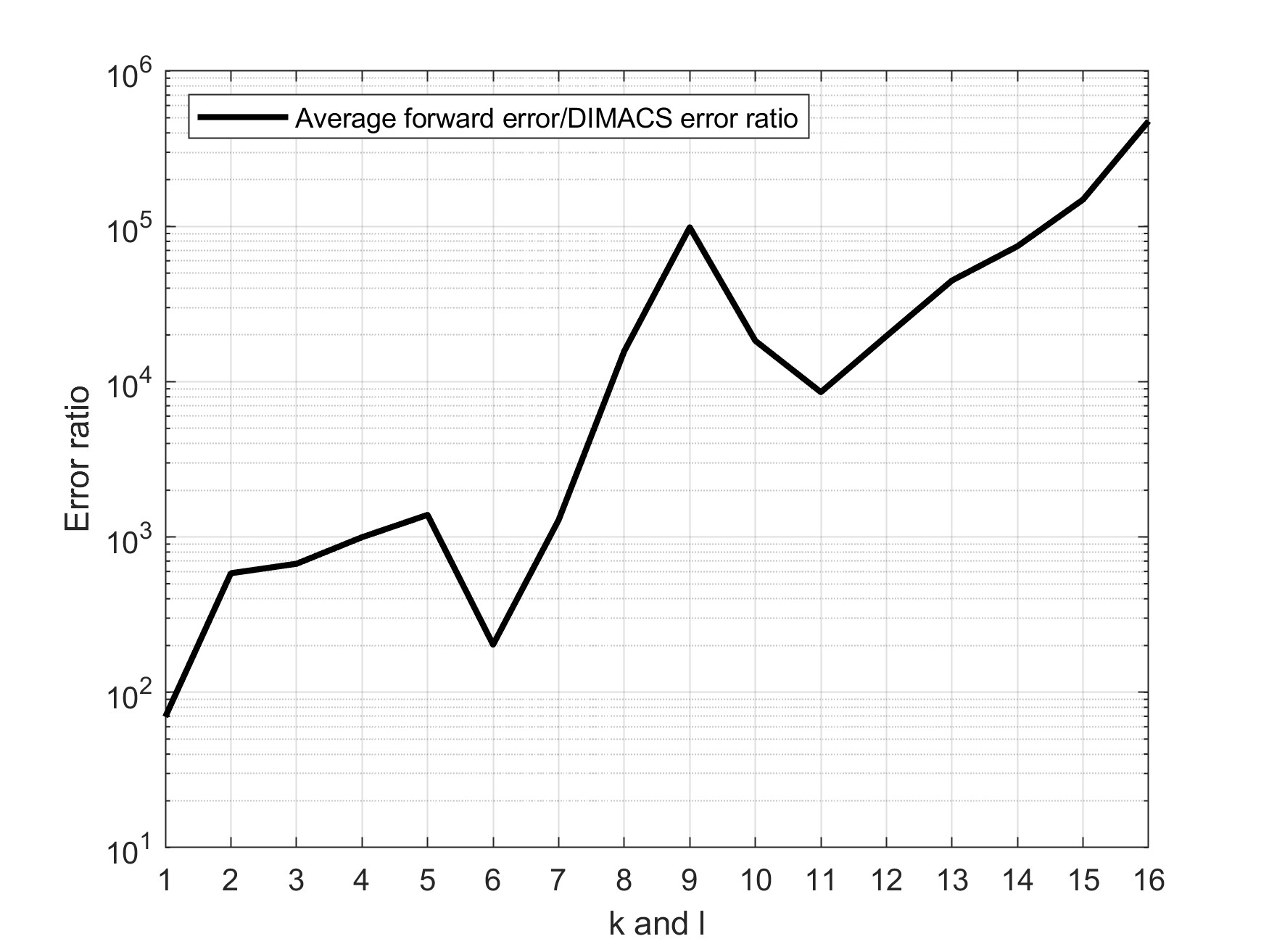}
	\end{minipage}
	\caption{Average ratios, $n=20, m=50$: nonSlater (left) and Slater (right).}
	\label{fig:avg-ratio-20-50}
\end{figure}
}

\begin{figure}[H]
	\centering
	\begin{minipage}{0.48\textwidth}
		\centering
		\includegraphics[width=\textwidth]{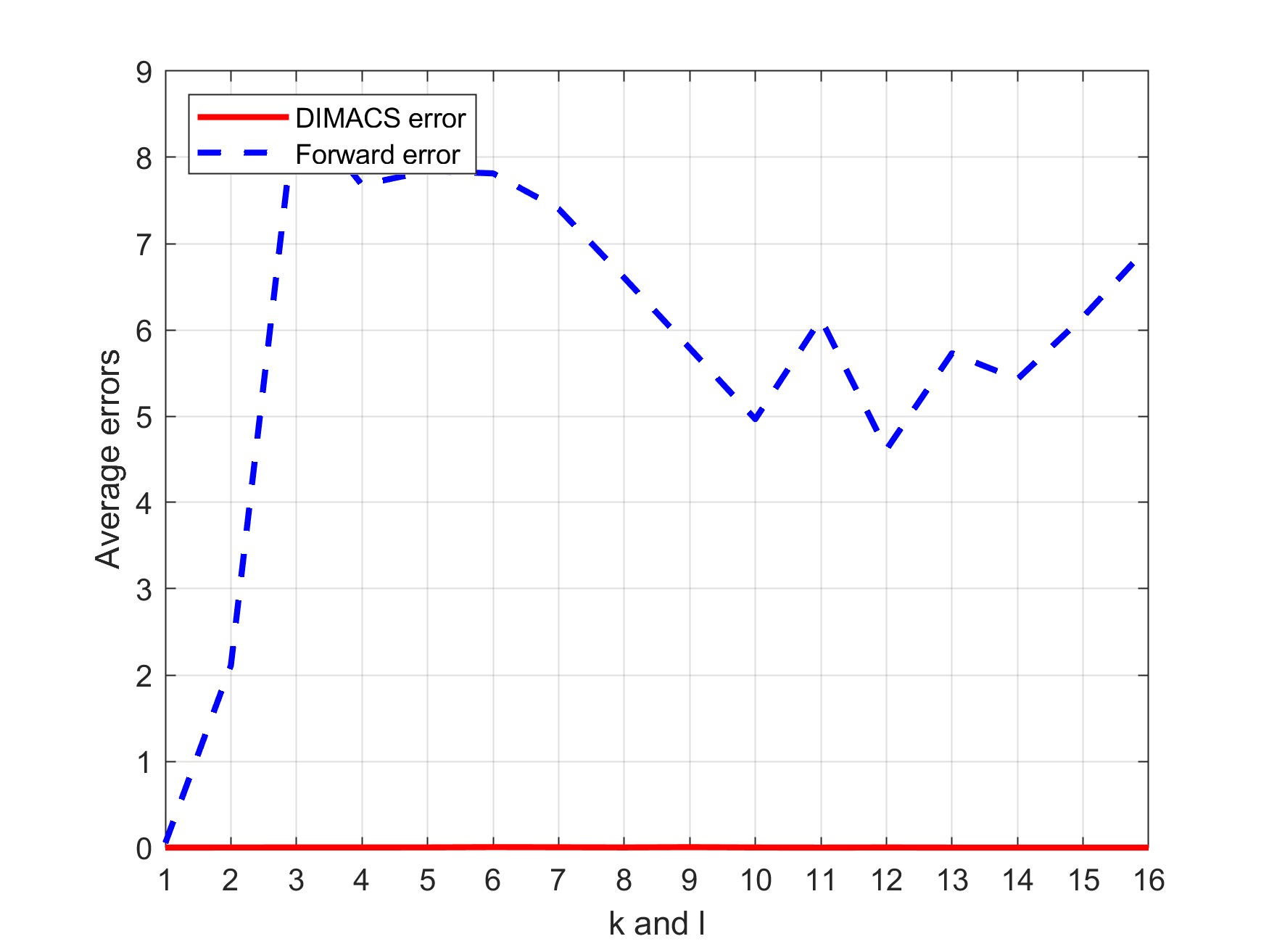}
	\end{minipage}
	\hfill
	\begin{minipage}{0.48\textwidth}
		\centering
		\includegraphics[width=\textwidth]{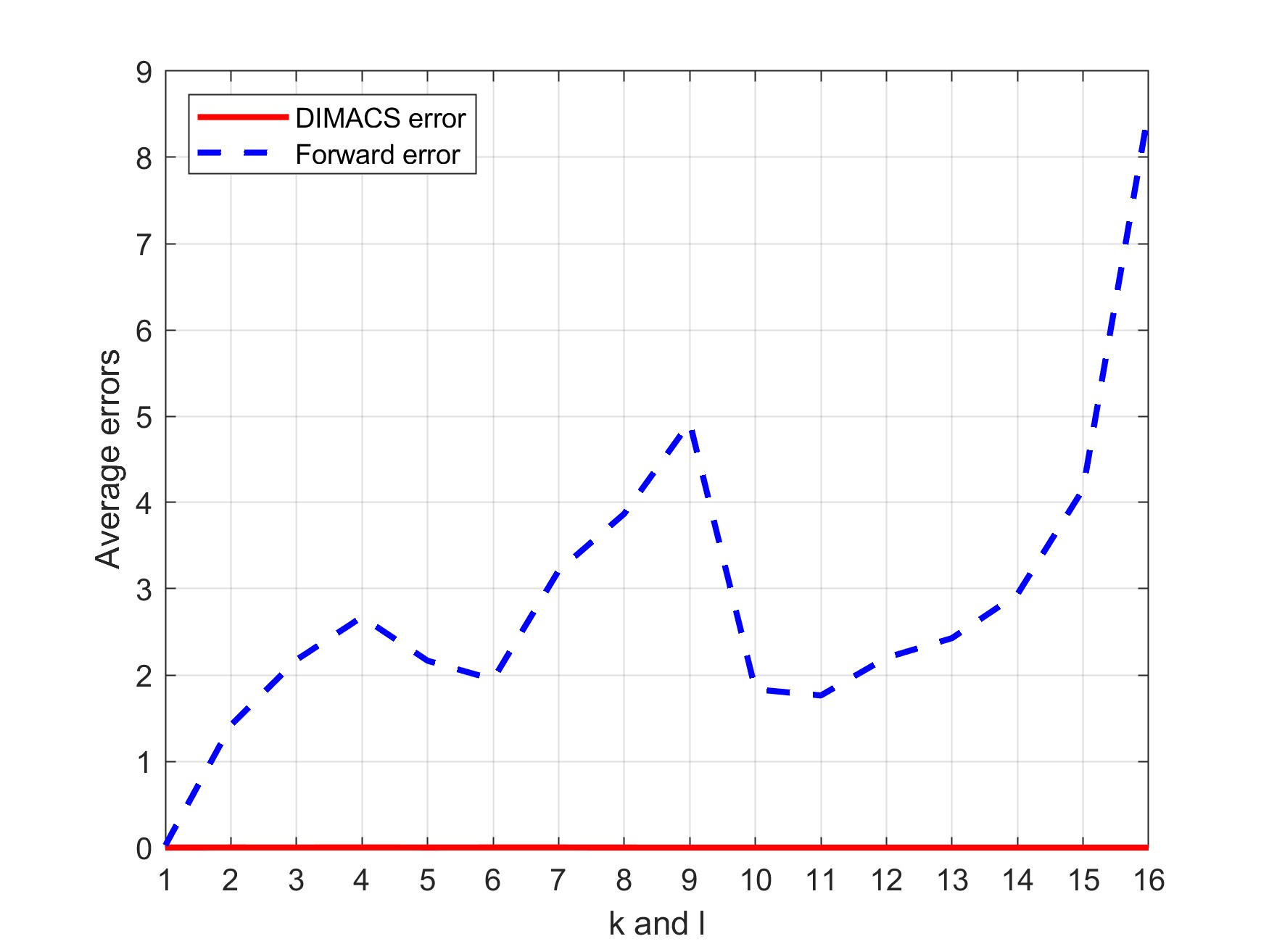}
	\end{minipage}
	\caption{Average errors, $n=20, m=60$: nonSlater (left) and Slater (right).}
	\label{fig:avg-error-20-60}
\end{figure}

\begin{figure}[H]
	\centering
	\begin{minipage}{0.48\textwidth}
		\centering
		\includegraphics[width=\textwidth]{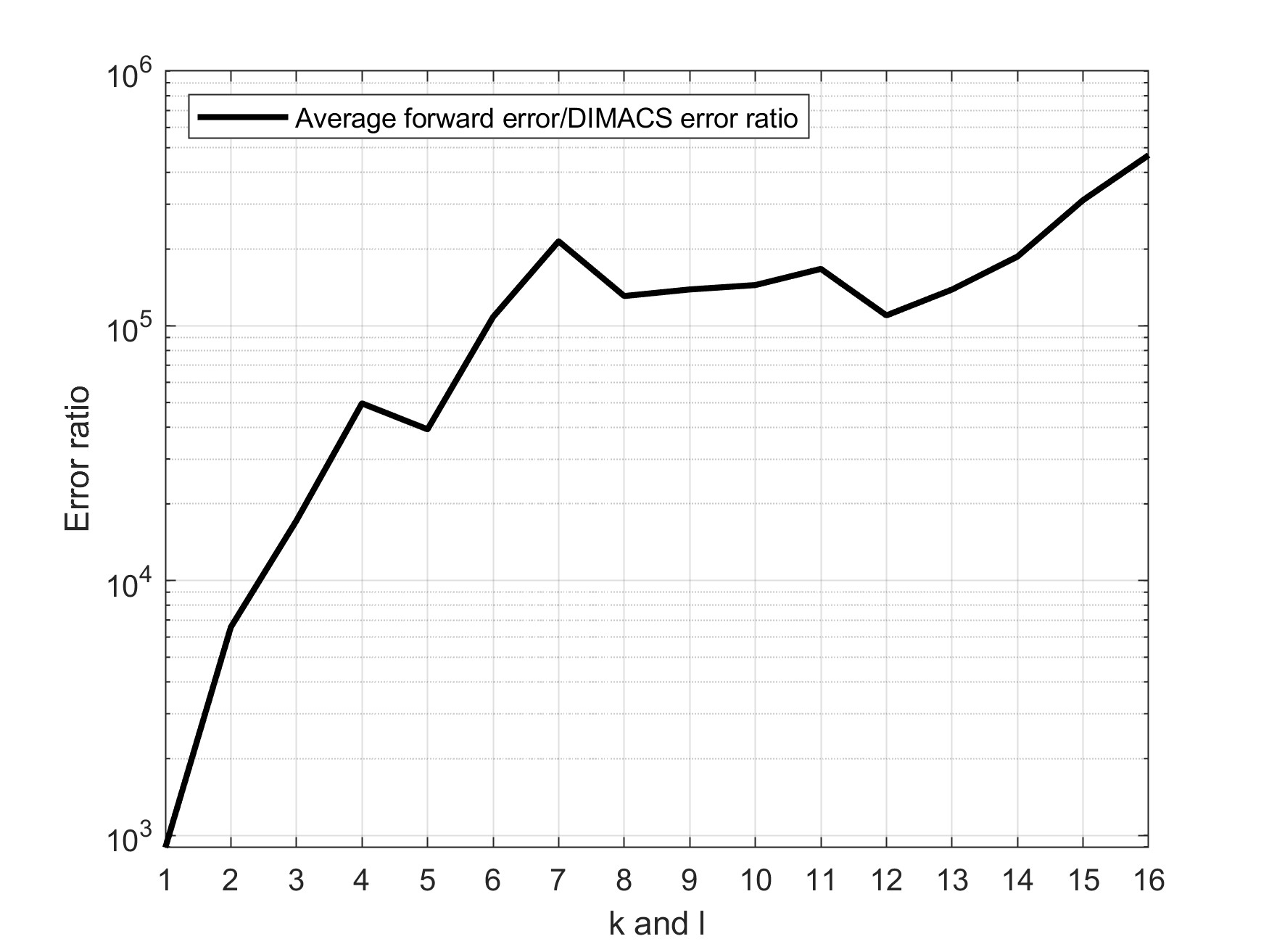}
	\end{minipage}
	\hfill
	\begin{minipage}{0.48\textwidth}
		\centering
		\includegraphics[width=\textwidth]{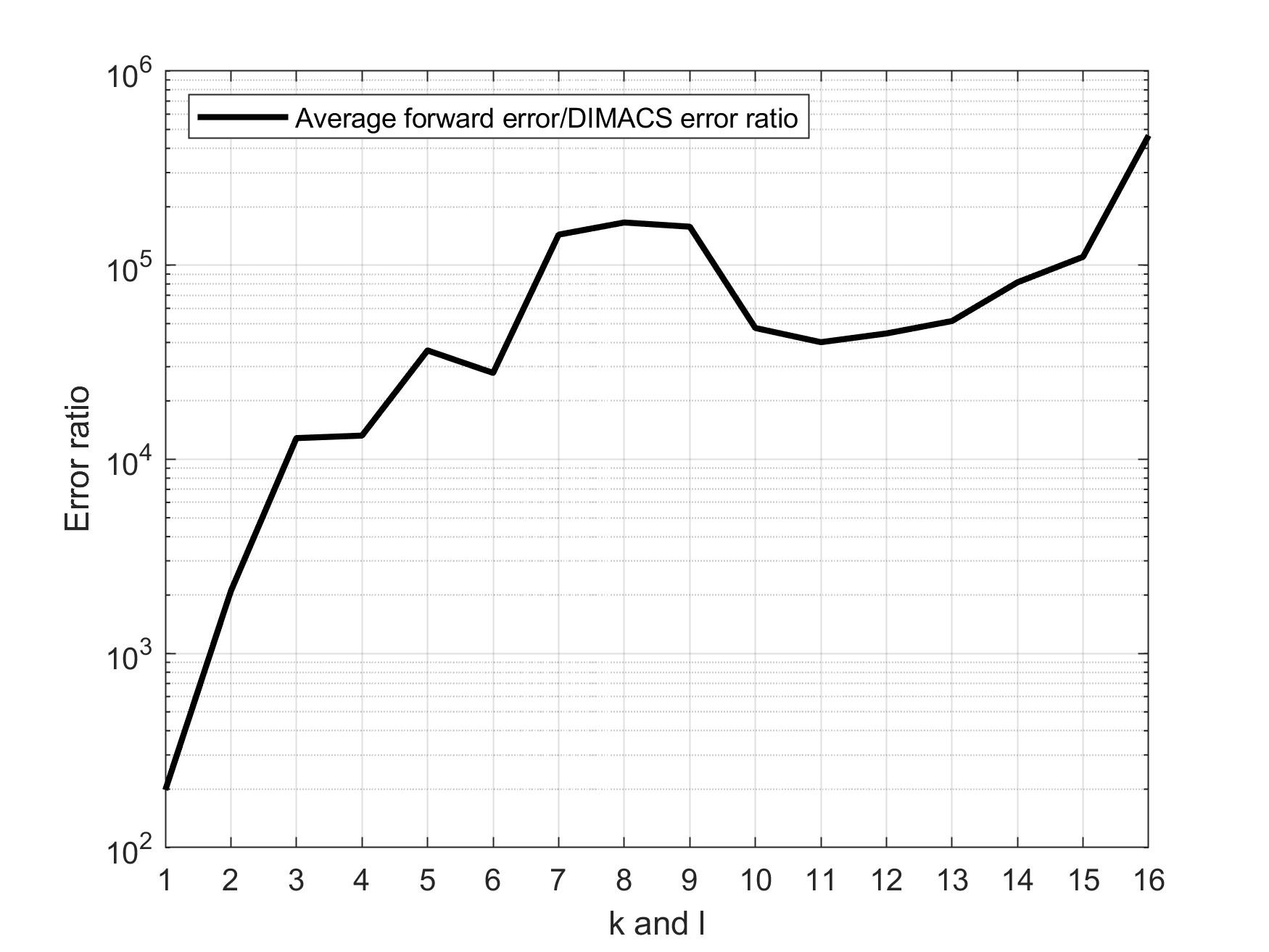}
	\end{minipage}
	\caption{Average ratios, $n=20, m=60$: nonSlater (left) and Slater (right).}
	\label{fig:avg-ratio-20-60}
\end{figure}

\subsection{Sturm structure}
\label{subsection-computational-results-Sturm}

Here we present our results with variations of Sturm's instance, which is Example 2 in \cite{Sturm:00}.
The version with $n=4$ has been studied in Examples \ref{example-sturm}, \ref{example-sturm-2}, \ref{example-sturm-3}, and 
\ref{example-sturm-4}.

To describe these instances for general $n, \,$ we need a definition. For a fixed $n$ we denote by $E_{ij}$ the $(i,j)$th symmetric unit matrix, in which the 
$(i,j)$ and $(j,i)$ elements are $1$ and all other elements are zero. The value of $n$ will always be clear from the context.

Then we consider the instance
\begin{equation}\label{Sturm}  
	\begin{array}{rl} 
		\inf  & \,\, \la E_{11}, X    \ra  \\
		s.t. & \,\, \la A_i, X \ra \, =  0  \; \text{for} \, i=1, \dots, n-2  \\
		& \,\, X \succeq 0, 
	\end{array} \tag{$Sturm$} 
\end{equation}
where for all $i$ we have
\begin{equation} \label{eqn-Sturm-define-Ai} 
A_i = 2 E_{i+1,i+1} - E_{i,n}. 
\end{equation} 
\co{
First we examine whether \eqref{Sturm} and its dual are Slater.
As to \eqref{Sturm}, the answer is "yes": first we  set $x_{22}, x_{33}, \dots, x_{n-1,n-1}$ to be positive; then 
 the last row and column of 
$X$ to make sure the constraints are satisfied; and $x_{11}$ and $x_{nn}$ large enough to ensure positive definiteness.
However, the dual is not Slater: in any  $Z$ which is  feasible in it, the $(n,n)$ element is zero, hence all $y_i$ must be zero.
}
A calculation, similar to the one we did in Example \ref{example-sturm}
shows  that in a primal optimal solution $X$ all elements, except possibly in the $(n,n)$ position, are zero.
Also, in a dual optimal solution $y=0.$

So a pair of optimal solutions is 
\begin{equation} \label{eqn-Xstar-Zstar-Sturm} 
	X^* = \alpha E_{n,n}, \, Z^* = E_{11}, \, \text{where} \, \alpha \geq 0.
\end{equation}
We see that $(Z^*, A_1, \dots, A_{n-2})$ is a regular facial reduction sequence that certifies maximum rank of $X^*.$
To certify that $Z^*$ is indeed a maximum rank optimal dual solution, we set 
\begin{equation} \label{eqn-Sturm-define-Yj} 
Y_j = E_{n-j, n-j}  + E_{n-j-1,n} \,\, \text{for} \, j=1, \dots, n-2. 
\end{equation}
We see that all the $Y_j$ and also $X^*$ have zero inner product with all the $A_i$ and $C = E_{11}.$ 
Thus, they also have zero inner product with any $Z$ feasible in the dual of \eqref{Sturm}.
Hence $(X^*, Y_1, \dots, Y_{n-2})$ is a regular facial reduction sequence, which hence certifies maximum rank of 
$Z^*.$ 

Thus, \eqref{Sturm} with the $Y_j$ is in the normal form presented in Theorem \ref{thm-main}.

We study this instance in detail, since it is classical: it   
 shows that Sturm's error bound \eqref{eqn-sturm-bound} is best possible.
 Indeed, let us consider 
 the feasible set of \eqref{Sturm} with the equation $\la E_{11}, X \ra =0$ attached.
  On the one hand, Theorem \ref{thm-singularity-degree} implies that 
 the singularity degree of this system is exactly $n-1.$ 
 On the other hand, Sturm \cite{Sturm:00} constructed an approximate solution with  $\epsilon$ constraint violation, 
 and a forward error of $\epsilon^{1/2^{n-1}}.$ 
 
The questions we address are:
\begin{itemize}
	\item Suppose we 
	  add equations to \eqref{Sturm}, while making sure that 
	  \begin{itemize}
	  	\item  $X^*$ and $Z^*$ given in \eqref{eqn-Xstar-Zstar-Sturm} remain 
	  	maximum rank optimal solutions;
	  	\item and the same regular facial reduction sequences certify their maximum rank.
	  \end{itemize}
	   How does this affect the huge gap between backward and  forward errors? 
	   Yet again, does Slaterness help? 
	  \end{itemize}
To answer this question we first generated variants of \eqref{Sturm} by Algorithm \ref{algo-nonstrict-SDP}.
Given that the $A_i$ and $Y_j$ above satisfy the \eqref{eqn-base} equations, 
we only had to run  Step \ref{algo-nonstrict-SDP-extend}, and the part of \ref{algo-nonstrict-SDP-set-bi} which sets the $b_i.$ 

We chose $n=10$ and $n=20.$ 
The smallest value for $m$ we considered was $m = n-2,$ i.e., in this case we do not add any equations to 
\eqref{Sturm}. The largest value of $m$ we considered was $m=\lfloor n(n+1)/3 \rfloor.$ 
This maximum value of $m$ \co{$m=\lfloor n(n+1)/3 \rfloor$}
 allows us to  generate a good portion of possible equations in a space of dimension 
$n(n+1)/2, \,$ while keeping the spread of $\svec(\A)$ small. 
The resulting instances belong in the nonSlater category, i.e., neither primal, nor dual are guaranteed to be Slater.
(Note that \eqref{Sturm} satisfies Slater's condition, while its dual does not. But adding more constraints may change 
Slaterness of both the primal and dual.) 

We then generated variants which satisfy Slater's condition on both sides,
by Algorithm \ref{algo-nonstrict-Slater-SDP}.
We chose 
$$
X^* = \alpha E_{nn}, \, Z^* = \alpha E_{11}, \, \bX = \bZ = I, \, \text{where} \, \alpha = 1 + \dfrac{n-2}{2}.
$$
These satisfy \eqref{sosickofthis}.  
 We chose the $A_i$ defined in \eqref{eqn-Sturm-define-Ai} and the $Y_j$ defined in 
\eqref{eqn-Sturm-define-Yj}. These satisfy the \eqref{eqn-base} equations, so we  start the algorithm at 
 Step \eqref{step:three}\ref{step:three:a}.

We found that the $5$th DIMACS errors were all at least  $- 5 \times 10^{-8}.$
SDPT3 returned a $0$ termination code on all instances.

Tables \ref{tab:worstdimacserror:Sturm}, \ref{tab:worstforwarderror:Sturm}, and \ref{tab:worstratio:Sturm} 
report the worst DIMACS, worst forward errors, and the worst ratios,
just like  in the previous subsection. For example, 
in table \ref{tab:worstdimacserror:Sturm} the upper left entry 
	$(35,\;8,\;2.6207 \times 10^{-4})$ means: when $n=10, \,$ the worst DIMACS error was achieved when 
	$m=35$ and $k=8$ and it equaled $2.6207 \times 10^{-4}.$ (Recall that in these instances 
	$k$  always equals $n-2$).

We see that: 
\begin{enumerate}
	\item The DIMACS errors are sometimes of the order of $10^{-3}.$ 
	Again, some of these instances are challenging even when we aim to 
	minimize the DIMACS error. 
	\item The forward errors and the ratio  of the forward errors to DIMACS errors are strikingly large. The forward errors 
	are of the order $10^1$ and the ratios 
	are of the order $10^6.$

	\item Now Slaterness improves matters considerably, both in terms of  forward errors, and the ratios. 
\end{enumerate}

\co{

\begin{table}[h]
	\centering
	\begin{tabular}{|c|c|c|}
		\hline
		& $n=10$ & $n=20$ \\
		\hline
		nonSlater & $(34,8,1.0661\times 10^{-3})$ & $(131,18,5.2727\times 10^{-4})$ \\
		\hline
		Slater & $(30,8,1.2964\times 10^{-3})$ & $(131,18,2.3101\times 10^{-4})$ \\
		\hline
	\end{tabular}
	
		\caption{Worst DIMACS errors, Sturm instances}
	\label{tab:worstdimacserror:Sturm:old}
	
\end{table}
}

\begin{table}[h]
	\centering
	\begin{tabular}{|c|c|c|}
		\hline
		& $n=10$ & $n=20$ \\
		\hline
		nonSlater & $(28,8,9.3619\times 10^{-4})$ & $(134,18,1.1247\times 10^{-4})$ \\
		\hline
		Slater & $(32,8,1.4580\times 10^{-3})$ & $(135,18,1.6296\times 10^{-4})$ \\
		\hline
	\end{tabular}
	
		\caption{Worst DIMACS errors, Sturm instances}
	\label{tab:worstdimacserror:Sturm}
	
\end{table}

\co{

\begin{table}[h]
	\centering
	\begin{tabular}{|c|c|c|}
		\hline
		& $n=10$ & $n=20$ \\
		\hline
		nonSlater & $(8,8,8.6226\times 10^{0})$ & $(19,18,2.0359\times 10^{1})$ \\
		\hline
		Slater & $(10,8,1.6292\times 10^{0})$ & $(20,18,3.0578\times 10^{0})$ \\
		\hline
	\end{tabular}
	
		\caption{Worst forward errors, Sturm instances}
	\label{tab:worstforwarderror:Sturm:old}
	
\end{table}

}
\begin{table}[h]
	\centering
	\begin{tabular}{|c|c|c|}
		\hline
		& $n=10$ & $n=20$ \\
		\hline
		nonSlater & $(8,8,8.6226\times 10^{0})$ & $(21,18,1.8545\times 10^{1})$ \\
		\hline
		Slater & $(10,8,1.6292\times 10^{0})$ & $(20,18,3.0501\times 10^{0})$ \\
		\hline
	\end{tabular}

	\caption{Worst forward errors, Sturm instances}
	\label{tab:worstforwarderror:Sturm}
	
\end{table}

\co{

\begin{table}[h]
	\centering
	\begin{tabular}{|c|c|c|}
		\hline
		& $n=10$ & $n=20$ \\
		\hline
		nonSlater & $(9,8,9.0234\times 10^{5})$ & $(19,18,4.7754\times 10^{6})$ \\
		\hline
		Slater & $(23,8,1.0377\times 10^{5})$ & $(69,18,4.5705\times 10^{5})$ \\
		\hline
	\end{tabular}
	
		\caption{Worst ratios, Sturm instances}
	\label{tab:worstratio:Sturm:old}
	
\end{table}

}
\begin{table}[h]
	\centering
	\begin{tabular}{|c|c|c|}
		\hline
		& $n=10$ & $n=20$ \\
		\hline
		nonSlater & $(9,8,8.9385\times 10^{5})$ & $(21,18,4.2007\times 10^{6})$ \\
		\hline
		Slater & $(14,8,9.3642\times 10^{4})$ & $(71,18,4.4131\times 10^{5})$ \\
		\hline
	\end{tabular}
	
		\caption{Worst ratios, Sturm instances}
	\label{tab:worstratio:Sturm}
	
\end{table}

\FloatBarrier

In figures \ref{fig:Sturm10} and \ref{fig:Sturm20} we show the DIMACS and forward errors as a function of $m.$ 
As expected, the forward error decreases as $m$ increases. However, even when $m$ is large, the forward errors are still substantial.

	\begin{figure}[H]
	\centering
	
	\begin{subfigure}[t]{0.48\textwidth}
		\centering
		\includegraphics[width=\textwidth]{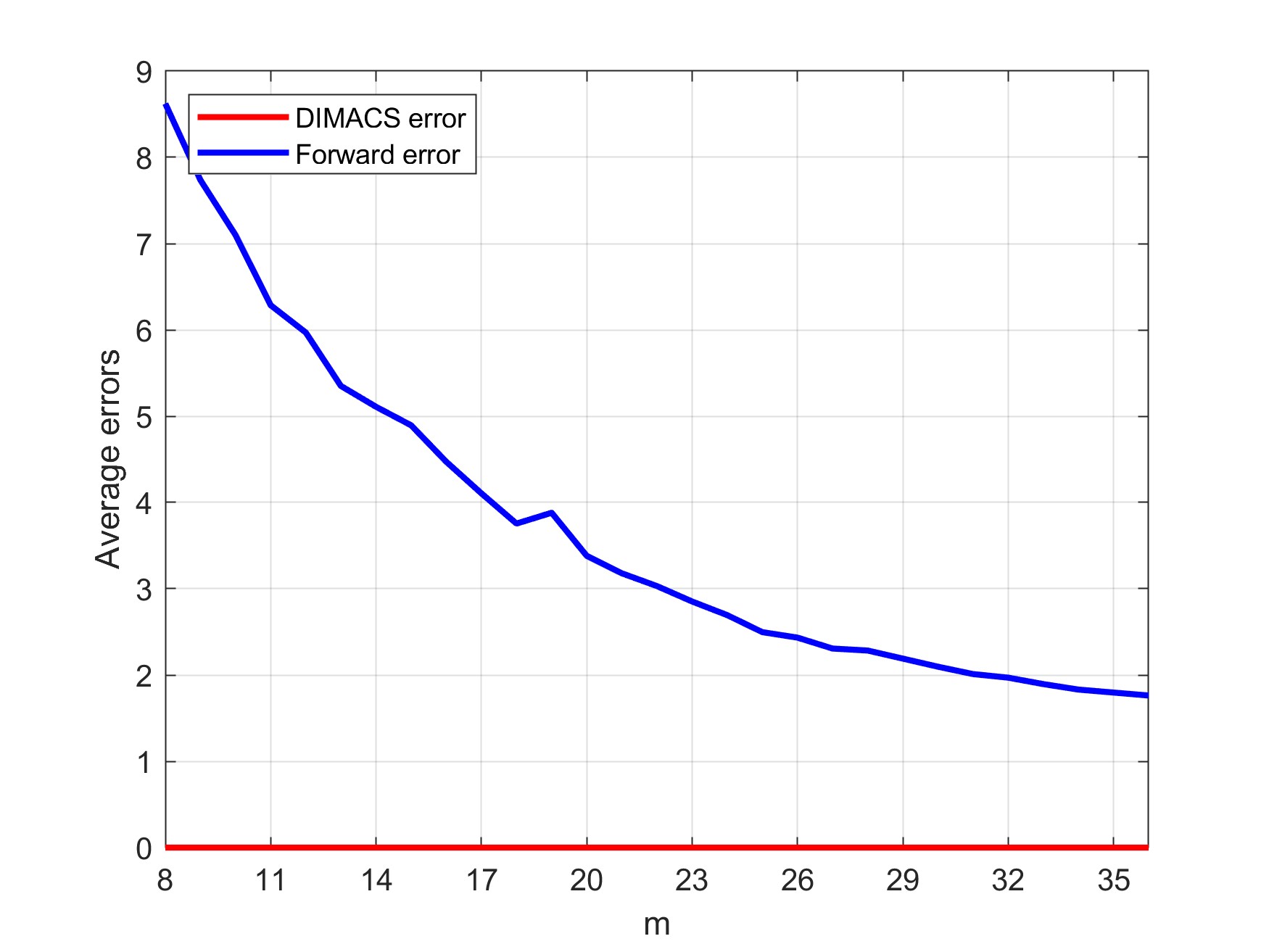}
		\end{subfigure}
	\hfill
	\begin{subfigure}[t]{0.48\textwidth}
		\centering
		\includegraphics[width=\textwidth]{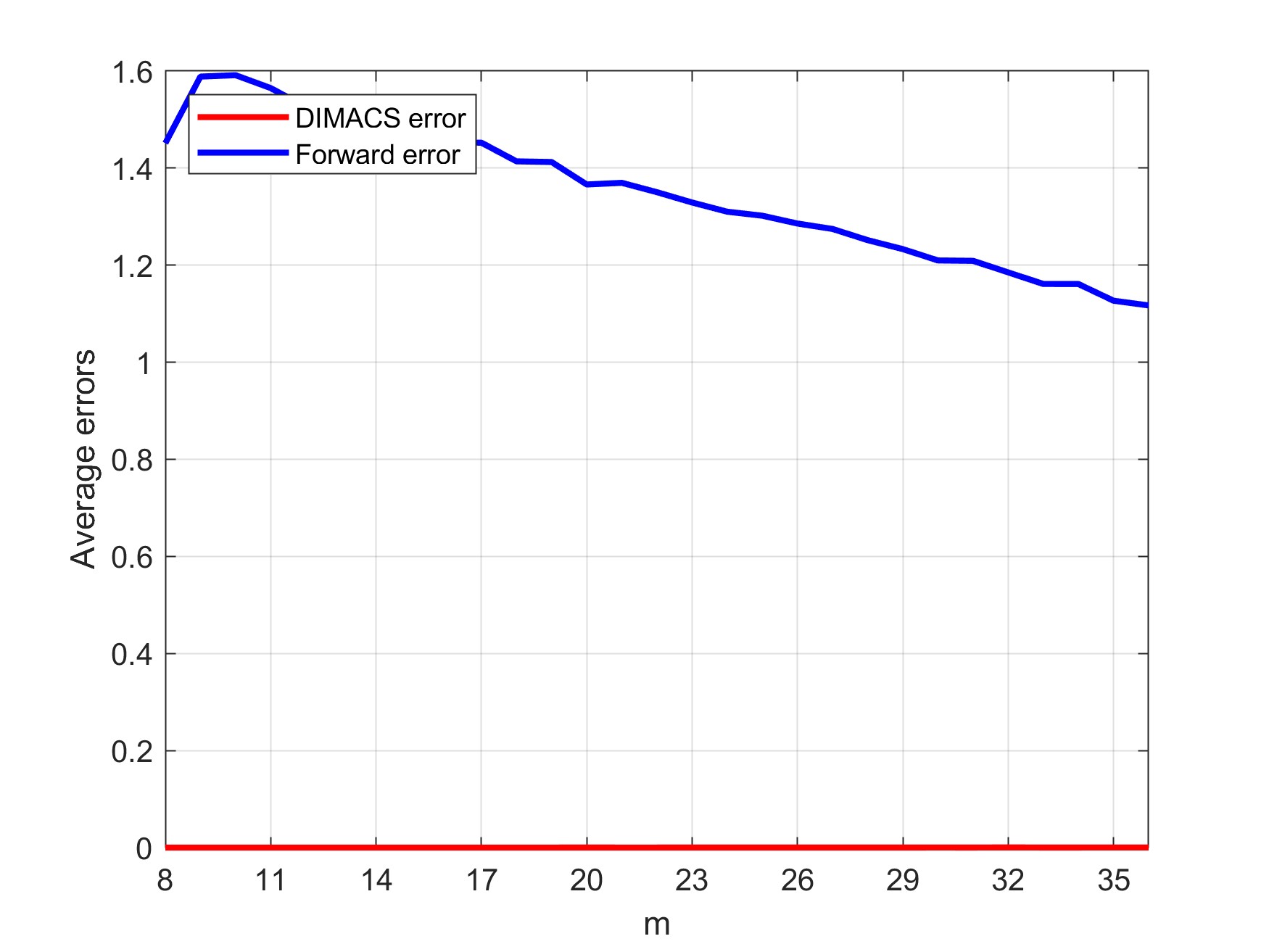}
		\end{subfigure}
	\caption{Average errors in SDPs with Sturm structure, nonSlater vs Slater, $n=10$}
	\label{fig:Sturm10} 
\end{figure}

	\begin{figure}[H]
	\centering
	
	\begin{subfigure}[t]{0.48\textwidth}
		\centering
		\includegraphics[width=\textwidth]{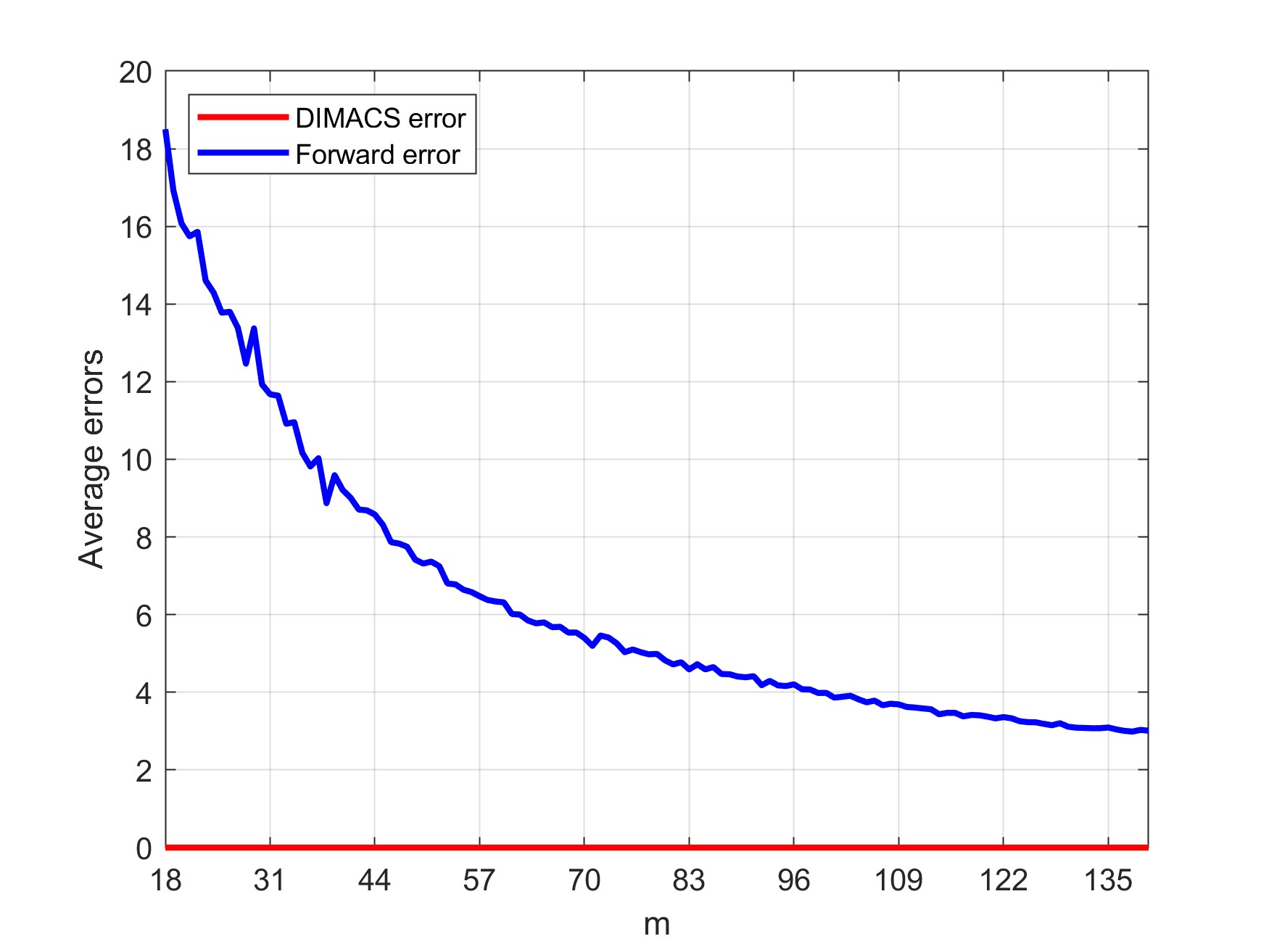}
			\end{subfigure}
	\hfill
	\begin{subfigure}[t]{0.48\textwidth}
		\centering
		\includegraphics[width=\textwidth]{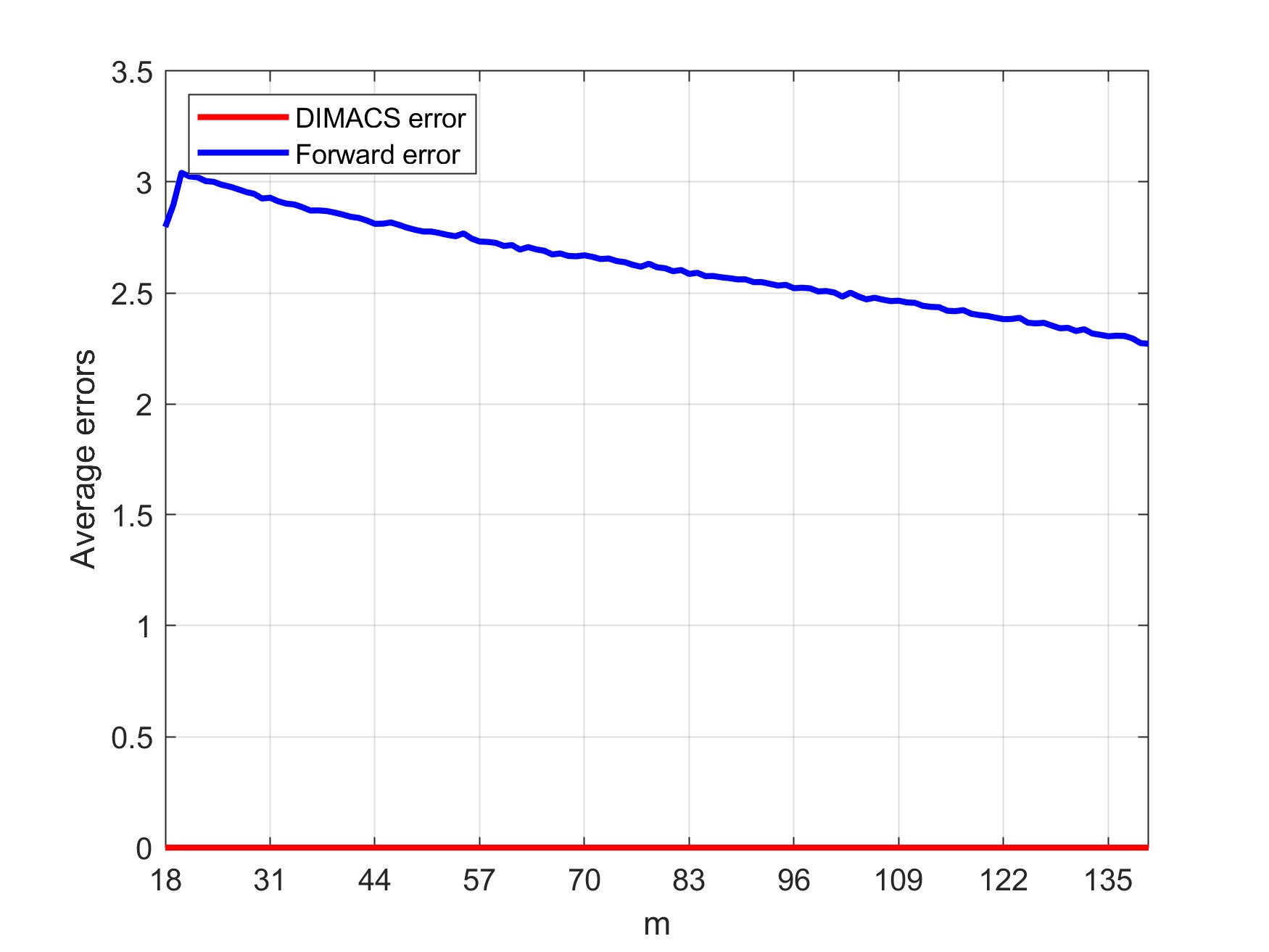}
		\end{subfigure}
\caption{Average errors in SDPs with Sturm structure, nonSlater vs Slater, $n=20$}
\label{fig:Sturm20} 

\end{figure}

\subsection{Comparison of the interval partition and Sturm instances} 

In this subsection we compare the instances with the interval partition and Sturm structures.
\co{Recall that (i) the Sturm instances are variants of a classical example,
with provably large theoretical worst case forward error, and (ii) the interval partition instances are much less structured.}
The main and intriguing finding is that for comparable values of $n$ 
the former may have forward errors which are as large as the forward errors in the latter.

We first look at the "nonSlater" row of tables \ref{tab:worstforwarderror:random} and \ref{tab:worstforwarderror:Sturm}.
We see that the worst forward error is larger in the interval partition instances.

Next we look at Figures \ref{fig:avg-error-20-20} and \ref{fig:Sturm20}.
The "nonSlater" instances with interval partition, and $n=20, m = 20, \,$ and $k=3$ have average  forward errors more than $20;$ 
whereas the Sturm instances, and $ n = 20, m = 18$ 
have average  forward errors about $18.$ 

Even more intriguing is to look at the singularity degree of the enhanced system \eqref{p-Zstar}.
When $n=m=20, \,$ and $k=3, \,$ 
then  \eqref{p-Zstar} associated with the interval partition instances has singularity degree {\em at most} $3.$
On the other hand, \eqref{p-Zstar} associated with the Sturm  instances has singularity degree {\em exactly} $18.$
The former instances are also less structured: the nonzero entries were generated randomly, as we described in Subsection
\ref{subsection-computational-results-random}. Again, we refer to Figure \ref{fig:sparsity-interval}. In contrast,
in the Sturm instance the $A_i$ have just three nonzeros.

In conclusion,  the instances with the interval partition structure and much smaller singularity degree
 are much less structured; however, some of them turned  out to be as challenging,
as the classical worst case SDP of Sturm.

\subsection{Sanity check} 

To make sure that our SDPs are difficult precisely because of lack of strict complementarity, we ran 
a simple experiment. We fixed $n=20$ and generated SDPs for $m=10,20, \dots, 60.$ For each $m$ we generated $50$ SDPs.
In each SDP  
we chose an  $X^*$ optimal in \eqref{p} and $Z^*$ optimal in the dual the same way we chose them in Algorithms 
\ref{algo-nonstrict-SDP} and \ref{algo-nonstrict-Slater-SDP}. We chose  their ranks at random, but made sure that 
their sum was $20, \, $ i.e., strict complementarity was satisfied. 
We also chose the spread of all $A_i$ to be $10^6, \,$ i.e., the upper bound on the spread of both of $\svec(\A)$ and $\svec(Y)$ we used both for the interval partition and the Sturm instances.

The worst DIMACS error over  these runs turned out to be $\leq 4 \times 10^{-3}$ and 
the worst forward error $\leq 2 \times 10^{-6}.$

\subsection{Data availability}

The interval SDPs and Sturm SDPs are available at 

\url{https://www.dropbox.com/scl/fo/vgwreu1hdtw2usfvgrqo2/AEX4sGrp3fy5FwPvSpHdrxw?rlkey=olh8t8k96zpyutr30ldaczftr&dl=0}

In the Interval\_SDPs.mat file there are $8$ cell arrays, each of dimension $40 \times 400.$ They are called
\begin{itemize}
	\item \texttt{nonSlaterworstdimacsSDPs},
	\item \texttt{nonSlaterworstforwardSDPs},
	\item \texttt{nonSlaterworstratioSDPs},
	\item \texttt{nonSlaterunsolvedSDPs},
	\item \texttt{SlaterworstdimacsSDPs},
	\item \texttt{SlaterworstforwardSDPs},
	\item \texttt{SlaterworstratioSDPs},
	\item \texttt{SlaterunsolvedSDPs},
\end{itemize}
respectively. The only $(n,m)$ pairs for which the 
cell arrays is nonempty, are the ones for which we ran a test. 
That is,  $n=10$ and $n=20; $ and for each $n, \, $ 
$m \in \{n, n+10, n+20, \dots, 3n \}.$ 

We first explain how the data is stored for the cell array \texttt{nonSlaterworstdimacsSDPs.}  For that, we recall that in these instances we chose the rank of both $X^*$ and $Z^*$ to be two.
For each $(n,m)$ pair for which this cell array is nonempty, 
\texttt{nonSlaterworstdimacsSDPs\{n,m\}} is an array of length $n.$ 
 
  For $k \in \{1, 2, \dots, n-4 \}$ the field \texttt{nonSlaterworstdimacsSDPs\{n,m\}(k)} is a Matlab structure. For a fixed $k$ 
  the structure stores the "nonSlater" SDP 
  which gave the worst, i.e., largest DIMACS error among  the $20$ runs \footnote{Thus it would suffice to have the length of 
  \texttt{nonSlaterworstdimacsSDPs\{n,m\}} to be $n-4.$}. The  fields of the Matlab structure  are: 
\begin{itemize}
	\item[]  
\texttt{A}, \texttt{b}, \texttt{C}, \texttt{Y}, \texttt{T},
\texttt{Xstar}, \texttt{Zstar}, \texttt{ystar},
\texttt{Xbar}, \texttt{Zbar}, \texttt{P}, \texttt{Q},
\texttt{cleanA}, \texttt{cleanC}, \texttt{cleanY},
\texttt{cleanXstar}, \texttt{cleanZstar},
\texttt{cleanXbar}, \texttt{cleanZbar.}
\end{itemize}

We explain \co{the meaning of } these below:
\begin{itemize}
	\item The field \texttt{A} is a cell array, and \texttt{A\{i\}} contains $A_i$ for $i=1, \dots, m.$ 
	Similarly, the field \texttt{Y} stores the $Y_j$ matrices.
	\item The fields 
	\texttt{b}, \texttt{C}, \texttt{Xstar}, etc. store $b, C, X^*, $ etc. respectively. 
	Now the fields 
	\texttt{barX} and \texttt{barZ} are redundant. 
		\item  The fields \texttt{P} and \texttt{Q} store the structure of $(Z^*, A_1, \dots, A_k), \, $ and the structure of 
	$(X^*, Y_1, \dots, Y_k), \,$ respectively. 
	\item The fields \texttt{clean A,} \texttt{cleanC, } etc. contain the $A_i, C, \,$ etc. before they were rotated.
	\item The field \texttt{T} contains the matrix by which we rotated the original $A_i$ and $Y_j $ (which are stored in 
	\texttt{cleanA} and \texttt{cleanY})  to get the  $A_i$ and $Y_j $ (which are stored in 
	\texttt{A} and \texttt{Y}).
\end{itemize}
 The explanation for the other cell arrays 
 is analogous. 
 Thus, 
 \begin{itemize}
 	\item  	\texttt{nonSlaterworstforwardSDPs} stores the "nonSlater" SDPs with the worst, i.e., largest forward error; 
 \item \texttt{nonSlaterworstratioSDPs} stores the  "nonSlater" SDPs with the worst, i.e., largest forward error/backward error ratio; and 
 \item \texttt{nonSlaterunsolvedSDPs}  stores the   "nonSlater" SDPs which were not solved by SDPT3, i.e., the termination code was not zero.
 Among these, for each $(n,m)$ pair and each $k$ we store at most one SDP. 
  \end{itemize}
 
 The fields in the "Slater" instances have an entirely analogous explanation. 
 The fields 
 \texttt{barX} and \texttt{barZ} store the $\bX$ and $\bZ, \,$ respectively.

 In the Sturm\_SDPs.mat we store the Sturm instances. All cell arrays have the same meaning as in the 
 Interval partition instances, with the following caveats: 
 \begin{itemize}
 	\item Since now $k=n-2, \,$ now the \texttt{nonSlaterworstdimacsSDPs\{n,m\}(k)} structure is nonempty only when $k=n-2.$  
 	An analogous statement is true of 	the  \texttt{nonSlaterworstforwardSDPs}, and the 
 	\texttt{nonSlaterworstratioSDPs} cell arrays; and the "Slater" variants.
 	\item Given that SDPT3 now solved all the tested instances, 
 	the fields in the 
 	\texttt{nonSlaterunsolvedSDPs} and \texttt{SlaterunsolvedSDPs} cell arrays are all empty. 
 \end{itemize}

\co{

The field \texttt{A} is a cell array, and \texttt{A\{i\}} contains $A_i$ for $i=1, \dots, m.$ 
Similarly,
\texttt{Y} stores the $Y_j$ certificates. The fields 
\texttt{b}, \texttt{C}, \texttt{Xstar}, \texttt{Zstar}, \texttt{ystar},
\texttt{Xbar}, \texttt{Zbar} store $b, C, X^*, $ etc. respectively. The fields \texttt{P} and \texttt{Q} store the structure of $(Z^*, A_1, \dots, A_k), \, $ and the structure of 
$(X^*, Y_1, \dots, Y_k), \,$ respectively. 

}

\co{
\begin{verbatim}
	Ed Klotz:
	Spreads < 1e+6 don?t worry me at all.   But, as seen in the second example above, other model characteristics could be problematic.
	Spreads in the 1e+6 - 1e+8 don?t worry me much.    Between Gurobi?s scaling, the numerical safeguards in the code, and the availability of the NumericFocus parameter to instruct the solver to add levels of
	caution in the computations,  these kind of spreads typically don?t cause any trouble.   But setting
	the scale flag parameter to 2 to use geometric mean scaling may be useful.
	Spreads in the 1e+8 - 1e+10 range are a bit more worrisome.
	Spreads above 1e+10 would typically prompt me to take a closer look at the model, even if it seems to be solving just fine.
	
	
\end{verbatim}

}

\section{Proof of Theorem \ref{thm-main}} 
\label{section-proof-of-theorem-1} 

In this section we prove Theorem \ref{thm-main}. We first prove two lemmas. The first, 
 Lemma 
\ref{lemma-Xstar} proves there is a reformulation that satisfies items 
\eqref{thm-main-1} and \eqref{thm-main-3} in Theorem \ref{thm-main}.
Lemma \ref{lemma-Zstar} then proves there is 
a reformulation that satisfies items 
\eqref{thm-main-2} and \eqref{thm-main-3} in Theorem \ref{thm-main}.
The most technical part of the proof is merging these two reformulations into one.

The following is a theorem of the alternative, which is the only nontrivial result we will use in this paper: 
\begin{Proposition} \label{proposition-alternative} 
	Suppose \eqref{p} is feasible. Then it fails Slater's condition $\Leftrightarrow$ 
	the system 
	\begin{equation} \label{eqn-A*y-alternative} 
		\A^* y \in \psd{n} \setminus \{ 0 \}, \,  \la b, y \ra = 0
	\end{equation}
	is feasible 
\end{Proposition}

\subsection{Lemmas} 
\label{subsection-proof-of-theorem-1-lemmas} 

\begin{Lemma} \label{lemma-Xstar} 
	 Assume that $X^*$ and $Z^*$ given in 
	\eqref{eqn-Xstar-Zstar} is a pair of optimal solutions 
	in \eqref{p}$\mhyphen$\eqref{d}.
	
	Then $X^*$ is a maximum rank  optimal solution of \eqref{p} $\Leftrightarrow$  there is a nonnegative integer $k$ and a reformulation 
	\begin{equation} \label{p-prime-lemma1} 
		\begin{array}{rcl} 
			\A^\prime X & = & b^\prime \\  
			X & \succeq & 0
		\end{array}  \tag{$P^\prime$} 
	\end{equation}
	which satisfies \eqref{thm-main-1} and \eqref{thm-main-3} in Theorem \ref{thm-main}.
\end{Lemma}
{\bf Proof}  The $\Leftarrow$ direction is just like in Theorem \ref{thm-main}. To prove $\Rightarrow,$ 
define $A_0 := Z^*, b_0 := 0$ and  
consider the semidefinite system 
\begin{equation}\label{p-opt}  
	\begin{array}{rcl} 
		\la A_i, X \ra & = & b_i \, (i=0,1,  \dots, m) \\
		X & \succeq & 0, 
	\end{array} \tag{$P_{\rm opt}$} 
\end{equation}
Clearly, $X$ is optimal in \eqref{p} iff it is feasible in \eqref{p-opt}.
Let us also define the linear operator $	\bar{\A} $ and vector 
$\bar{b} $ as  
\begin{equation} \label{eqn-repr-Aprime} 
	\begin{array}{rcl} 
		\bar{\A}  X & = &  (\la A_0, X \ra,  \dots, \la A_m, X \ra)^{\top} \\
		\bar{b} & = & (b_0, b_1, \dots, b_m)^\top. 
	\end{array} 
\end{equation}
If \eqref{p-opt} is strictly feasible, then we do not need to do anything. (In the very first step this can only happen 
if $r=n, \, s =0$).  Otherwise, we apply Proposition \ref{proposition-alternative} (with $\bar{\A},$ and $\bar{b}$ in place of 
$\A$ and $b$) and 
find $y \in \rad{m+1}$ such that 
$\bar{\A}^* y \in \psd{n} \setminus \{ 0 \}, \,  \la \bar{b}, y \ra = 0.$

We observe that 
\begin{equation} \label{eqn-barA*yX*}
	\la \bA^* y, X^* \ra = \la y, \bA X^* \ra = \la y, \bb \ra = 0.
\end{equation}
Thus the last  $r$ rows and columns of $\bA^* y$ are zero. We rotate all $A_i$ and hence 
$\bA^* y$ by a matrix of the form $\left( \begin{smallmatrix}
	U & 0 \\ 0 & I_r 
\end{smallmatrix} \right)$ so after this rotation we have 
$$
\bA^* y \, = \, \bpx \Theta  & 0 \\
0 & 0 \epx,
$$
where $\Theta $ is of order say $t, \, $ and positive definite. Note that $0 < t \leq n-r.$ 

Since $\bA^* y $ is nonzero, there is $j \in \{0, \dots, m \}$ such that $y_j \neq 0.$ We then 
perform the operation 
$$
(A_j, b_j) \leftarrow (\bA^* y, 0), 
$$
delete the $j$th equation, and delete the first $t$ rows and columns from all $A_i.$ Note that these rows and columns are all zero in $X^*.$ 
We then proceed in like fashion with this smaller system. When we are done, we just renumber the equations in the final reformulation, and add  the $()^\prime$ to the $A_i$ and $b_i.$ 

The proof is complete, if we note  the following:
\begin{enumerate}
	\item  in the very first step we can take 
	$y = (1,0, \dots, 0)^\top, \, $ so $\bA^* y = Z^*$ and no rotation is needed;
	\item in all steps the last $r$ rows and columns of 
	$\bA^* y$ are zero, and in the first step we delete 
	the first $s$ rows and columns of all matrices.
\end{enumerate}
\qed 

The algorithm we gave above deletes rows and columns of all $A_i.$ 
Next we sketch how it can be made to  reformulate the original system 
\eqref{p-opt}, we number the iterations $0, 1, 2, \dots,$ and define $y^i$ as the  
$y$ vector, and  $U_i$ the upper left minor of the rotation matrix found in 
 found in iteration $i,$ respectively. 

Then in iteration $i$ we can reformulate the original 
\eqref{p-opt} by 1) padding the $y^i$ with zeros, which correspond to deleted equations, and 2) 
using the rotation matrix 
$$
T_i = \bpx I_s & 0 & 0 \\ 0 & U_i & 0 \\ 0 & 0 & I_r \epx.
$$

\begin{Lemma} \label{lemma-Zstar}
	 Assume that $X^*$ and $Z^*$ given in 
	\eqref{eqn-Xstar-Zstar} is a pair of optimal solutions 
	in \eqref{p}$\mhyphen$\eqref{d}.
	
	Then $Z^*$ is a maximum rank  optimal solution of \eqref{d} $\Leftrightarrow$  there is a nonnegative integer $\ell, $ 
	$Y_1, \dots, Y_\ell \in \symn$ and a reformulation 
	\begin{equation} \label{p-prime-lemma2} 
		\begin{array}{rcl} 
			\A^\prime X & = & b^\prime \\  
			X & \succeq & 0
		\end{array} \tag{$P^\prime_{L2}$} 
	\end{equation}
	which satisfies \eqref{thm-main-2} and \eqref{thm-main-3} in Theorem \ref{thm-main}.
\end{Lemma} 
{\bf Proof} Let $t = n(n+1)/2 - m, \, $ and $D_1, \dots, D_t \in \sym{n}$ be  linearly independent, such that 
$$\la A_i, D_j \ra = 0 \,\, \text{for all} \,\, i \,\, \text{and} \,\, j. $$ 
Also let  $X_0 \in \symn$ be such that $\A X_0 = b.$ 

Consider the SDP 
\begin{equation}\label{re-d}
	\begin{array}{rl} 
		\inf  & \,\, \la X_0, Z   \ra  \\
		s.t. & \,\, \la D_j, Z \ra \, = \, \la D_j, C \ra  \, \text{for} \, j=1, \dots, t \\
		& \,\, Z \succeq 0, 
	\end{array}. \tag{\em  Re-$D$} 
\end{equation}
We claim that 
\bit
\item[] 
$Z$ is feasible/optimal/maximum rank  optimal  in \eqref{re-d} $\Leftrightarrow$ it has the same status in \eqref{d}.
\eit 
Indeed, a psd matrix $Z$ is feasible in \eqref{re-d} iff $\la Z - C, D_j \ra = 0$ for all $j, \,$
i.e., iff $Z - C$ is a linear combination of the $A_i.$ 
This implies the statement about feasible solutions.
Let us next fix $Z$ feasible in \eqref{d} and $y \in \rad{m}$ such that $Z = C - \A^* y.$ 
Then 
\begin{equation} \label{eqn-X0y} 
	\la X_0, Z \ra = \la X_0, C - \A^* y \ra = \la X_0, C  \ra - \la \A X_0, y \ra = \la X_0, C \ra  - \la y, b \ra,
\end{equation}
so $ \la X_0, Z \ra  + \la y, b \ra$ is constant for all feasible solutions of \eqref{re-d}. 
This argument implies the statement about optimal solutions. The statement about
 maximum rank  optimal solutions is immediate. 

In summary, \eqref{re-d} restates \eqref{d} in equality constrained, i.e.. primal form.

Thus we apply Lemma \ref{lemma-Xstar} with the roles of $X^*$ and $Z^*$ exchanged,
and deduce there is a reformulation of \eqref{re-d} in which the first, say $\ell$ equations are of the form 
$$
\la Y_i, Z \ra = 0 \,\, \text{for} \,\, i=1, \dots, \ell,
$$
and \eqref{thm-main-2-b} and \eqref{thm-main-3} hold. 

We just need to show that 
\eqref{thm-main-2-a} holds as well. For that, let $T$ be the product of all rotation matrices used in the 
reformulation process, and  
\begin{equation} \label{eqn-D_j_prime_define} 
	\begin{array}{rclr} 
		D_j^\prime & = & T^\top D_j  T  & \text{for  all} \, j \\
		A_i^\prime & = & T^{-1} A_i T^{- \top}  & \text{for  all} \,  i \\
		C^\prime & = & T^{-1} C  T^{- \top}.
	\end{array}
\end{equation}
Then 
\begin{equation} \label{eqn-Aiprime-Djprime} 
	\ba{rcl} 
	\la A_i^\prime, D_j^\prime \ra & = & 0 \; \text{for all} \; i, j \\ 
	\la C^\prime, D_j^\prime \ra & = & \la C, D_j \ra \; \text{for all} \; j. 
	\ena
\end{equation}
Let us fix $s \in \{1, \dots, \ell \}.$ 
Given that the equation $\la Y_s, Z \ra = 0$ was obtained  
by rotating \eqref{re-d} by $T$ and by elementary row operations, there is  $\lambda \in \rad{t}$ such that 
$$
\begin{array}{rcl}
	Y_s & = & \sum_{i=1}^t \lambda_i D_i^\prime \\
	0   & = &  \sum_{i=1}^t \lambda_i \la D_i, C \ra. 
\end{array}
$$
Combining this with \eqref{eqn-D_j_prime_define}  and \eqref{eqn-Aiprime-Djprime} 
we see that $Y_s$ has $0$ inner product with 
all $A_i^\prime$ and with $C^\prime, \, $ i.e., \eqref{thm-main-2-a} indeed holds. 
\qed 

Lemma \ref{lemma-Zstar} actually implies that \eqref{p} needs to be only rotated, elementary row operations are not needed in order to arrive at its conclusion.


\subsection{Proof of Theorem \ref{thm-main}} 
\label{subsection-proof-of-theorem-1-proof}  
We can now prove Theorem \ref{thm-main}. For that we need some more  technical preliminaries (which are not used in the rest of the paper).

\begin{Definition} \label{definition-R-block-ij-column} 
		\begin{enumerate} 
				\item Suppose $R \subseteq [n].$ We say that an $n \times n$ matrix is an $R$-block matrix if it arises from 
		$I$ by replacing $I(R)$ by some invertible matrix. 
		\item We say that an $n \times n$ matrix is an $i \mhyphen j$ column addition matrix, 
		if it arises from 
		$I$ by replacing the $(i,j)$ element by some real number. 
	\end{enumerate}
\end{Definition}
For example, if $n=4, \, R = \{1,4 \}, \, $ then 
$$
T \, = \, 
\begin{pmatrix} 2 & 0 & 0 & 3 \\
	0 & 1  & 0 & 0 \\
	0 & 0 & 1 & 0 \\
	4 & 0 & 0 & 5 \end{pmatrix}. 
$$
is an $R$-block matrix. 

The following statements follow from Definition 
\ref{definition-R-block-ij-column} and  are listed for reference:

\begin{Proposition} The following hold:
	\begin{enumerate}
		\item Suppose $T \in \mymatrix{n \times n}$ is an $R$-block matrix. 
		Then 
		\begin{enumerate}
			\item $T^\top$ and $T^{-1}$ are also $R$-block matrices. 
			\item If $M \in \mymatrix{n \times n}, $ then right multiplying $M$ by $T$ replaces $M([n], R)$ by 
			$M([n], R) T(R).$ 
		\end{enumerate} 
		\item If $T$ is an $i \mhyphen j$ column addition matrix, then 
		\begin{enumerate}
			\item  $T^{-1}$ is also  an $i \mhyphen j$ column addition matrix.
			\item $T^{\top}$ and  $T^{- \top}$ are $j \mhyphen i$ column addition matrices.
			\item right multiplying $M\in\mathbb{R}^{n\times n}$ by $T$ adds a multiple of column $i$ of $M$ to 
			column $j.$ 
		\end{enumerate}
	\end{enumerate}
\end{Proposition}
Facial reduction sequences appear in the theory of facial reduction 
\cite{BorWolk:81, Pataki:13}.

\begin{Definition}
	We say that a sequence of symmetric matrices $Y_1, \dots, Y_t$ is a {\em facial reduction sequence for} $\psdn$ if 
	$$
	Y_1 \in \psdn, \quad \text{and} \quad Y_{i+1} \in \bigl(   \psdn \cap Y_1^\perp \cap \dots \cap Y_i^\perp \bigr)^* \quad \text{for} ~~ i=1, \dots, t-1.
	$$ 
	Here, for a set $C \subseteq \symn$  we write  $C^* \, = \, \{ Y \, : \, \la X, Y \ra \geq 0 ~ \text{for  all} ~  X \in C \,   \}$ for its dual cone. 
\end{Definition}
From this definition it follows that 
if $(Y_1, \dots, Y_t)$ is a regular facial reduction sequence, then 
it is a facial reduction sequence. The converse is not true in general. However, we 
have the following proposition, which is excerpted from from Lemma 1 in \cite{liu2017exact}.

\begin{Proposition} \label{prop-fr} Suppose  $(Y_1, \dots, Y_t)$ is a facial reduction sequence. Then 
	\begin{enumerate}
		\item \label{prop-fr-1} If 	$T$ is an invertible matrix, then 
		$(T^{\top} Y_1 T, \dots, T^{\top} Y_{t} T)$ is also a  facial reduction sequence. 
		\item \label{prop-fr-2} There is a $T$ is an invertible matrix, such that 
		$(T^{\top} Y_1 T, \dots, T^{\top} Y_{t} T)$ is a regular facial reduction sequence. 
	\end{enumerate}
	\qed 
\end{Proposition}

\nin{Proof of Theorem \ref{thm-main}:}
Let $(P^\prime)$ be the reformulation produced from $(P)$ in Lemma \ref{lemma-Xstar}.
Then $(P^\prime)$ satisfies items \eqref{thm-main-1} and \eqref{thm-main-3} in Theorem \ref{thm-main}.
Also, by some mild abuse of notation, we let 
$(P^{\prime \prime})$ be the reformulation produced from $(P)$ in Lemma \ref{lemma-Zstar}, and 
$Y_1, \dots, Y_\ell$ the $Y_j$ matrices obtained therein. Then 
$(P^{\prime \prime})$ with $Y_1, \dots, Y_\ell$ satisfy items \eqref{thm-main-2} and \eqref{thm-main-3} in 
Theorem \ref{thm-main}.  

For convenience, define $A_0^\prime := Z^*, \, Y_0 := X^*.$ 
We see that 
\co{\begin{equation} \label{eqn-Z*-in-psdn-cap-Y0}
	Z^* \; \text{is maximum rank  in} \; \psd{n} \cap (X^*)^\perp \cap Y_1^\perp \cap \dots \cap Y_\ell^\perp.
\end{equation} }
\begin{equation} \label{eqn-Z*-in-psdn-cap-Y0}
	Z^* \; \text{is maximum rank  in} \; \psd{n}  \cap Y_0^\perp \cap \dots \cap Y_\ell^\perp.
\end{equation}
We next reformulate $(P^{\prime \prime})$ into $(P^\prime).$ In this process we rotate 
$(P^{\prime \prime})$ by $T, \,$ where $T$ is admissible; this can be done, since both 
$(P^{\prime })$  and $(P^{\prime \prime})$ were obtained from \eqref{p} using only admissible 
rotations. 
We also rotate all $Y_j$ by $T^{- \top}$ to make sure that 
\eqref{thm-main-2-a} in Theorem \ref{thm-main} remains true.  
Afterwards 
$(P^\prime)$ still satisfies items \eqref{thm-main-1-a}, 
\eqref{thm-main-1-b}, \eqref{thm-main-2-a} and \eqref{thm-main-3} 
Since these rotations do not change $X^* = Y_0$ and $Z^* = A_0^\prime, $ so \eqref{eqn-Z*-in-psdn-cap-Y0} remains true.
However, 
instead of \eqref{thm-main-2-b}, by part \eqref{prop-fr-1} of 
Proposition \ref{prop-fr}  we have the  weaker statement that 
$(X^*, Y_1, \dots, Y_\ell)$ is a facial reduction sequence.

The goal is to ensure that $(P^{\prime })$ satisfies all conditions of Theorem \ref{thm-main}. 
Let $j \in \{0, \dots, \ell \}, \,$ 
and consider the following 
condition:
\begin{equation} \label{eqn-invariant-j} 
	(Y_0, \dots, Y_{j}) \text{ is  a regular facial reduction sequence.}
\end{equation}
This statement holds when $j=0.$ We will rotate $Y_0, \dots, Y_\ell$ by suitable 
admissible rotations.
Further, anytime we rotate all $Y_j$ by an admissible $T, \,$ we also rotate 
all $A_0^\prime, \dots, A_k^\prime$ by $T^{- \top}$ to make sure that 
\eqref{thm-main-2-a} in Theorem \ref{thm-main} remains true. Overall, these rotations
will maintain \eqref{thm-main-1-a}, \eqref{thm-main-2-a}, \eqref{thm-main-3}, and 
\eqref{eqn-Z*-in-psdn-cap-Y0}. 

The goal is to make sure that 
\begin{align} 
	\label{to-ensure-1} 
	\text{After the rotations}  \, \eqref{eqn-invariant-j} & \, \text{holds with } \, j+1 \, \text{in place of } j \\
	\label{to-ensure-2}  
	\text{After the rotations} \,  \eqref{thm-main-1-b} & \, \text{remains true}.
\end{align} 
Since \eqref{eqn-Z*-in-psdn-cap-Y0} holds throughout, 
when \eqref{eqn-invariant-j} holds with $j = \ell$ then 
\eqref{thm-main-2-b} will hold, and we will be done.

So our task is to perform admissible rotations to ensure \eqref{to-ensure-1} and \eqref{to-ensure-2}. 
To better explain our reasoning, we argue 
why we don't just rotate all $Y_j$ by the $T$ furnished by 
part \ref{prop-fr-2} of Proposition \ref{prop-fr}? The reason is that if we 
arbitrarily chose such a $T, \,$ and rotated all $A_i^\prime$ by $T^{- \top}, \, $ 
then \eqref{thm-main-1-b} may not hold afterwards. So we need to choose our rotations with more care.

Let $(P_0, \dots, P_k)$ be the structure of $(A_0^\prime, \dots, A_k^\prime).$ For convenience, define 
$$
P_{k+1} := [n] \setminus P_{0:k}.
$$
Also let $(Q_0, \dots, Q_j)$ be the structure of $(Y_0, \dots, Y_j).$ 

We see that 
$$Y_{j+1} \in \left( \psd{n} \cap Y_0^\perp \cap \dots \cap Y_j^\perp \right)^*  \, \Rightarrow \, Y_{j+1}([n] \setminus Q_{0:j}) \succeq 0,   $$ 
and the other elements of $Y_{j+1}$ are arbitrary. 

Since $Z^*$ has zero inner product with all $Y_j$ (by \eqref{eqn-Z*-in-psdn-cap-Y0})
we have  
\begin{equation} \label{eqn-P0-cap-Q0j} 
	P_0 \cap  Q_{0:j} = \emptyset.  
\end{equation}
Set 
$$
R := [n] \setminus Q_{0:j} \, \text{and} \, R_i := R \cap P_i \,\, \text{for} \,\, i=0, \dots, k+1.
$$
The rotation we produce will make $Y_{j+1}(R)$ diagonal.
Since by \eqref{eqn-P0-cap-Q0j}  we have 
$R_0 = P_0, $ and  $\la Y_{j+1}, Z^* \ra = 0, \,$ 
we deduce $Y_{j+1}(R_0, R)=0.$ So $Y_{j+1}$ looks like on the first figure on Figure 
\ref{fig:Yj+1}.

We next perform the following transformations for $i=1,2, \dots, k+1:$ 

\begin{enumerate}[label=(Tr-\arabic*), ref=(Tr-\arabic*),
	leftmargin=4em,    
	labelsep=0.8em,    
	align=right]       
	\item \label{tr-1} Choose a $T$ $R_i$-block matrix such that i) the $R_i$ block of this matrix is orthonormal, and ii) the $(R_i, R_i)$ block of $T^\top Y_{j+1}T$ is diagonal. Rotate all $Y_0, \dots, Y_{\ell}$ by $T$ and all $A_i^\prime$ and $C^\prime$ by $T^{- \top}.$ 
	
	\item \label{tr-2} For $t=i+1, \dots, k$ add columns of $Y_{j+1}([n], R_i)$ to columns of $Y_{j+1}([n], R_t)$ to  ensure 
	$$
	Y_{j+1}(R_i, R_t) = 0. 
	$$
	Perform the analogous operations on the rows. 
	
	Supposing these column and row operations are represented by rotating by 
	a matrix $T, $ then also rotate $Y_0, \dots, Y_j, Y_{j+2}, \dots Y_\ell$ by the same $T.$  
	Also rotate all $A_i^\prime$ and $C^\prime$ by $T^{- \top}.$ 
\end{enumerate}

We illustrate this process on Figure \ref{fig:Yj+1} when $k=1.$ Here the $\oplus$ symbols  correspond to a positive semidefinite block,
the $+$ symbol to a positive definite block, and the $\times$ symbol to a block with arbitrary elements. The first arrow corresponds
to transformation \ref{tr-1} with $i=1.$ The second arrow corresponds to transformation 
\ref{tr-2}.

	\begin{figure}[h]
	\centering
	
	\begin{tikzpicture}[
		scale=0.580,
		transform shape,
		every node/.style={font=\normalsize},
		grid/.style={line width=0.6pt},
		mymark/.style={font=\Large},
		mybrace/.style={
			decorate,
			decoration={brace,amplitude=5pt}
		},
		myarrow/.style={
			-{Latex[length=1.8mm,width=1.2mm]},
			line width=0.35pt
		}
		]
		
		\node at (-3.5,2.5) {$Y_{j+1}=$};
		
		
		\begin{scope}[shift={(0,0)}]
			
			\draw[grid] (0,0) rectangle (5,5);
			
			\draw[grid] (1.2,0) -- (1.2,5);
			\draw[grid] (2.6,0) -- (2.6,5);
			\draw[grid] (3.6,2.4) -- (3.6,5);
			
			\draw[grid] (0,3.9) -- (5,3.9);
			\draw[grid] (0,2.4) -- (5,2.4);
			\draw[grid] (0,1.2) -- (2.6,1.2);
			
			\node[mymark] at (0.6,4.45) {$\times$};
			\node[mymark] at (1.9,4.45) {$\times$};
			\node[mymark] at (3.1,4.45) {$\times$};
			\node[mymark] at (4.3,4.45) {$\times$};
			
			\node[mymark] at (0.6,3.15) {$\times$};
			\node[mymark] at (0.6,1.8) {$\times$};
			\node[mymark] at (0.6,0.6) {$\times$};
			
			\node[mymark] at (1.9,3.15) {$0$};
			\node[mymark] at (3.1,3.15) {$0$};
			\node[mymark] at (4.3,3.15) {$0$};
			
			\node[mymark] at (1.9,1.8) {$0$};
			\node[mymark] at (1.9,0.6) {$0$};
			
			\node[
			draw,
			circle,
			inner sep=2pt,
			line width=0.6pt,
			font=\Large
			] at (3.6,1.2) {$+$};
			
			\draw[mybrace] (0,5.25) -- (1.2,5.25)
			node[midway,above=11pt] {$Q_{0:j}$};
			
			\draw[mybrace] (1.2,5.25) -- (2.6,5.25)
			node[midway,above=11pt] {$R_0$};
			
			\draw[mybrace] (2.6,5.25) -- (3.6,5.25)
			node[midway,above=11pt] {$R_1$};
			
			\draw[mybrace] (3.6,5.25) -- (5,5.25)
			node[midway,above=11pt] {$R_2$};
			
			\draw[
			decorate,
			decoration={brace,mirror,amplitude=5pt}
			]
			(-0.2,5) -- (-0.2,3.9)
			node[midway,left=11pt] {$Q_{0:j}$};
			
			\draw[
			decorate,
			decoration={brace,mirror,amplitude=5pt}
			]
			(-0.2,3.9) -- (-0.2,2.4)
			node[midway,left=11pt] {$R_0$};
			
			\draw[
			decorate,
			decoration={brace,mirror,amplitude=5pt}
			]
			(-0.2,2.4) -- (-0.2,1.2)
			node[midway,left=11pt] {$R_1$};
			
			\draw[
			decorate,
			decoration={brace,mirror,amplitude=5pt}
			]
			(-0.2,1.2) -- (-0.2,0)
			node[midway,left=11pt] {$R_2$};
			
		\end{scope}
		
		\draw[myarrow] (6.0,2.5) -- (7.3,2.5);
		
		
		\begin{scope}[shift={(8.0,0)}]
			
			\draw[grid] (0,0) rectangle (5,5);
			
			\draw[grid] (1.2,0) -- (1.2,5);
			\draw[grid] (2.6,0) -- (2.6,5);
			\draw[grid] (3.6,0) -- (3.6,5);
			
			\draw[grid] (0,3.9) -- (5,3.9);
			\draw[grid] (0,2.4) -- (5,2.4);
			\draw[grid] (0,1.2) -- (5,1.2);
			
			\draw[grid] (3.1,1.2) -- (3.1,2.4);
			\draw[grid] (2.6,1.8) -- (3.6,1.8);
			
			\node[mymark] at (0.6,4.45) {$\times$};
			\node[mymark] at (1.9,4.45) {$\times$};
			\node[mymark] at (3.1,4.45) {$\times$};
			\node[mymark] at (4.3,4.45) {$\times$};
			
			\node[mymark] at (0.6,3.15) {$\times$};
			\node[mymark] at (0.6,1.8) {$\times$};
			\node[mymark] at (0.6,0.6) {$\times$};
			
			\node[mymark] at (1.9,3.15) {$0$};
			\node[mymark] at (3.1,3.15) {$0$};
			\node[mymark] at (4.3,3.15) {$0$};
			
			\node[mymark] at (1.9,1.8) {$0$};
			\node[mymark] at (1.9,0.6) {$0$};
			
			\node[mymark] at (2.85,2.05) {$+$};
			\node[mymark] at (3.35,2.05) {$0$};
			
			\node[mymark] at (2.85,1.55) {$0$};
			\node[mymark] at (3.35,1.55) {$0$};
			
			\node[mymark] at (4.3,2.05) {$\times$};
			\node[mymark] at (4.3,1.55) {$0$};
			
			\node[mymark] at (3.1,0.6) {$\times$};
			\node[mymark] at (4.3,0.6) {$\oplus$};
			
		\end{scope}
		
		\draw[myarrow] (14.0,2.5) -- (15.3,2.5);
		
		
		\begin{scope}[shift={(16.0,0)}]
			
			\draw[grid] (0,0) rectangle (5,5);
			
			\draw[grid] (1.2,0) -- (1.2,5);
			\draw[grid] (2.6,0) -- (2.6,5);
			\draw[grid] (3.6,0) -- (3.6,5);
			
			\draw[grid] (0,3.9) -- (5,3.9);
			\draw[grid] (0,2.4) -- (5,2.4);
			\draw[grid] (0,1.2) -- (5,1.2);
			
			\draw[grid] (3.1,1.2) -- (3.1,2.4);
			\draw[grid] (2.6,1.8) -- (3.6,1.8);
			\draw[grid] (3.6,1.8) -- (5,1.8);
			\draw[grid] (3.1,0) -- (3.1,1.2);
			
			\node[mymark] at (0.6,4.45) {$\times$};
			\node[mymark] at (1.9,4.45) {$\times$};
			\node[mymark] at (3.1,4.45) {$\times$};
			\node[mymark] at (4.3,4.45) {$\times$};
			
			\node[mymark] at (0.6,3.15) {$\times$};
			\node[mymark] at (0.6,1.8) {$\times$};
			\node[mymark] at (0.6,0.6) {$\times$};
			
			\node[mymark] at (1.9,3.15) {$0$};
			\node[mymark] at (3.1,3.15) {$0$};
			\node[mymark] at (4.3,3.15) {$0$};
			
			\node[mymark] at (1.9,1.8) {$0$};
			\node[mymark] at (1.9,0.6) {$0$};
			
			\node[mymark] at (2.85,2.05) {$+$};
			\node[mymark] at (3.35,2.05) {$0$};
			
			\node[mymark] at (2.85,1.55) {$0$};
			\node[mymark] at (3.35,1.55) {$0$};
			
			\node[mymark] at (4.3,2.05) {$0$};
			\node[mymark] at (4.3,1.55) {$0$};
			
			\node[mymark] at (2.85,0.6) {$0$};
			\node[mymark] at (3.35,0.6) {$0$};
			
			\node[mymark] at (4.3,0.6) {$\oplus$};
			
		\end{scope}
		
	\end{tikzpicture}
	
	\caption{Transforming \(Y_{j+1}\).}
	\label{fig:Yj+1}
\end{figure}
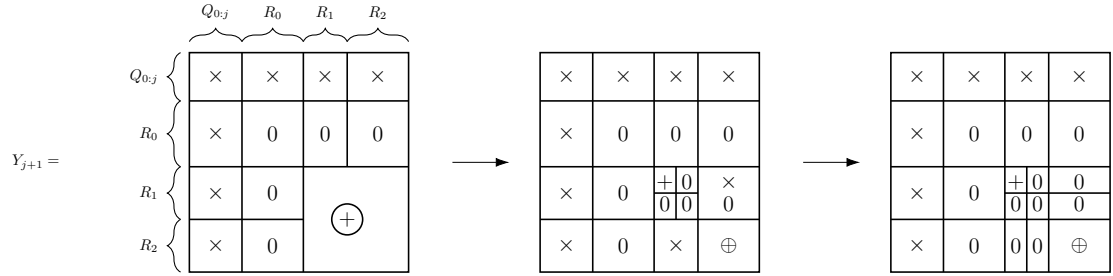

When we are done, we set 
\begin{equation} \label{eqn-defineQj+1} 
	Q_{j+1} := \{ v ~|~ \text{\rm the }(v,v)\text{ \rm element of }Y_{j+1}(R) \text{ \rm is positive} \}.  
\end{equation}
We next claim that all rotations given above are admissible.  
For that, we first consider  a rotation by a matrix $T$ in \ref{tr-1}. Recall that the $R_i$ diagonal block of $T$ is invertible, 
where $R_i \subseteq P_i.$ We have that $R_i$ is disjoint from $Q_0$ (since $R_i \subseteq R$) and from $P_0$ (since $R_i \subseteq P_i$ and $i > 0$). Thus $T$ is indeed 
admissible.

Next we consider a rotation by a matrix $T$ in \ref{tr-2}. 
 We assume without loss of generality that 
$T$ is a $u \mhyphen v$ column addition matrix, where $u \in R_i, \, v \in R_t, \, $ and $0 < i < t.$ (We do not lose generality, since $T$ is the product of such matrices.)
Since $T$ arises from the identity by replacing an element in the $(R_i, R_t)$ block, and 
$R_i$ and $R_t$ are disjoint from $P_0 \cup Q_0, \,$ we see that $T$ is an admissible matrix.

Next we prove \eqref{to-ensure-1}. For that, by the definition of $Q_{j+1}$ it suffices to show that 
after the operations $(Y_0, \dots , Y_j)$ \text{is still a regular facial reduction sequence.}
\co{\begin{equation} \label{eqn-Y0-Yj-stays-regfr} 
		\text{After these operations,} \, (Y_0, \dots , Y_j) \, \text{is still a regular facial reduction sequence.} 
\end{equation}}
To prove this, first let $T$ be a rotation defined in \ref{tr-1}. We claim that rotating  
$(Y_0, \dots , Y_j)$ by $T$ keeps it a regular facial reduction sequence.  
For that,  we assume without loss of generality that rows and 
columns in $Y_0, \dots, Y_j$ were permuted so the indices in $Q_0, \dots, Q_j$ come 
first. Let   $s \in \{0, \dots, j \}.$  Since 
$R_i \subseteq [n] \setminus Q_{0:j}$ the changed portion of $Y_s$ 
is shown on Figure \ref{fig:Ys-rotate} as a striped block, so our claim follows.

Next let $T$ be a rotation defined in \ref{tr-2}. We claim  that rotating by $T, \,$ where $T$ is defined in \ref{tr-2}, keeps 
$(Y_0, \dots , Y_j)$ a regular facial reduction sequence. Again, 
it suffices to prove this when $T$ is a $u \mhyphen v$ column addition matrix,
where $u \in R_i, \, v \in R_t, \,$ and $ i < t.$ Again, fix   $s \in \{0, \dots, j \}.$  
Since $R_i, R_t \subseteq [n] \setminus Q_{0:j}, \, $ the changed portion of $Y_s$ is shown on Figure \ref{fig:Ys-rotate}, with $t$ in place of $i.$ 
So this claim follows as well, hence we are done with the proof of \eqref{to-ensure-1}.
\begin{figure}[h]
	\centering
	\[
\begin{tikzpicture}[
    x=0.75cm,
    y=0.75cm,
    baseline=(current bounding box.center),
    every node/.style={font=\normalsize}
]


\node at (-0.85,2.5) {$Y_s=$};

\fill[pattern=north east lines] (0,1) rectangle (1,2);
\fill[pattern=north east lines] (3,4) rectangle (4,5);

\foreach \x in {0,1,2,3,4,5}
{
    \draw (\x,0) -- (\x,5);
}

\foreach \y in {0,1,2,3,4,5}
{
    \draw (0,\y) -- (5,\y);
}

\node at (0.5,4.5) {$\times$};
\node at (1.5,4.5) {$\times$};
\node at (2.5,4.5) {$\times$};
\node at (3.5,4.5) {$\times$};
\node at (4.5,4.5) {$\times$};

\node at (0.5,3.5) {$\times$};
\node at (1.5,3.5) {$+$};

\node at (0.5,2.5) {$\times$};
\node at (0.5,1.5) {$\times$};
\node at (0.5,0.5) {$\times$};

\draw[
    decorate,
    decoration={brace,amplitude=4pt,raise=3pt}
]
(0,5.12) -- (1,5.12)
node[midway,yshift=17pt] {$Q_{0:s}$};

\draw[
    decorate,
    decoration={brace,amplitude=4pt,raise=3pt}
]
(3,5.12) -- (4,5.12)
node[midway,yshift=17pt] {$R_i$};

\end{tikzpicture} 
	\]
	\caption{The change in $Y_s$} 
		\label{fig:Ys-rotate}
\end{figure}
We next prove \eqref{to-ensure-2}. For that we first consider  a rotation by a matrix $T$ in \ref{tr-1}. Recall that we rotate by $T^{- \top}$ all 
$A_i^\prime$ and $C^\prime.$ Also recall that $R_i \subseteq P_i$ and 
the $R_i$ diagonal block of $T$ is orthonormal. Thus the $R_i$ diagonal block of $T^{- \top}$ is also orthonormal.
Since $R_i \subseteq P_i, \,$ and the structure of $A_0^\prime, \dots, A_k^\prime$ is $P_0, \dots,  P_k, \,$ our claim follows.

Next we consider a rotation by a matrix $T$ in \ref{tr-2}. Again we assume without loss of generality that 
$T$ is a $u \mhyphen v$ column addition matrix, where $u \in R_i, \, v \in R_t, \, $ and $i < t.$ 
Thus $T^{- \top}$ is a $v \mhyphen u$ column addition matrix. Since we rotate 
$A_0^\prime, \dots, A_k^\prime$ by $T^{- \top}$ the change in $A_s^\prime, \, $ where $s \in \{0, \dots, k \}$ is 
shown on Figure 
\ref{fig-A-s-change-rotation}.    Here we only show the change that takes place when we replace $A_s^\prime$ by $A_s^\prime T^{- \top}.$ 
We distinguish the cases $s=i, \, s < i, \,$ and $s > i.$

   \begin{figure}[h]
   	\centering
   	
   	\begin{tikzpicture}[
   		scale=0.50,
   		transform shape,
   		every node/.style={font=\normalsize},
   		grid/.style={line width=0.5pt},
   		mymark/.style={font=\Large},
   		myarrow/.style={-{Latex[length=2.5mm]}, line width=0.7pt}
   		]
   		
   		\begin{scope}[shift={(0,0)}]
   			
   			\node at (-1.5,2.5) {$A_s^\prime=$};
   			
   			\draw[grid] (0,0) rectangle (5,5);
   			\foreach \x in {1,2,3,4} \draw[grid] (\x,0) -- (\x,5);
   			\foreach \y in {1,2,3,4} \draw[grid] (0,\y) -- (5,\y);
   			
   			\foreach \x in {0.5,1.5,2.5,3.5,4.5}
   			\node[mymark] at (\x,4.5) {$\times$};
   			\foreach \y in {0.5,1.5,2.5,3.5}
   			\node[mymark] at (0.5,\y) {$\times$};
   			
   			\node[mymark] at (1.5,3.5) {$+$};
   			
   			\draw[decorate,decoration={brace,amplitude=4pt}]
   			(1,5.25) -- (2,5.25) node[midway,above=5pt] {$P_i$};
   			
   			\draw[decorate,decoration={brace,amplitude=4pt}]
   			(3,5.25) -- (4,5.25) node[midway,above=5pt] {$P_t$};
   			
   			\draw[myarrow] (3.3,5.8) to[out=120,in=60] (1.7,5.8);
   			
   			\node at (2.5,-0.9) {$s<i$};
   			
   		\end{scope}
   		
   		\begin{scope}[shift={(8,0)}]  
   			
   			\node at (-1.5,2.5) {$A_s^\prime=$};
   			
   			\draw[grid] (0,0) rectangle (5,5);
   			\foreach \x in {1,2,3,4} \draw[grid] (\x,0) -- (\x,5);
   			\foreach \y in {1,2,3,4} \draw[grid] (0,\y) -- (5,\y);
   			
   			\foreach \x in {0.5,1.5,2.5,3.5,4.5}
   			\node[mymark] at (\x,4.5) {$\times$};
   			\foreach \y in {0.5,1.5,2.5,3.5}
   			\node[mymark] at (0.5,\y) {$\times$};
   			
   			\node[mymark] at (1.5,3.5) {$+$};
   			
   			\draw[decorate,decoration={brace,amplitude=4pt}]
   			(1,5.25) -- (2,5.25) node[midway,above=5pt] {$P_s$};
   			
   			\draw[decorate,decoration={brace,amplitude=4pt}]
   			(3,5.25) -- (4,5.25) node[midway,above=5pt] {$P_i$};
   			
   			\draw[decorate,decoration={brace,amplitude=4pt}]
   			(4,5.25) -- (5,5.25) node[midway,above=5pt] {$P_t$};
   			
   			\draw[myarrow] (4.65,5.8) to[out=120,in=60] (3.35,5.8);
   			
   			\node at (2.5,-0.9) {$s=i$};
   			
   		\end{scope}
   		
   		\begin{scope}[shift={(16,0)}]
   			
   			\node at (-1.5,2.5) {$A_s^\prime=$};
   			
   			\draw[grid] (0,0) rectangle (5,5);
   			\foreach \x in {1,2,3,4} \draw[grid] (\x,0) -- (\x,5);
   			\foreach \y in {1,2,3,4} \draw[grid] (0,\y) -- (5,\y);
   			
   			\foreach \x in {0.5,1.5,2.5,3.5,4.5}
   			\node[mymark] at (\x,4.5) {$\times$};
   			
   			\foreach \x in {0.5,1.5,2.5,3.5,4.5}
   			\node[mymark] at (\x,3.5) {$\times$};
   			
   			\foreach \y in {0.5,1.5,2.5}
   			{
   				\node[mymark] at (0.5,\y) {$\times$};
   				\node[mymark] at (1.5,\y) {$\times$};
   			}
   			
   			\node[mymark] at (2.5,2.5) {$+$};
   			
   			\draw[decorate,decoration={brace,amplitude=4pt}]
   			(1,5.25) -- (2,5.25) node[midway,above=5pt] {$P_i$};
   			
   			\draw[decorate,decoration={brace,amplitude=4pt}]
   			(2,5.25) -- (3,5.25) node[midway,above=5pt] {$P_s$};
   			
   			\draw[decorate,decoration={brace,amplitude=4pt}]
   			(4,5.25) -- (5,5.25) node[midway,above=5pt] {$P_t$};
   			
   			\draw[myarrow] (4.65,5.8) to[out=120,in=60] (1.35,5.8);
   			
   			\node at (2.5,-0.9) {$s>i$};
   			
   		\end{scope}
   		
   	\end{tikzpicture}
   	
   	\caption{The change in $A_s^\prime$ when we replace it by $A_s^\prime T^{- \top}.$ }
   	\label{fig-A-s-change-rotation}

   \end{figure}

\qed

\co{
For your paper, I would think in terms of coverage, not count. You probably need to cite the key clusters:

strict complementarity and IPMs;
facial reduction / singularity degree;
error bounds and Hoffman-type bounds;
generators or examples of pathological/non-strictly complementary SDPs;
numerical SDP solvers and DIMACS-type residuals;
your own related structural/reformulation papers.

If that naturally gives 50?80 references, fine. If it gives 40 very relevant references, also fine. But if FoCM papers in your area often have long introductions, then yes, it is probably wise to make the related-work section more complete than for a shorter optimization note.

A possible mindset:

Do not pad the bibliography, but do not look under-informed.
?

So I would not add random marginal papers. I would add papers only where they help support a sentence like ?this issue appears in IPM convergence,? ?this relates to singularity degree,? ?this is an error-bound phenomenon,? or ?this has been observed computationally.?

}

\section{Conclusion}

Our contributions are of interest from two viewpoints. From a purely mathematical
perspective, the normal form in Theorem \ref{thm-main} parametrizes all SDPs that
lack strict complementarity. Such normal forms have a long history in linear
algebra; perhaps the best-known example is the Jordan normal form. Our normal form also provides the  simple-to-implement  Algorithms
\ref{algo-nonstrict-SDP} and \ref{algo-nonstrict-Slater-SDP}, which capture all
SDPs that fail strict complementarity. 

A key ingredient in our  algorithms is
Theorem \ref{thm-singularity-degree}, which precisely characterizes the singularity degree of SDPs in a special case. 
Given the great  interest
in the singularity degree, and the relative scarcity of results
about it, we believe that this theorem is of independent interest.

From the computational viewpoint, the SDPs generated by our algorithms have a
striking feature. Apart from the failure of strict complementarity, their
numerical parameters are quite benign: the condition numbers are small and the
coefficient spread is moderate. Thus, in these examples, the numerical difficulty
can be pinned on the failure of strict complementarity. Indeed, SDPT3 solved most of these
instances with small DIMACS errors; however, the corresponding forward errors
were up to seven orders of magnitude larger.

\co{
From the computational viewpoint, the SDPs generated by our algorithms have a striking feature.
Apart from the failure of strict complementarity, their numerical parameters are
quite benign: the condition numbers are small and the coefficient spread is
moderate. Indeed, SDPT3 solved most of these instances  with small
DIMACS errors. However, the corresponding forward errors were up to seven
orders of magnitude larger. Thus, in these examples, the numerical difficulty can be pinned on 
the failure of strict complementarity.
}
\co{ rather than to poor conditioning
or large coefficient spread.
}

One might object that strict complementarity holds generically for random SDPs;
see \cite{alizadeh1997complementarity}. However, SDPs arising from applications
are typically not random, and the analogy with Slater's condition is instructive.
Although random feasible SDPs satisfy Slater's condition generically
\cite{dur2017genericity}, SDPs arising in practice may fail it. Our results
suggest that, in applications, it is not enough to verify Slater's condition
alone; whenever possible, one should also verify mathematically whether strict
complementarity holds.

\co{
One might object that strict complementarity holds generically for random SDPs;
see \cite{alizadeh1997complementarity}. However, SDPs arising from applications
are typically not random, hence the situation is analogous to the situation with Slater's condition.
Although random feasible SDPs satisfy Slater's condition generically
\cite{dur2017genericity}, SDPs arising in practice may not. Our results
suggest that, in applications, it is not enough to verify Slater's condition
alone; whenever possible, one should also verify mathematically whether strict
complementarity holds.
}

\paragraph{Acknowledgements} I am very grateful to Erling Andersen, Ed Klotz, and especially to Michael Overton for  answering my questions on 
numerical stability and the relevance of spread.

\bibliographystyle{plain}

\bibliography{mysdp}
\end{document}